\documentclass[twoside, reqno]{amsart}
\usepackage{amsmath,amssymb,amscd,amsthm,amsxtra,amsfonts}
\usepackage{mathtools,mathrsfs,dsfont,xparse}
\usepackage{graphicx,epstopdf}
\usepackage{multirow}
\usepackage{cases}
\usepackage{enumerate}
\usepackage{xcolor}
\usepackage{tikz,pgfplots}
\usetikzlibrary{matrix}
\usepackage{mathtools}
\usepackage[margin=0.9 in]{geometry}
\usepackage{todonotes}
\usepackage{listings}
\usepackage{tcolorbox}

\mathtoolsset{showonlyrefs}

\newtheorem{innercustomgeneric}{\customgenericname}
\theoremstyle{plain}
\providecommand{\customgenericname}{}
\newcommand{\newcustomtheorem}[2]{%
  \newenvironment{#1}[1]
  {%
   \renewcommand\customgenericname{#2}%
   \renewcommand\theinnercustomgeneric{##1}%
   \innercustomgeneric
  }
  {\endinnercustomgeneric}
}

\newcustomtheorem{customthm}{Theorem}

\allowdisplaybreaks[4]

\makeatletter
\newcommand{\dotDelta}{{\vphantom{\Delta}\mathpalette\d@tD@lta\relax}}
\newcommand{\d@tD@lta}[2]{%
  \ooalign{\hidewidth$\m@th#1\mkern-1mu\cdot$\hidewidth\cr$\m@th#1\Delta$\cr}%
}
\makeatother

\date{}

\usepackage{multicol,bbm}
\usepackage[colorlinks=true, pdfstartview=FitV, linkcolor=blue, citecolor=blue, urlcolor=blue]{hyperref}

\numberwithin{equation}{section}

\usepackage{cjhebrew}

\DeclareMathOperator{\re}{Re}
\DeclareMathOperator{\im}{Im}

\usepackage[numbers,sort]{natbib}

{\theoremstyle {definition} \newtheorem {defn} {Definition} [section] }
{\theoremstyle {plain}  \newtheorem {thm} [defn] {Theorem}}
{\theoremstyle {plain}  }
{\theoremstyle {plain} }
{\theoremstyle {plain} \newtheorem {lem}[defn] {Lemma}}
{\theoremstyle {definition} \newtheorem {rmk}[defn] {Remark}}
{\theoremstyle {plain} \newtheorem {claim}[defn] {Claim}}

\def\R{{\mathbb{R}}}
\def\C{{\mathbb{C}}}
\def\N{{\mathbb{N}}}
\def\Z{{\mathbb{Z}}}
\def\D{{\bf \dotDelta}}
\def\e{{\varepsilon}}

\newcommand{\abs}[1]{\lvert#1\rvert}
\newcommand{\norm}[1]{\left\|#1\right\|}
\newcommand{\inner}[1]{\left \langle#1\right \rangle}

\newcommand{\parenthese}[1]{\left(#1\right)}

\begin{document}

\author[Lee and Yu]{Zachary Lee and Xueying Yu}

\address{Zachary Lee 
\newline \indent Department of Mathematics, The University of Texas at Austin \indent 
\newline \indent 
2515 Speedway, PMA 8.100, Austin, TX 78712
\newline \indent 
And
\newline \indent 
Department of Mathematics, MIT \indent 
\newline \indent 77 Massachusetts Ave, Cambridge, MA 02139\indent 
}
\email{zl9868@my.utexas.edu}

\address{Xueying  Yu
\newline \indent Department of Mathematics, Oregon State University\indent 
\newline \indent  Kidder Hall 368
Corvallis, OR 97331 \indent 
\newline \indent 
And 
\newline \indent 
Department of Mathematics, University of Washington, 
\newline \indent  C138 Padelford Hall Box 354350, Seattle, WA 98195}
\email{xueying.yu@oregonstate.edu}

\title[Uniqueness continuation of fourth-order Schr\"odinger equations]{On uniqueness properties of solutions of the generalized fourth-order Schr\"odinger equations}



\begin{abstract}
In this paper, we study uniqueness properties of solutions to the  generalized fourth-order Schr\"odinger equations in any dimension $d$ of the following forms,
$$i \partial_t u +  \sum_{j=1}^d \partial_{x_j}^{\, 4}  u = V(t, x) u,  \quad \text{and} \quad i \partial_t u +  \sum_{j=1}^d \partial_{x_j}^{\, 4}  u + F (u, \overline{u}) = 0.$$
We show that a linear solution $u$ with fast enough decay in certain Sobolev spaces at two different times has to be trivial. Consequently, if the difference between two nonlinear solutions $u_1$ and $u_2$ decays sufficiently fast at two different times, it implies that $u_1 \equiv u_2$.

\end{abstract}

\maketitle

\setcounter{tocdepth}{1}
\tableofcontents

\parindent = 10pt     
\parskip = 5pt

\section{Introduction}
In this work, we consider both the linear  generalized fourth-order Schr\"odinger equations of the following form
\begin{align}\label{eq 4SE}
\begin{cases}
i \partial_t u +  \D^2 u = V(t, x) u , &  (t, x) \in \R \times \R^d ,\\
u(0,x) = u_0(x) ,
\end{cases}
\end{align}
and nonlinear ones of the type
\begin{align}\label{eq 4SE nonlinear}
\begin{cases}
i \partial_t u +  \D^2 u +F(u, \overline{u})=0 , &  (t, x) \in \R \times \R^d ,\\
u(0,x) = u_0(x) ,
\end{cases}
\end{align}
where $\D^2 := \sum_{j=1}^d \partial_{x_j}^{\, 4}$. Note that $\D^2$ is a fourth-order differential operator that removes the mixed derivative terms $\partial_{x_k x_k x_j x_j}$ ($k \neq j$) from the regular bi-Laplacian operator $\Delta^2$. 
We will use a slight abuse of language -- we will refer to $\D^2$ as a `separable' fourth-order Schr\"odinger operator in the rest of this work.

We establish two main results. First, we give sufficient conditions on the decay of the solution $u$ at two different times $t = 0$ and $t = 1$ which guarantee that $u \equiv 0$  is the unique solution of \eqref{eq 4SE}. Second, we deduce sufficient conditions on the decay of the difference of two solutions of \eqref{eq 4SE nonlinear} at two times $t=0$ and $t=1$ so that the two solutions are in fact equal. 
It is worth noting that in order to deduce a nonlinear unique continuation result from a linear one, the potential $V$ in \eqref{eq 4SE} must (i) allow complex values, and (ii) be time-dependent (this is because when considering the difference between two nonlinear solutions ($u$ and $v$), the difference in their nonlinearities $F(u,\overline{u}) - F(v, \overline{v})$ becomes complex-valued and time-dependent).

\subsection{Motivation}
\subsubsection{Unique continuation questions}
The study of unique continuation for partial differential equations (PDE) has been historically an active area of research. 
This line of research addresses the question of under what conditions two solutions of a PDE must coincide. 
In the context of dispersive PDEs, there are lots of works along this research line, see \cite{Z_KdV, Z_NLS, IK1, IK2, IV_JDE} and reference therein. These  uniqueness results typically assume that two solutions coincide in a large subdomain of $\R^d$ at two different times, then they conclude that they are identical on $\R^d$.

In Kenig-Ponce-Vega \cite{KPV_CPAM} and Escauriaza-Kenig-Ponce-Vega \cite{EKPV_CPDE}, the authors were motivated by Hardy's uncertainty principle \cite{Hardy} and initiated a different way to answer the unique continuation question for free Schr\"odinger equations, that is, they consider the solutions of linear Schr\"odinger equation of the following form, 
\begin{align}\label{eq LS}
 i\partial_t u +  \Delta u =  V(t,x) u, \quad (t,x) \in [0,T] \times \R^d.
\end{align}
that do not agree on a certain subset but have comparable (fast spatial Gaussian) decays at certain times.

Roughly speaking, \cite{KPV_CPAM, EKPV_CPDE} showed that if a solution has {\it fast decay} at two different times, the solution has to be trivial. Accordingly, for the nonlinear equation, they deduced the uniqueness of the solution from information on the decay of the difference of two possible solutions at two different times. Later, in a series of papers \cite{EKPV_Duke, EKPV_JEMS, EKPV_MRL,  KPV_MRL}, the authors obtained the sharpest {\it fast decay} requirement. In particular,  \cite{EKPV_Duke} obtained that for  classical Schr\"odinger equations if the potential $V$ satisfies certain boundedness properties (we neglect the assumption on $V$ here, but $V$ is allowed to be complex-valued and time-dependent, hence nonlinear uniqueness results available), then for a solution $u$ with fast enough decay, one has that the solution must be trivial.

To best connect Hardy's uncertainty principle to this form of results, let us first recall the {\it fast decay} type result in \cite{EKPV_Duke}.
\begin{customthm}{A}[\normalfont Theorem 1 in \cite{EKPV_Duke}]\label{thm Sharp}
Assume that $u \in C([0,T], L^2 (\R^d))$ verifies \eqref{eq LS}. 
$A, B >0$,   $AB > 1 /16 $ both $\norm{ e^{A \abs{x}^2} u(0,x)}_{L^2 (\R^d)}$ and $\norm{e^{B \abs{x}^2} u(1,x)}_{L^2 (\R^d)}$ are finite, and the potential $V$ satisfies certain boundedness properties. Then $u \equiv 0$. Moreover, if $AB  = \frac{1}{16} $, $u$ is  a constant multiple of $e^{-B + i/4T \abs{x}^2}$. 
\end{customthm}
Recall also the  following Hardy's uncertainty principle.
\begin{customthm}{B}[\normalfont Hardy's uncertainty principle in \cite{Hardy}]\label{thm Hardy}
 For any function $f : \R \to \C$, if the function $f$ itself and its Fourier transform $\widehat{f}$ satisfy 
 \begin{align*}
 f (x) = \mathcal{O} (e^{- A x^2}) \quad   \text{and} \quad \widehat{f} (\xi) = \mathcal{O} (e^{- 4B \xi^2})
 \end{align*}
with $A,B >0$ and $A B >\frac{1}{16}$, then $f \equiv 0$.  Moreover, if $A =B =\frac{1}{4}$, then $f (x) =Ce^{- \frac{1}{4}x^2}$.
\end{customthm}

The potential $V$ here can be thought as a perturbation of the free Schr\"odinger equation, with which the uniqueness of solutions for nonlinear equations  will be obtained as a direct consequence of Theorem \ref{thm Sharp} (by considering the equation satisfied by the difference of two nonlinear solutions). To see the heart of the problem, the discussion below safely ignores the potential, $V$.

Now one can see how Theorem \ref{thm Hardy} motivates Theorem \ref{thm Sharp} by writing the free Schr\"odinger  flow (taking $V =0$) with initial data $f$ in the following form using Fourier transforms
\begin{align}\label{eq FS}
u(t,x) := e^{it \Delta} f = (4 \pi it)^{-\frac{d}{2}} \int_{\R^d} e^{\frac{i\abs{x-y}^2}{4t}} f(y) \, dy = (2 \pi it)^{-\frac{d}{2}} e^{\frac{i \abs{x}^2}{4t}} \widehat{e^{\frac{i\abs{\cdot}^2}{4t}} f} \parenthese{\frac{x}{2t}} . 
\end{align}
Roughly speaking, the solution of the free Schr\"odinger equation at time $t$ is a rescaled multiple of the Fourier transform of the initial data. One would connect the decay requirement at time $t=1$ in Theorem \ref{thm Sharp} to the decay requirement on the Fourier transform of $f$ in Theorem \ref{thm Hardy} by simply evaluating \eqref{eq FS} at, for example, time $t=1$.

Let us point out that the {\it fast enough decay} measured in such Gaussian weight fashion is sharp since in the same work \cite{EKPV_Duke}, the authors provided an example for the threshold  case ($T\alpha \beta = 1/4$) which is a nonzero smooth solution with corresponding decay.

It is then natural to ask whether one can give quantitative unique continuation properties from two times for more general dispersive equations with a similar flavor. The answer is {\it Yes}. In \cite{EKPV_JFA}, the $k$-generalized KdV equations
\begin{align}\label{eq KdV}
\partial_t u + \partial_x^3 u + u^k \partial_x u =0, \quad (t,x) \in \R \times \R,  \quad k \in \N_+, \tag{KdV}
\end{align}
were considered and the authors obtained that the difference of two distinct solutions $u_1 -u_2$  cannot decay faster than $e^{-\lambda x^{3/2}}$ at two different times for $\lambda > 0$. 
It is worth pointing out that the decay rate $e^{-c x^{3/2}}$ corresponds to the behavior of the fundamental solution for the linear problem, which appears as a scaled Airy function. Moreover, \cite{ILP} showed that such decay assumption is optimal by finding an example solution persisting  the spatial decay that  initial data enjoys in a long time. There are also works on higher order KdV type equation \cite{Daw, Isa}.

We would like to mention that unique continuation results have been established for various dispersive models including  Schr\"odinger with gradient terms \cite{DS}, discrete Schr\"odinger equations \cite{BV, JLMP}, variable coefficient Schr\"odinger equations \cite{FLY} and Schr\"odinger equations with magnetic potential in \cite{BFGR, CF, BCF}. For further details and additional relevant references, we refer the interested reader to the aforementioned works.

As for higher order Schr\"odinger equations, a recent work by Huang-Huang-Zheng \cite{HHZ} obtained a unique continuation result of such  fast decay type  in  one spatial dimension for the model of the following form
\begin{align*}
i \partial_t u - (-\Delta)^m u = V (x) u , \quad (t,x) \in \R \times \R, \quad m = 2,3,4, \dots .
\end{align*}
The result reads that a non-trivial cannot decay faster than $e^{-\lambda \abs{x}^{\frac{2m}{2m-1}}}$ at two different times. 

In the general case,  in order  to realize the correct decay rate, one recalls the following result by H\"ormander \cite{Hor}, which is a variation of Hardy's uncertainty principle with conjugate convex weights. 
\begin{customthm}{C}[\normalfont Corollary in \cite{Hor}]\label{thm Hor}
If $\varphi$ and $\psi$ are conjugate convex functions, for example, $\varphi = \abs{x}^p /p$ and  $\psi = \abs{x}^q /q$, with $\frac{1}{p}+ \frac{1}{q} =1$. 
Then $f \equiv 0$ if
\begin{align*}
\int_{\R} \abs{f (x)} e^{\varphi (x)} \, dx <\infty , \quad \text{and} \quad \int_{\R} |\widehat{f}(\xi)| e^{\psi (\xi)} \, d\xi < \infty .
\end{align*}
\end{customthm}
The work \cite{HHZ} is considered sharp in terms of the decay exponent, since the  exponent $\frac{2m}{2m-1}$ in \cite{HHZ} agrees with  Theorem \ref{thm Hor} (recall  the kernel of associated fundamental solutions is of the form $e^{\mathcal{O} (\abs{x}^{2m})}$).

\subsubsection{Fourth-order Schr\"odinger equations}

Now let us come back to the topic of the current work.  Fourth-order Schr\"odinger equations with bi-Laplacian were introduced by Karpman \cite{Karpman} and Karpman-Shagalov \cite{KS} to investigate the stabilizing role of higher-order dispersive effects for the soliton instabilities in light propagation. They considered the following focusing  equation
\begin{align}
    i\partial_t u + \Delta u + \gamma \Delta^2 u =- |u|^{2p} u
\end{align}
for some $p\ge 1$. They report the following consequences of the sign of $\gamma$ on the physical dynamics observed: $\gamma>0$ gives rise to radiation and hence defocusing of a wave beam while $\gamma<0$ makes possible for formation of stable stationary wave beams \cite{Carles_2012}. The one dimensional equation with $p=1$ has been connected with nonlinear fiber optics and and optical solitons in gyrotropic media \cite{Carles_2012}. Still in one dimension, for $p>1$ and $\gamma>0$, no stable solitons exist \cite{Carles_2012}.
Also, the work by Fibich-Ilan-Papanicolaou \cite{FIP}  studies the self-focusing and singularity formation of such fourth-order Schr\"odinger equations from the mathematical viewpoint. 

Higher order Schr\"odinger equations are also important in modeling the behavior of spinless weakly relativistic and quantum mechanical particles. For a particle with mass $m$ (working in units where $\hbar =1$), the classical Schr\"odinger equation as it is known in physics takes the form
\begin{align}
    i\partial_t u = -\frac{1}{2m} \Delta u +V(x,t)u.
\end{align}
Here, $-\frac{1}{2m} \Delta u+V(x,t)$ represents the Hamiltonian (or energy) operator $\mathbf{p}^2/2m$, where $\mathbf{p}=-i\nabla$ is the \textit{momentum} operator. Here, of course, $\mathbf{p}^2/2m$ represents the kinetic energy of the particle and $V(x,t)$ its potential energy. It is natural to ask how this equation can be generalized to a particle with near-relativistic energy. Of course, the full solution was developed last century and involves the Dirac equation and the introduction of quantum fields. However, for sufficiently small energies (still relativistic), there exists a simpler and more direct approach that is still closely connected to the non-relativistic Schr\"odinger equation \cite{Carles_2012}. Indeed, we first note that the kinetic energy for a classical relativistic particle with momentum $\vec{p}\in \R^d$ takes the form
\begin{align}
    E(\vec{p})= mc^2 \left(\sqrt{1+|\vec{p}|^2/m^2c^2}-1\right)
\end{align}
with
\begin{align}
    \vec{p}= m\frac{\vec{v}}{\sqrt{1-|\vec{v}|^2/c^2}}
\end{align}
with $\vec{v}$ the velocity of the particle. 
As long as $|\vec{p}|<mc$, or $|\vec{v}|< c/\sqrt{2}$, we may express $E(\vec{p})$ as the following convergent infinite series
\begin{align}
    E(\vec{p}) = mc^2 \sum_{n=1}^\infty (-1)^{n+1} \alpha(n) \frac{|\vec{p}|^{2n}}{m^{2n}c^{2n}}
\end{align}
where
$\alpha(n) = \frac{1}{n} {2n-2\choose n-1} 2^{-2n+1}$. It is then natural to consider the Schr\"odinger equation (we assume $V\equiv 0$ from now on)
\begin{align}
    i\partial_t u = mc^2 \sum_{n=1}^\infty (-1)^{n+1} \alpha(n) \frac{\mathbf{p}^{2n}}{m^{2n}c^{2n}}u.
\end{align}
However, this is a nonlocal equation since the energy operator depends on an infinite amount of derivatives of $u$. We may thus focus on the truncated but local equation
\begin{align}
        i\partial_t u_N = mc^2 \sum_{n=1}^N (-1)^{n+1} \alpha(n) \frac{\mathbf{p}^{2n}}{m^{2n}c^{2n}}u_N.
\end{align}
Notice that if we take $N=2$, we have that
\begin{align}
    i\partial_t u_2 = \frac{\mathbf{p}^2}{2m}u_2- \frac{\mathbf{p}^4}{8m^3c^2}u_2 = -\frac{1}{2m} \Delta u_2 - \frac{1}{8m^3c^2} \Delta^2 u_2. 
\end{align}
We can see that fourth order derivatives may be seen as the lowest order relativistic correction for the Schr\"odinger equation. In fact, as long as the Fourier transform of the initial wavefunction $u_N(x,0)$ is supported inside the ball $\{\vec{p}\in \R^d: |\vec{p}|\le |\vec{p_0}|\}$ with $|\vec{p_0}|<mc$, and is smooth (hence $u_N(x,0)$ decays fast and the particle is well-localized), then we have that for any $T>0$ (see \cite{Carles_2012})
\begin{align}
    \sup_{t\in[0,T]} \norm{u-u_2}_{L^2(\R^d)} \le C_T \left(\frac{|\vec{p_0}|^2}{m^2c^2}\right)^3.
\end{align}
Hence, for small enough $|\vec{p_0}|/mc$, $u_2$ approximates the full solution $u$ quite well in the $L^2$ sense (and in particular much better than $u_1$, the solution to the standard Schr\"odinger equation).

The previous examples show the importance of studying Schr\"odinger equations containing fourth-order derivatives to better understand not only nonlinear effects in light propagation but also weakly-relativistic quantum systems through increased dispersion. \par 
The ``separable'', linear fourth-order Schr\"odinger in $d$ dimensions can be used more directly to model a system of weakly-relativistic $d$ quantum particles moving in one dimension (as long as we tolerate adding second order terms). Indeed, if we only keep second and fourth order terms and ignore dimensional constants for simplicity, the energy of such a  system is then the sum of the energies of the individual particles and the quantum evolution of the wavefunction $u(x_1, \ldots, x_d)$ takes the form
\begin{align}
    i\partial_t u +\frac{1}{2}\sum_{j=1}^d \partial_{x_j}^2 u + \frac{1}{8} \sum_{j=1}^d \partial_{x_j}^4 u= V(x_1, \ldots, x_d, t)u,
\end{align}
or using our notation,
\begin{align}
    i\partial_t u +\frac{1}{2}\sum_{j=1}^d \partial_{x_j}^2 u + \frac{1}{8} \dotDelta^2  u= V(x_1, \ldots, x_d, t)u.
\end{align}
Compare the above equation to \eqref{eq 4SE}. We can also write down a similar equation containing nonlinear interactions. Of course, the ``separable'' equations we study do not contain second-order derivatives, but our results suggest that obtaining unique continuation for these more general operators containing two or more differential operators of different order may be possible.

Analogues of unique continuation questions remain widely open for many high-dimensional dispersive equations, in particular, of the higher-order of Schr\"odinger equation type. As we pointed out \cite{HHZ} is the only work obtaining a unique continuation result of the {\it fast decay} for higher order Schr\"odinger equations in one spatial dimension. 
However, extending their argument to higher-dimensional analogues poses challenges. The authors themselves commented that the main obstacle lies in obtaining a suitable higher-dimensional Carleman estimate. This difficulty arises due to a potential phase degeneracy problem in the restriction estimate employed in their proof. It is worth noting that the result obtained by \cite{HHZ} pertains to linear unique continuation, but is not strong enough to deduce a nonlinear version. This limitation arises from the requirement that the potential function $V$ must be real, bounded, and not time-dependent in their analysis, whereas the nonlinear version requires $V$ to be at least complex-valued and time-dependent.

In this work, our goal is to extend a {\it fast decay} type of unique continuation propriety (initiated in \cite{KPV_CPAM, EKPV_CPDE}) to `separable' fourth-order Schr\"odinger equations (both linear and nonlinear), especially in higher dimensions. To the best of the authors’ knowledge, we believe that the current paper is the first result towards obtaining the unique continuation property of higher-order Schr\"odinger equations in higher dimensions. 
It is worth mentioning that this type of higher degree generalizations of the Schr\"odinger equation is not uncommon, see for example \cite{ACP} for the same generalization in the context of the study of pointwise convergence of Schr\"odinger operators.

\subsection{Main results and their sharpness}

Now let us present the main results. 
\begin{thm}[Linear unique continuation]\label{thm Main1}
Let $d \geq 1$. Assume that $u\in C^1([0,1]: H^6(\mathbb{R}^d))$ solves \eqref{eq 4SE} with $V(t,x), \nabla_x V(t,x), \nabla_x^{\, 2} V(t,x),  \nabla_x^{\, 3} V(t,x) \in L^{\infty} ([0,1]\times \R^d)$. If there exists $\lambda>0$ and $\alpha >\frac{4}{3}$ such that
\begin{align}\label{eq Weight1}
u(0,x), \, u(1,x)\in H^3(e^{\lambda|x|^\alpha}\,dx), 
\end{align}
and 
\begin{align}\label{eq supV}
    \lim_{r\to\infty} \int_0^1 \sup_{|x|>r} |V(t,x)|\,dt=0,
\end{align}
then $u(t,x) \equiv 0$. 
\end{thm}

\begin{thm}[Nonlinear unique continuation]\label{thm Main2}
Let $d \geq 1$. Assume that $u_1, u_2 \in C^1([0,1]: H^k(\mathbb{R}^d))$, with $k\in \Z^+, k > \max\{ \frac{d}{2},6 \}$ are strong solutions of \eqref{eq 4SE nonlinear} on $[0,1] \times \R^d$ with $F: \C^2 \to \C, F\in C^k$ and $F(0)=\partial_u F(0)= \partial_{\overline{u}} F(0)=0$. If there exists $\lambda>0$ and $\alpha >\frac{4}{3}$ such that
\begin{align}\label{eq Weight2}
u_1 (0, x) - u_2 (0, x) , \quad  u_1 (1, x) - u_2 (1, x) \in H^3(e^{\lambda|x|^\alpha}\,dx), 
\end{align}
then $u_1\equiv u_2 $. 
\end{thm}

\begin{rmk}[Decay notation]
Note that we say that $f \in L^2(e^{\lambda|x|^\alpha}\,dx)$ if 
\begin{align}
\int_{\R^d} \abs{f(x)}^2 e^{\lambda|x|^\alpha}\,dx < \infty ,
\end{align}
and that $f \in H^3(e^{\lambda|x|^\alpha}\,dx)$ if $f , \partial_{x_j} f , \partial_{x_j x_k } f, \partial_{x_j x_k, x_p} f \in L^2 (e^{\lambda|x|^\alpha}\,dx)$ for all $j,k,p =1 , \cdots d$. 
\end{rmk}

\begin{rmk}[Sharpness of the result and discussion on assumptions]
With the main results stated, let us make a few comments on the order of exponential weight, $e^{\lambda \abs{x}^{\alpha}}$,  $\alpha > \frac{4}{3}$. 
\begin{enumerate}
\item 
As one can see from \cite{EKPV_CPDE}, such super-Gaussian weight in the measurement of the decay of data is closely related to the quadratic weight in Hardy's uncertainty principle. 
For the general case, as we recalled in Theorem \ref{thm Hor}, the analogue of Gaussian weight in our case would be expected to be the conjugate convex weights, $ e^{\mathcal{O} (\abs{x}^{\frac{4}{3}})}$. This implies that our decay power $\alpha > \frac{4}{3}$ is almost sharp, even for the case of complex-valued and time-dependent potential $V(t,x)$.

\item
The decay rate which is described by $\lambda \abs{x}^{\alpha} $, $\alpha > \frac{4}{3}$ in \eqref{eq Weight1}  can be made better by replacing the exponential weight in \eqref{eq Weight1} by $ \lambda \abs{x}^{\frac{4}{3}}$, where $\lambda > \lambda_0$, for some $\lambda_0 >0$ well chosen. The choice of such $\lambda_0$ can be made using the same argument done in the proof of Theorem 1.1 in \cite{HHZ}.

\item
The $H^3$ regularity requirement for both solutions and the potential is not necessary. We included it in the statement of Theorem \ref{thm Main1} simply because in the proof of it, we need to differentiate the equation when deriving an exponential decay estimate for solutions with derivatives. In fact, by following the strategy in \cite{EKPV_JEMS} and introducing an artificial diffusion into the equation, we should be able to get rid of the regularity assumption. 
That is, we consider the modified equation (to fix the idea, we consider the $V=0$ case)
\begin{align}
\partial_t u = (A + iB) \D^2 u 
\end{align}
where $A >0$. 
An inherent decay given by the artificial diffusion $A \D^2$ allows one to do integration by parts freely and prove the solution up to certain derivatives preserves the same decay properties (via a logarithmic convexity) as the initial and terminal data without requiring extra regularity of the solution at all (since no differentiation of the equation is needed). Hence as a byproduct, we could even remove the $H^3$ regularity requirement on the solution. Then by taking the parameter $A \to 0$, a limiting argument gives the unique continuation properties that we desire.   
We do not plan to introduce any artificial diffusion in our proof, but instead make use of frequency cut-off operators to allow complex $V$, such as integration by parts, needed to obtain our results.

\item 
We note that \eqref{eq supV} is the same assumption made on $V$ in \cite{EKPV_CPDE}, and it will be verified when used in the nonlinear result.

\item 
The following example shows that our theorem is essentially sharp for $V(x)$ real-valued and constant in time. Indeed, we show that for any $V\in L^\infty(\R^d)$, there exists non-zero $u(t,x)$ such that $\norm{e^{\alpha|x|^{4/3}} u(0,x)}_{L^2(\R^d)}$ and $\norm{e^{\beta|x|^{4/3}} u(1,x)}_{L^2(\R^d)}$ are both finite for some $\alpha, \beta>0$. \par 
    Indeed, let $f(x)=e^{-2|x|^{4/3}}$. Now, let 
    \begin{align}
        u(t,x)= e^{-(\e+it)(\D^2+V)}f,
    \end{align}
    which solves
    \begin{align}
        i\partial u = \D^2 u + V(x) u. 
    \end{align}
    Next, we state a version of Lemma 2.1 from \cite{HHZ} which holds for our modified operator $\D^2+V(x)$ by virtue of the higher order heat kernel estimate that is Theorem 1 in \cite{DENG}. 

\begin{lem}
    Suppose $A, B\in\R$ and $V(x)\in L^\infty(\R^d)$ is real-valued. Then there exists $N_1, N_2>0$ independent of $A, B, V$ and $\Theta_{A,B}(\gamma)>0$ such that
    \begin{align}
        \norm{e^{\Theta_{A,B}(\gamma)|x|^{4/3}}e^{-(A+iB)(\D^2+V)}f}_{L^2} &\le N_1 e^{\omega_0 A \norm{V}_{L^\infty}}\left(1+ B^2/A^2 \right)^{n/2} \norm{e^{\gamma |x|^{4/3}}f}_{L^2}.
    \end{align}
\end{lem}
Letting $A=\e, B=t, \gamma=1$, we have
\begin{align}
    \norm{ e^{\Theta_{\e,t}(1)|x|^{4/3}} u(t,x)}_{L^2} \lesssim_{d}(1+t^2/\e^2)^{d/2} \norm{e^{|x|^{4/3}}f}_{L^2}.
\end{align}
Hence, we see that both $\norm{e^{\alpha|x|^{4/3}} u(0,x)}_{L^2(\R^d)}$ and $\norm{e^{\beta|x|^{4/3}} u(1,x)}_{L^2(\R^d)}$ are finite \emph{and} $u(t,x)$ is non-trivial, demonstrating the sharpness of our result. 

\end{enumerate}
\end{rmk}

\subsection{Outline of and challenges in the proof}

As we mentioned, in a series of works \cite{KPV_CPAM, EKPV_CPDE, EKPV_Duke, EKPV_JEMS, EKPV_MRL,  KPV_MRL}, the authors set out a systematic procedure to tackle the unique continuation problems with a {\it fast decay} flavor. This general method is based on a contradiction argument. We will outline the major step of the method below while listing the main ingredients needed in the contradiction arguments. Additionally, we will highlight the new ingredients we introduce to adapt to our specific fourth-order Schr\"odinger case (both linear and nonlinear).

\begin{enumerate}[(1)]
\item 
{\it Persistence of fast decay.}  The solution to \eqref{eq 4SE} is initially assumed to have fast decay only at times $t=0, 1$. By examining the evolution of the weighted norm $\|e^{\lambda |x|/d} \,u(t,x)\|_{L_x^2}$, one obtains an exponential decay estimate of the following form, for $\lambda >0 $,
\begin{align}\label{eq LogCon1}
\|e^{\lambda |x|/d} \,u (t,x)\|_{L_x^2} \leq C \|e^{\lambda  |x| } u (0,x)\|_{L_x^2} + C\|e^{\lambda |x|} u(1,x)\|_{L_x^2} ,  
\end{align}
via an energy estimate{\footnote{The weight $e^{\lambda |x|}$ can be upgraded to a weight $e^{\lambda |x|^\alpha}$ with a subordination-type inequality}}. This allows passing the fast decay property to any intermediate times ($0 < t < 1$) using a technique based on the earlier result \cite{KPV_CPAM}. This result is sufficient for our purposes, though it is slightly weaker than a similar logarithmic convexity result pertaining to solutions $u(t,x)$ of \eqref{eq LS}, 
\begin{align}\label{eq LogCon}
\|e^{\lambda\abs{x}^2} u (t,x)\|_{L_x^2} \leq C \|e^{\lambda\abs{x}^2} u (0,x)\|_{L_x^2}^{1-t} \|e^{\lambda\abs{x}^2} u(1,x)\|_{L_x^2}^t,
\end{align} 
for $\lambda>0$.

\item
{\it Carleman estimates with well-chosen weights.} As we mentioned earlier, the Carleman type of inequalities was introduced into the consideration of  uniqueness principles, and is now widely used to tackle unique continuation problems. Here is an example of Carleman type  estimates that were used in \cite{EKPV_CPDE}:
\begin{align}\label{eq Carleman}
C(\alpha ) \norm{e^{\alpha \phi(t,x)} g}_{L_{t,x}^2} \leq  \norm{e^{\alpha \phi(t,x)} (i \partial_t + \Delta) g}_{L_{t,x}^2} . 
\end{align}
It is worth emphasizing that the major challenge  in obtaining such type of inequality falls on hunting a suitable (carefully designed) weight function{\footnote{In the context of Carleman estimates, the weight function usually is the function appearing in the exponential. When we say a `linear weight function' or a `quadratic weight function', we mean $\phi$ is linear or quadratic in spatial variables.}} $\phi$  that allows one to get the estimate.  
As expected, one might utilize very different    weights when considering different models. In \cite{EKPV_JEMS}, the weight function is quadratic in space with a proper translation in the first spatial variable. We use a different but still quadratic weight to prove our Carleman estimate.

\item
{\it  Lower bounds for the solution.} Having derived a Carleman estimate, with proper localization of the solution, we are able to obtain an absolute lower bound away from $t=0,1$ for non-trivial solutions (supported on an annulus domain as a consequence). Let us remark that the lower bound depends greatly on the weight function chosen in the Carleman inequality. In fact, in order to reach a contradiction in the next step, such a lower bound has to match the fast decay rate of the solution.

\item
{\it A contradiction argument.} At this point, there are two different rates discussed: (a)  the fast decay rate of non-zero solutions at intermediate time  inherit the fast decay at $t=0,1$ (due to the  log-convexity); (b) the asymptotic lower bounds established via Carleman estimates. The difference between these two rates forces such solutions with assumed fast decay to be trivial, and completes the proof of unique continuation for the linear equation. To handle the nonlinear equation, we consider the difference of two solutions and the resulting nonlinearity as a potential, after which we can repeat the steps for the linear unique continuation result as the hypotheses of Theorem \ref{thm Main2} guarantee that the new potential in the non-linear setting satisfies the necessary conditions in  Theorem \ref{thm Main1}. 
\end{enumerate}

The main challenges in our paper lie in the proof of the exponential decay estimate for $L^2$ norms of solutions to an inhomogeneous version of \eqref{eq 4SE} as well as that of our Carleman inequality.

For the {\it persistence of fast decay} part, if one employs the strategy of the logarithmic convexity type (as \cite{HHZ} did in their work), tools such as introducing artificial dissipation into the equation, utilizing parabolic estimates, and subordination-type inequalities are commonly involved. However, this route only would yield decay estimates for real-valued and time-independent potentials. While these estimates are sufficient to establish a linear unique continuation result, it is not enough to deduce a nonlinear unique continuation result.
To this end, we decided to approach the problem slightly differently. Inspired by \cite{KPV_CPAM}, we in fact are able to obtain an $L^2$-based exponential decay for solutions to \eqref{eq 4SE} with respect to a measure of the form $e^{\beta|x|}, \beta>0$, then extend it via a subordination type inequality (Corollary 2.2 in \cite{EKPV_CPDE}) to a super-linear exponential measure of the form $e^{\lambda |x|^\alpha}, \lambda>0, \alpha>1$, from which we are able to prove persistence of \emph{fast decay} at two times. 
More precisely, to obtain Lemma \ref{lem con}, our fundamental decay estimate, that we later modify suitably, we must first cut off a weight function multiplying a solution to \eqref{inhomogeneous 4SE} to be able to rigorously apply various technologies such as integration by parts to the resulting $L^2$ norm of such weighted solutions. The difficult part lies in controlling the growth of $L^2$ norms cutoff and projected in frequency space of the weighted solutions to \eqref{inhomogeneous 4SE} since there is a large number of terms that each require qualitatively different techniques to suitably bound.

To obtain a {\it Carleman type inequality} for the operator $i\partial_t + \D^2$, we expect to see many more terms in the computation of commutators (arising from splitting the conjugate operator 
\begin{align}
\widehat{T}f(x):=e^{-\phi(x)}\,(\partial_{x_i})^{4} \left[e^{\phi(x)} f(x)\right]
\end{align}
into a symmetric and an anti-symmetric components) compared to the case for the operator $i\partial_t + \Delta$ (since the number of derivatives is twice as high as that of the classical Schr\"odinger case). Among these commutators,  most terms are computed manually, but in a couple of cases during the computation, we use a computational software to simplify extremely lengthy expressions. 
This is in fact the major reason why we consider the operator $\D^2$.
After such simplifications, we need to manipulate certain $L^2$ inner-products containing many terms in such a way that they can be lower-bounded in a positive fashion (for more details see the proof of Lemma \ref{lem Carleman}). We use this lower bound to derive a new Carleman inequality, which we would later use in our lower bound proof.

\subsection{Organization of the paper}
In Section \ref{sec Pre}, we discuss some notations and define some cutoff functions that will be used in the rest of the paper; in Section \ref{sec Logcon}, we present an exponential decay estimate for solutions to \eqref{eq 4SE} with {\it fast decay} and we upgrade it to a super-linear exponential decay estimate in Section \ref{sec Logcon+}; in Section \ref{sec Carleman}, we derive a Carleman inequality for the `separable' fourth-order Schr\"odinger operator; in Section \ref{sec Lower}, we prove a lower bound for the {\it fast decay} solutions; in Section \ref{sec Proof}, we prove the linear and nonlinear unique continuation results, by combining the lower bound and the exponential decay proved in previous sections.

\subsection*{Acknowledgement} Both authors would like to thank Gigliola Staffilani for suggesting this problem and Zongyuan Li and Luis Vega for very insightful conversations. Part of this work was done while the second author was in residence at the Institute for Computational and Experimental Research in Mathematics in Providence, RI, during the Hamiltonian Methods in Dispersive and Wave Evolution Equations program. Z. L. was supported by the Undergraduate Research Opportunities Program at the Massachusetts Institute of Technology. X.Y. was partially supported by an AMS-Simons travel grant.

\section{Preliminaries}\label{sec Pre}
In this section, we list some notations and define some cutoff functions that will be used in the rest of the paper.
\subsection{Notations}
We use the usual notation that $A \lesssim  B$ or $B \gtrsim A$ to denote an estimate of the form $A \leq C B$, for some constant $0 < C < \infty$ depending only on the {\it a priori} fixed constants of the problem.

We define the Fourier transform on $\R^d$ by
\begin{align}
\widehat{f} (\xi) : = \frac{1}{(2\pi)^{\frac{d}{2}}} \int_{\R^d} e^{-ix \cdot \xi } f(x) \, dx ,
\end{align}
and Fourier inversion 
\begin{align}
f (x) : = \frac{1}{(2\pi)^{\frac{d}{2}}} \int_{\R^d} e^{ix \cdot \xi } \widehat{f}(\xi) \, d\xi .
\end{align}

For a time interval $I$, we have the following spacetime norms $L_t^p L_x^q(I \times \R^d)$ 
\begin{align*}
\norm{u}_{L_t^p L_x^q (I \times \R^d)} : = \parenthese{\int_I \parenthese{\int_{\R^d} \abs{u(t,x)}^q \, dx }^{\frac{p}{q}}\, dt}^{\frac{1}{p}} .
\end{align*}

\subsection{Cutoff functions and chain rule}
We will frequently apply cutoff functions to the solution $u$ in later sections, hence we provide a general formula for the chain rule and product rule calculation here.

If $u$ solves \eqref{eq 4SE}, we then can find the equation for the modified $u$. That is, for a smooth function $\sigma (t,x)$, then the modified solution $\sigma (t,x) u (t,x)$ satisfies
\begin{align}\label{eq ChainRule1}
\begin{aligned}
& \quad (i\partial_t+ \D^2) [\sigma (t,x) u (t,x) ]  = i \partial_t (\sigma u ) + \D^2 ( \sigma u ) \\
& =  i (\partial_t \sigma)  u + i\sigma ( \partial_t u ) + \sum_{j=1}^d  \sigma \partial_{x_j}^{\, 4} u  +4 (\partial_{x_j} \sigma )\partial_{x_j}^{\, 3} u   + 6 (\partial_{x_j}^{\, 2} \sigma) \partial_{x_j}^{\, 2} u  + 4(\partial_{x_j}^{\, 3} \sigma ) \partial_{x_j} u +  (\partial_{x_j}^{\, 4} \sigma ) u\\
& = i (\partial_t \sigma)  u + \sigma (i\partial_t+ \D^2)u +  \sum_{j=1}^d 4 (\partial_{x_j} \sigma )\partial_{x_j}^{\, 3} u + 6 (\partial_{x_j}^{\, 2} \sigma) \partial_{x_j}^{\, 2} u + 4(\partial_{x_j}^{\, 3} \sigma )\partial_{x_j} u+ (\partial_{x_j}^{\, 4} \sigma) u  .
\end{aligned}
\end{align}

When $\sigma= \sigma (x)$, we write
\begin{align}\label{eq ChainRule2}
\begin{aligned}
& \quad (i\partial_t+\D^2) [\sigma (x) u (t,x) ] \\
& = \sigma (i\partial_t+ \D^2)u +  \sum_{j=1}^d 4 (\partial_{x_j} \sigma )\partial_{x_j}^{\, 3} u + 6 (\partial_{x_j}^{\, 2} \sigma) \partial_{x_j}^{\, 2} u + 4(\partial_{x_j}^{\, 3} \sigma )\partial_{x_j} u+ (\partial_{x_j}^{\, 4} \sigma) u  .
\end{aligned}
\end{align}

\section{Linear Exponetial Decay Estimate}\label{sec Logcon}

In this section, we prove an $L^2$-based decay estimate for solutions to an inhomogeenous version of \eqref{eq 4SE} with an exponentially weighted measure in one spatial direction. 
\begin{lem}\label{lem con}
There exists $\varepsilon_0 > 0$ such that if $V: [0,1] \times \R^d \to \C$ satisfies
\begin{align}\label{eq V}
\norm{V}_{L_t^1 L_x^{\infty}} \leq \varepsilon_0 ,
\end{align}
and $u \in C([0,1] : L_x^2 (\R^d))$ is a solution of the following perturbed equation
\begin{align}\label{inhomogeneous 4SE}
\begin{cases}
(i \partial_t + \D^2 ) u = Vu + H , & (t,x)  \in [0,1] \times \R^d \\
u(0,x) = u_0 (x)
\end{cases}
\end{align}
with $H \in L_t^1 ([0,1] : L^2 (\R^d))$ and for some $\beta \in \R$, 
\begin{align}
u_0, \, u_1 = u(1,\cdot) \in L^2 (e^{2\beta x_1} \, dx) , \quad H \in L_t^1 ([0,1] : L^2 (e^{2\beta x_1} \, dx)) .
\end{align}
Then
\begin{align}\label{eq 11}
\sup_{t \in [0,1]} \norm{u(t)}_{L^2 (e^{2\beta x_1} \, dx)}^2 \leq C ( \norm{u_0}_{L^2 (e^{2\beta x_1} \, dx)}^2 + \norm{u_1}_{L^2 (e^{2\beta x_1} \, dx)}^2 + \norm{H}_{L_t^1 L_x^2  (e^{2\beta x_1} \, dx)}^2 ) .
\end{align}
Note the constant $C$ is independent on $\beta$.
\end{lem}

\begin{rmk}[A formal proof]\label{rmk Idea}
The proof of this lemma is very computational and involves introducing cutoff functions and handling  error term produced by such truncation. Before starting the proof, let us present the main idea of the calculation. 

Let us forget the perturbation $H$ and potential $V$ for a moment. 
By considering a change of variables: $v = e^{\beta x_1}u$, we reduce \eqref{eq 11} into the following  inequality 
\begin{align}
\sup_{t \in [0,1]} \norm{v(t)}_{L_x^2 }^2 \leq C ( \norm{v(0)}_{L_x^2 }^2 + \norm{v(1)}_{L_x^2 }^2  ) .
\end{align}

Under such change of variables, we have
\begin{align}
\partial_t u & = e^{- \beta x_1} \partial_t v ,\\
\partial_{x_1}^{\, 4} v & = e^{- \beta x_1} \partial_{x_1}^{\, 4} v - 4 \beta e^{- \beta x_1} \partial_{x_1}^{\, 3} v + 6 \beta^2 e^{- \beta x_1} \partial_{x_1}^{\, 2} v - 4 \beta^3 e^{- \beta x_1} \partial_{x_1} v + \beta^4 e^{- \beta x_1} v ,\\
\partial_{x_j}^{\, 4} v & = e^{- \beta x_1} \partial_{x_j}^{\, 4} v, \quad j =2, \cdots , d .
\end{align}
which gives the differential equation that $v$ solves
\begin{align}
( i \partial_t + \D^2 ) v = - 4 \beta \partial_{x_1}^{\, 3} v + 6 \beta^2 \partial_{x_1}^{\, 2} v - 4 \beta^4 \partial_{x_1}v + \beta^4 v .
\end{align}
Taking Fourier transforms on both sides, we obtain a separable differential equation for $\widehat{v}$
\begin{align}
i \partial_t \widehat{v} + \sum_{j=1}^d (i \xi_j)^4 \widehat{v} = -4\beta (i \xi_1)^3 \widehat{v}  + 6 \beta^2 (i \xi_1)^2 \widehat{v}  - 4 \beta^3 (i \xi_1) \widehat{v}  + \beta^4 \widehat{v} 
\end{align}
which implies
\begin{align}
\partial_t \widehat{v} = i\widehat{v} ( \sum_{j=1}^d\xi_j^4 +  6 \beta^2 \xi_1^2 - \beta^4) + \widehat{v} (4 \beta \xi^3 - 4 \beta^3 \xi_1) .
\end{align}
Then we obtain $\widehat{v}$ of the following form
\begin{align}
\widehat{v} (t) = Ce^{i(\sum_{j=1}^d\xi_j^4 + 6 \beta^2 \xi_1^2 - \beta^4)t} e^{4 (\beta \xi_1^3 - \beta^3 \xi_1) t}
\end{align}
where $C$ is some initial data. 
Then  compute the $L_x^2$ norm of $v$ by Plancherel theorem
\begin{align}\label{eq Dom}
\norm{v(t)}_{L_x^2} = \norm{\widehat{v}(t)}_{L_{\xi}^2} = \norm{C e^{4 (\beta \xi_1^3 - \beta^3 \xi_1) t}}_{L_{\xi}^2}.
\end{align}
Now we see that when $\beta \xi_1^3 - \beta^3 \xi_1 >0$, $\norm{e^{4 (\beta \xi_1^3 - \beta^3 \xi_1) t}}_{L_{\xi}^2}$ increases, hence 
\begin{align}
\norm{P_{\beta \xi_1^3 - \beta^3 \xi_1 >0} v (t)}_{L_x^2} \leq  \norm{v(1)}_{L_x^2} ;
\end{align}
and when $\beta \xi_1^3 - \beta^3 \xi_1 <0$, $\norm{e^{4 (\beta \xi_1^3 - \beta^3 \xi_1) t}}_{L_x^2}$ decreases, hence 
\begin{align}
\norm{P_{\beta \xi_1^3 - \beta^3 \xi_1 <0} v (t)}_{L_x^2} \leq  \norm{v(0)}_{L_x^2} .
\end{align}
Combining these two inequalities, we arrive at our conclusion  that for $t \in [0,1]$
\begin{align}
\norm{v(t)}_{L_x^2}^2 = \norm{P_{\beta \xi_1^3 - \beta^3 \xi_1 >0} v (t)}_{L_x^2}^2 + \norm{P_{\beta \xi_1^3 - \beta^3 \xi_1 <0} v (t)}_{L_x^2}^2 \leq  \norm{v(0)}_{L_x^2}^2 + \norm{v(1)}_{L_x^2}^2.
\end{align}
This computation is considered formal since initially we did not know the $L^2$-finiteness of the new variable $v =e^{\beta x_1}u$. However, the real proof utilizes the same idea in this remark. To make sense of such a change of variables and ensure its finiteness, we need to introduce several cutoff functions and carefully handle the resulting error terms through Calder\'on's first commutator estimates.

In the rest of this section, we prove Lemma \ref{lem con} by employing the strategy outlined in this formal proof which includes the careful treatment of error terms.
\end{rmk}

Now we are ready to start the proof. 
\begin{proof}[Proof of Lemma \ref{lem con}]
Without loss of generality, we assume $\beta >0$.

\noindent {\bf Step 1: Cutoff functions and first change of variables.}
As mentioned earlier, the intuition presented in Remark \ref{rmk Idea} is based on a formal computation, and the finiteness of several quantities involved in the calculation is unknown. To address this concern, we introduce cutoff functions as a means of handling this issue. This  proof is based on an energy estimate.
\begin{itemize}
\item 
Define $\varphi_n \in C^{\infty} (\R)$, $0 \leq \varphi_n \leq 1$ such that
\begin{align}
\varphi_n (s) = 
\begin{cases}
1 , & s \leq n ,\\
0 , & s > 10n ,
\end{cases}
\end{align}
and
\begin{align} 
\abs{\varphi_n^{(k)} (s)} \leq \frac{c_k}{n^k} .
\end{align}

\item
Based on $\varphi_n$, we define $\theta_n \in C^{\infty} (\R)$, 
\begin{align}
\theta_n (s) : = \int_0^s \varphi_n^2 (\ell) \, d \ell ,
\end{align}
which satisfies
\begin{align}
\theta_n (s) = 
\begin{cases}
\beta s , & s \leq n ,\\
c_n \beta , & s > 10n .
\end{cases}
\end{align}
and
\begin{align}
\theta_n ' (s ) = \beta \varphi_n^2 (s) \leq \beta , \quad \abs{\theta_n^{(k)} (s)} \leq \frac{c_k \beta}{n^{k-1}} .
\end{align}

\item
Finally, we obtain the important modification of the weight $e^{\beta x_1}$, which is given by
\begin{align}
\Phi_n (s) = e^{\theta_n (s)}
\end{align}
and satisfies
\begin{align}
\Phi_n (s) \leq e^{\beta s} \quad \text{and} \quad \lim_{n \to \infty} \Phi_n (s) = e^{\beta s} .
\end{align}
\end{itemize}

Recalling the change of variables that  we did in Remark \ref{rmk Idea}, we write
\begin{align}
v_n (t,x) = \Phi_n (x_1) u(t,x) = e^{\theta_n (x_1) } u .
\end{align}
Note here $v_n$ is almost the $v = e^{\beta x_1} u$ in the change of variables that we did in Remark \ref{rmk Idea}. 

Then we want to find a differential equation that $v_n$ satisfies. First, we compute
\begin{align}
i \Phi_n \partial_t u  & = i \partial_t v_n ,\\
\Phi_n \partial_{x_1} u & = - \theta_n' v_n + \partial_{x_1} v_n ,\\
\Phi_n \partial_{x_1}^{\, 2} u & = (-\theta_n')^2 v_n + (-\theta_n'')v_n + 2(-\theta_n') \partial_{x_1}v_n + \partial_{x_1}^{\, 2} v_n ;\\
\Phi_n \partial_{x_1}^{\, 3} u & = (-\theta_n')^3 v_n + 3 (-\theta_n')(-\theta_n'') v_n + (-\theta_n''')v_n + 3 (-\theta_n')^2 \partial_{x_1}v_n + 3 (-\theta_n'')\partial_{x_1} v_n + 3 (-\theta_n') \partial_{x_1}^{\, 2} v_n + \partial_{x_1}^{\, 3} v_n\\
\Phi_n \partial_{x_1}^{\, 4} u & = (-\theta_n')^4 v_n + 6 (-\theta_n')^2 (-\theta_n'') v_n  + 4 (-\theta_n')(-\theta_n''')v_n + 3 (-\theta_n'')^2 v_n + (-\theta_n'''')v_n \\
& \quad + 4 (-\theta_n')^3 \partial_{x_1}v_n + 12 (-\theta_n')(-\theta_n'')\partial_{x_1}v_n + 4 (-\theta_n''') \partial_{x_1} v_n\\
& \quad + 6 (-\theta_n')^2 \partial_{x_1}^{\, 2} v_n + 6 (-\theta_n'') \partial_{x_1}^{\, 2} v_n + 4 (-\theta_n')\partial_{x_1}^{\, 3} v_n    + \partial_{x_1}^{\, 4} v_n ,\\
\Phi_n \partial_{x_j}^{\, 4} u & = e^{\theta_n} \partial_{x_j}^{\, 4} (e^{-\theta_n} v_n) = \partial_{x_j}^{\, 4} v_n , \qquad j = 2, \cdots , d .
\end{align} 
Then putting the derivatives above, we get the following equation
\begin{align}
\Phi_n H + \Phi_n V u & = \Phi_n (i \partial_t u + \D^2 u  ) \\
& = i\partial_t v_n + \D^2  v_n  \\
& \quad + [(-\theta_n')^4 v_n + 6 (-\theta_n')^2 (-\theta_n'') v_n  + 4 (-\theta_n')(-\theta_n''')v_n + 3 (-\theta_n'')^2 v_n + (-\theta_n'''')v_n\\
& \quad + 4 (-\theta_n')^3 \partial_{x_1}v_n + 12 (-\theta_n')(-\theta_n'')\partial_{x_1}v_n + 4 (-\theta_n''') \partial_{x_1} v_n\\
& \quad + 6 (-\theta_n')^2 \partial_{x_1}^{\, 2} v_n + 6 (-\theta_n'') \partial_{x_1}^{\, 2} v_n + 4 (-\theta_n')\partial_{x_1}^{\, 3} v_n ] .
\end{align}
Using \eqref{eq 4SE}, we write
\begin{align}\label{eq v_n}
\begin{aligned}
i\partial_t v_n + \D^2  v_n & =  - v_n [(-\theta_n')^4  + 6 (-\theta_n')^2 (-\theta_n'')   + 4 (-\theta_n')(-\theta_n''') + 3 (-\theta_n'')^2  + (-\theta_n'''')]\\
& \quad - \partial_{x_1}v_n [ 4 (-\theta_n')^3  + 12 (-\theta_n')(-\theta_n'') + 4 (-\theta_n''') ]\\
& \quad - \partial_{x_1}^{\, 2} v_n[ 6 (-\theta_n')^2  + 6 (-\theta_n'')] \\
& \quad -  \partial_{x_1}^{\, 3} v_n [ 4(-\theta_n')] + \Phi_n H +  V v_n ,
\end{aligned}
\end{align}
where
\begin{align}
\theta_n'& = \beta \varphi_n^2 ,\\
\theta_n'' & = 2 \beta \varphi_n \varphi_n' ,\\
\theta_n'''& = 2\beta [ (\varphi_n')^2 + \varphi_n \varphi_n''] ,\\
\theta_n''''& = 2\beta [3 \varphi_n' \varphi_n'' + \varphi_n \varphi_n'''] .
\end{align}

\noindent {\bf Step 2: Second change of variables.}
We wish to compute $\partial_t \norm{v_n}_{L_x^2}^2$, however, there is a   constant multiply of $\norm{v_n}_{L_x^2}^2$ on the right-hand side of \eqref{eq v_n} (to be more precise, it is the  first term $-v_n (- \theta_n')^4$), which will not be made small when doing  estimates at the very end (recall $\varphi_n \sim 1$ when $s \leq n$). Hence we remove this term by another change of variables (this only removes the non-vanishing term and will not change the $L^2$ norm at all) before computing $\partial_t \norm{v_n}_{L^2}$ via
\begin{align}
w_n = e^{-i (-\theta_n')^4 t} v_n = : e^{\mu} v_n .
\end{align}
Similarly, we need to find a differential equation that $w_n$ satisfies. 
\begin{align}
e^{\mu}\partial_t v_n & =  i (-\theta_n')^4 w_n +  \partial_t w_n , \\
e^{\mu} \partial_{x_1}^{\, 4} v_n & = e^{\mu} \partial_{x_1}^{\, 4} (e^{-\mu} w_n) ,\\
& = \partial_{x_1}^{\, 4} w_n - 4 \partial_{x_1}^{\, 3} w_n \mu' + \partial_{x_1}^{\, 2} w_n [6 (\mu')^2 - 6 \mu''] + \partial_{x_1} w_n [12 \mu' \mu'' - 4(\mu')^3 - 4\mu'''] \\
& \quad + w_n [ -6(\mu')^2 \mu'' + (\mu')^4 + 4 \mu''' \mu' + 3 (\mu'')^2 - \mu''''] ,\\
e^{\mu} \partial_{x_j}^{\, 4} v_n & = e^{\mu} \partial_{x_j}^{\, 4} e^{-\mu} w_n = \partial_{x_j}^{\, 4} w_n , \qquad j = 2, \cdots , d .
\end{align}
where
\begin{align}
\mu & = - i (-\theta_n')^4 t ,\\
\mu' & = -it [4 (-\theta_n')^3 (-\theta_n'')] ,\\
\mu''& = -it [12 (-\theta_n')^2 (-\theta_n'')^2 + 4 (-\theta_n')^3 (-\theta_n''')] ,\\
\mu''' & = -it [24 (-\theta_n')(-\theta_n'')^3 + 36 (-\theta_n')^2(-\theta_n'')(-\theta_n''') + 4 (-\theta_n')^3 (-\theta_n'''')] ,\\
\mu''''& = -it [24(-\theta_n'')^4 + 144(-\theta_n')(-\theta_n'')^2 (-\theta_n''') + 36(-\theta_n')^2 (-\theta_n''')^2 \\
& \quad + 48(-\theta_n')^2(-\theta_n'')(-\theta_n'''') + 4(-\theta_n')^3 (-\theta_n''''')] .
\end{align}

Putting the derivatives together, we obtain
\begin{align}\label{eq v_n1}
\begin{aligned}
e^{\mu} (i \partial_t + \D^2 ) v_n & = - (-\theta_n')^4 w_n + i \partial_t w_n + \D^2 w_n \\
& \quad - 4 \partial_{x_1}^{\, 3} w_n \mu' + \partial_{x_1}^{\, 2} w_n [6 (\mu')^2 - 6 \mu''] + \partial_{x_1} w_n [12 \mu' \mu'' - 4(\mu')^3 - 4\mu'''] \\
& \quad + w_n [ -6(\mu')^2 \mu'' + (\mu')^4 + 4 \mu''' \mu' + 3 (\mu'')^2 - \mu''''] .
\end{aligned}
\end{align}
Substituting  the RHS of  \eqref{eq v_n} into the LHS of \eqref{eq v_n1}, we then write
\begin{align}
& e^{\mu}  \{- v_n [(-\theta_n')^4  + 6 (-\theta_n')^2 (-\theta_n'')   + 4 (-\theta_n')(-\theta_n''') + 3 (-\theta_n'')^2  + (-\theta_n'''')]\\
& \quad - \partial_{x_1}v_n [ 4 (-\theta_n')^3  + 12 (-\theta_n')(-\theta_n'') + 4 (-\theta_n''') ]\\
& \quad - \partial_{x_1}^{\, 2} v_n[ 6 (-\theta_n')^2  + 6 (-\theta_n'')] \\
& \quad -  \partial_{x_1}^{\, 3} v_n [ 4(-\theta_n')] + \Phi_n H  + V v_n\} \\
& = - (-\theta_n')^4 w_n + i \partial_t w_n + \D^2 w_n \\
& \quad - 4 \partial_{x_1}^{\, 3} w_n \mu' + \partial_{x_1}^{\, 2} w_n [6 (\mu')^2 - 6 \mu''] + \partial_{x_1} w_n [12 \mu' \mu'' - 4(\mu')^3 - 4\mu'''] \\
& \quad + w_n [ -6(\mu')^2 \mu'' + (\mu')^4 + 4 \mu''' \mu' + 3 (\mu'')^2 - \mu''''] .
\end{align}

Hence shuffling the terms in the equation above, we get the following 
\begin{align}\label{eq w_n1}
\begin{aligned}
(i \partial_t + \D^2) w_n & = - \{- 4 \partial_{x_1}^{\, 3} w_n \mu' + \partial_{x_1}^{\, 2} w_n [6 (\mu')^2 - 6 \mu''] + \partial_{x_1} w_n [12 \mu' \mu'' - 4(\mu')^3 - 4\mu'''] \\
& \quad + w_n [ -6(\mu')^2 \mu'' + (\mu')^4 + 4 \mu''' \mu' + 3 (\mu'')^2 - \mu'''']\} \\
& \quad - e^{\mu} v_n [6 (-\theta_n')^2 (-\theta_n'')   + 4 (-\theta_n')(-\theta_n''') + 3 (-\theta_n'')^2  + (-\theta_n'''')]\\
& \quad - e^{\mu} \partial_{x_1}v_n [ 4 (-\theta_n')^3  + 12 (-\theta_n')(-\theta_n'') + 4 (-\theta_n''') ]\\
& \quad - e^{\mu} \partial_{x_1}^{\, 2} v_n[ 6 (-\theta_n')^2  + 6 (-\theta_n'')] \\
& \quad - e^{\mu} \partial_{x_1}^{\, 3} v_n [ 4(-\theta_n')] + e^{\mu} [\Phi_n H  + V v_n ] .
\end{aligned}
\end{align}

Noticing that  the $v_n$ terms in \eqref{eq w_n1} can be rewritten as
\begin{align}
e^{\mu} \partial_{x_1} v_n & = - \mu' w_n + \partial_{x_1} w_n , \\
e^{\mu} \partial_{x_1}^{\, 2} v_n & = (-\mu')^2 w_n + (-\mu'') w_n + 2(-\mu') \partial_{x_1}w_n + \partial_{x_1}^{\, 2} w_n ,\\
e^{\mu} \partial_{x_1}^{\, 3} v_n & = (-\mu')^3 w_n + 3 (-\mu')(-\mu'') w_n + (-\mu''')w_n + 3 (-\mu')^2 \partial_{x_1} w_n \\
& \quad + 3 (-\mu'')\partial_{x_1} w_n + 3(-\mu') \partial_{x_1}^{\, 2} w_n + \partial_{x_1}^{\, 3} w_n ,
\end{align}
then we have the following equivalent form for the last four lines involving $v_n$ in \eqref{eq w_n1}
\begin{align}
\text{Last four lines in \eqref{eq w_n1}}
& = w_n \{-6 (-\theta_n')^2 (-\theta_n'') -4 (-\theta_n') (-\theta_n''') - 3 (-\theta_n'')^2 - (-\theta_n'''') + (-\mu') [-4(-\theta_n')^3 \\
& \quad - 12 (-\theta_n')(-\theta_n'') -4(-\theta_n''')]  + [(-\mu')^2  + (-\mu'')] \cdot [-6(-\theta_n')^2 - 6(-\theta_n'')] \\
& \quad + [(-\mu')^3 + 3(-\mu')(-\mu'') + (-\mu''')] \cdot [-4(-\theta_n')]\} \\
& \quad + \partial_{x_1} w_n \{-4 (-\theta_n')^3 -12(-\theta_n')(-\theta_n'') - 4 (-\theta_n''') + 2(-\mu') [-6(-\theta_n')^2 - 6(-\theta_n'')] \\
& \quad + [3 (-\mu')^2  + 3 (-\mu'')]\cdot [-4(-\theta_n')]  \} \\
& \quad + \partial_{x_1}^{\, 2} w_n [-6 (-\theta_n')^2 - 6(-\theta_n'') + 3 (-\mu')] \cdot [-4 (-\theta_n')] \\
& \quad + \partial_{x_1}^{\, 3} w_n [-4(-\theta_n')]  + e^{\mu} \Phi_n H  + V w_n  .
\end{align}

Now we finally find a differential equation that $w_n$ satisfies, and \eqref{eq w_n1} becomes
\begin{align}\label{eq w_n}
\begin{aligned}
(i\partial_t + \D^2 ) w_n & = w_n [q_0 (x_1) ] + \partial_{x_1} w_n [a_1^2 (x_1) + q_1 (x_1)] + \partial_{x_1}^{\, 2} w_n [-a_2^2(x_1) + q_2 (x_1)] \\
& \quad + \partial_{x_1}^{\, 3} w_n [a_3^2(x_1) + i t b (x_1)] + e^{\mu} \Phi_n (x_1) H  + V w_n  , 
\end{aligned}
\end{align}
where
\begin{align}
q_0 (x_1) & = -[ -6(\mu')^2 \mu'' + (\mu')^4 + 4 \mu''' \mu' + 3 (\mu'')^2 - \mu'''']  -6 (-\theta_n')^2 (-\theta_n'')  \\
& \quad -4 (-\theta_n') (-\theta_n''') - 3 (-\theta_n'')^2 - (-\theta_n'''') + (-\mu')\cdot [-4(-\theta_n')^3 - 12 (-\theta_n')(-\theta_n'') -4(-\theta_n''')] \\
& \quad + [(-\mu')^2  + (-\mu'')] \cdot [-6(-\theta_n')^2 - 6(-\theta_n'')] + [(-\mu')^3 + 3(-\mu')(-\mu'') + (-\mu''')] \cdot [-4(-\theta_n')] ,\\
a_1^2(x_1) & = -4 (-\theta_n')^3 ,\\
q_1(x_1) & = -[12 \mu' \mu'' - 4(\mu')^3 - 4\mu'''] ,\\
& \quad  -12(-\theta_n')(-\theta_n'') - 4 (-\theta_n''') + 2(-\mu') \cdot [-6(-\theta_n')^2 - 6(-\theta_n'')] + [3 (-\mu')^2 + 3 (-\mu'')]\cdot [-4(-\theta_n')] \\
-a_2^2 (x_1) & = -6 (-\theta_n')^2 [-4(-\theta_n')] ,\\
q_2 (x_1) & = - [6 (\mu')^2 - 6 \mu''] + [ - 6(-\theta_n'') + 3 (-\mu')] \cdot [-4 (-\theta_n')] ,\\
a_3^2 (x_1) & = - 4 (-\theta_n') ,\\
it b(x_1) & = 4 \mu' .
\end{align}
Recall 
\begin{align}
\mu & = - i (-\theta_n')^4 t ,\\
\mu' & = -it [4 (-\theta_n')^3 (-\theta_n'')] ,\\
\mu''& = -it [12 (-\theta_n')^2 (-\theta_n'')^2 + 4 (-\theta_n')^3 (-\theta_n''')] ,\\
\mu''' & = -it [24 (-\theta_n')(-\theta_n'')^3 + 36 (-\theta_n')^2(-\theta_n'')(-\theta_n''') + 4 (-\theta_n')^3 (-\theta_n'''')] , \\
\mu''''& = -it [24(-\theta_n'')^4 + 144(-\theta_n')(-\theta_n'')^2 (-\theta_n''') + 36(-\theta_n')^2 (-\theta_n''')^2 , \\
& \quad + 48(-\theta_n')^2(-\theta_n'')(-\theta_n'''') + 4(-\theta_n')^3 (-\theta_n''''')] ,
\end{align}
and 
\begin{align}
\theta_n'& = \beta \varphi_n^2 , \\
\theta_n'' & = 2 \beta \varphi_n \varphi_n' , \\
\theta_n'''& = 2\beta [ (\varphi_n')^2 + \varphi_n \varphi_n''] ,\\
\theta_n''''& = 2\beta [3 \varphi_n' \varphi_n'' + \varphi_n \varphi_n'''] .
\end{align}

We observe the following decay properties from the coefficients in \eqref{eq w_n}, that is, for $k \in \N$
\begin{align}\label{eq Decay}
\begin{aligned}
&\norm{\partial_{x_1}^{\, k} q_j }_{L^{\infty}} \leq \frac{c}{n^{k+1}}, \quad j =0,1,2 ,\\
&\norm{\partial_{x_1}^{\, k} a_j^2}_{L^{\infty}} \leq \frac{c}{n^{k}} ,\quad j=1,2,3 , \\
&\norm{\partial_{x_1}^{\, k} b}_{L^{\infty}} \leq \frac{c}{n^{k+1}} .
\end{aligned}
\end{align}

We remark here that due to the second change of variables $v_n \to w_n$, we successfully removed a constant multiply of $v_n$ and only left with a decaying  coefficient $q_0 (x_1)$ times $w_n$.

\noindent {\bf Step 3: An energy estimate on $w_n$.}
In the rest of the proof, we work on estimating the $L^2$ norm of $w_n$. Starting by introducing a couple of Fourier multipliers. 
\begin{itemize}
\item 
Define
\begin{align}
\chi_{+} (\xi)  = 
\begin{cases}
1 & \text{if } \xi_1 \in (-\beta ,0) \cup (\beta , \infty) ,\\
0 & \text{if } \xi_1 \in  (-\infty, -\beta] \cup [0, \beta]  ,
\end{cases}
\end{align}
and 
\begin{align}
\chi_{-} (\xi)  = 
\begin{cases}
1 & \text{if } \xi_1 \in (-\infty, -\beta] \cup [0, \beta] ,\\
0 & \text{if } \xi_1 \in  (-\beta ,0) \cup (\beta , \infty) . 
\end{cases}
\end{align}

\item
We also define $\eta \in C_0^{\infty} (\R^d)$ with $0 \leq \eta (x) \leq 1$ and 
\begin{align}
\eta (x) = 
\begin{cases}
1 & \text{if }  \abs{x} \leq \frac{1}{2} , \\
0 & \text{if } \abs{x} \geq 1  .
\end{cases}
\end{align}

\item
Then we define two projections, for $\varepsilon \in (0,1]$
\begin{align}
\widehat{P_{\varepsilon}f} (\xi)  & : = \eta_\varepsilon (\xi) \widehat{f} (\xi)  =\eta (\varepsilon \xi) \widehat{f} (\xi) ,\\
\widehat{P_{\pm}f} (\xi) & : = \chi_{\pm} (\xi) \widehat{f} (\xi) .
\end{align}
\end{itemize}

We remark that (1) the projections $P_{\pm}$ are defined based on the formal calculation in Remark \ref{rmk Idea}, which allow the dominant term in \eqref{eq Dom} in the $L^2$ norm of $v_n$ (or $w_n$) to have a sign; (2) the projection $P_{\varepsilon}$ permits the freedom to do any integration by parts in the frequency space.

We now want to derive equations for $P_{\varepsilon}P_{\pm} w_n$ by applying the projection to each term in \eqref{eq w_n}:
\begin{align}\label{eq P1}
\begin{aligned}
i \partial_t P_{\varepsilon}P_{+} w_n + \D^2 P_{\varepsilon}P_{+} w_n & = P_{\varepsilon}P_{+} w_n [q_0 (x_1) ] + P_{\varepsilon}P_{+} \partial_{x_1} w_n [a_1^2 (x_1) + q_1 (x_1)] \\
& \quad + P_{\varepsilon}P_{+} \partial_{x_1}^{\, 2} w_n [-a_2^2(x_1) + q_2 (x_1)] + P_{\varepsilon}P_{+} \partial_{x_1}^{\, 3} w_n [a_2^3(x_1) + i t b (x_1)] \\
& \quad + P_{\varepsilon}P_{+} e^{\mu} \Phi_n (x_1) H + P_{\varepsilon}P_{+} V w_n  ,
\end{aligned}
\end{align}
and
\begin{align}\label{eq P2}
\begin{aligned}
-i \partial_t \overline{P_{\varepsilon}P_{+} w_n} + \D^2 \overline{ P_{\varepsilon}P_{+} w_n} & = \overline{P_{\varepsilon}P_{+} w_n [q_0 (x_1) ]} + \overline{P_{\varepsilon}P_{+} \partial_{x_1} w_n [a_1^2 (x_1) + q_1 (x_1)]} \\
& \quad+ \overline{P_{\varepsilon}P_{+} \partial_{x_1}^{\, 2} w_n [-a_2^2(x_1) + q_2 (x_1)]}  + \overline{P_{\varepsilon}P_{+} \partial_{x_1}^{\, 3} w_n [a_2^3(x_1) + i t b (x_1)]} \\
& \quad + \overline{P_{\varepsilon}P_{+} e^{\mu} \Phi_n (x_1) H} + \overline{P_{\varepsilon}P_{+} V w_n } .
\end{aligned}
\end{align}

Multiplying \eqref{eq P1} and \eqref{eq P2} by $\overline{P_{\varepsilon}P_{+} w_n}$ and $- P_{\varepsilon}P_{+} w_n$, respectively, and adding the result, we obtain
\begin{align}
& i \partial_t \abs{P_{\varepsilon}P_{+} w_n}^2 + \D^2  P_{\varepsilon}P_{+} w_n \cdot \overline{P_{\varepsilon}P_{+} w_n} - \overline{\D^2 P_{\varepsilon}P_{+} w_n} \cdot P_{\varepsilon}P_{+} w_n \\
& = P_{\varepsilon}P_{+} w_n [q_0 (x_1) ] \cdot  \overline{P_{\varepsilon}P_{+} w_n} - \overline{P_{\varepsilon}P_{+} w_n [q_0 (x_1) ]} \cdot P_{\varepsilon}P_{+} w_n \\
& \quad + P_{\varepsilon}P_{+} \partial_{x_1} w_n [a_1^2 (x_1) + q_1 (x_1)] \cdot \overline{P_{\varepsilon}P_{+} w_n} - \overline{P_{\varepsilon}P_{+} \partial_{x_1} w_n [a_1^2 (x_1) + q_1 (x_1)]} \cdot P_{\varepsilon}P_{+} w_n \\
& \quad + P_{\varepsilon}P_{+} \partial_{x_1}^{\, 2} w_n [-a_2^2(x_1) + q_2 (x_1)] \cdot \overline{P_{\varepsilon}P_{+} w_n} - \overline{P_{\varepsilon}P_{+} \partial_{x_1}^{\, 2} w_n [-a_2^2(x_1) + q_2 (x_1)]} \cdot P_{\varepsilon}P_{+} w_n \\
& \quad + P_{\varepsilon}P_{+} \partial_{x_1}^{\, 3} w_n [a_2^3(x_1) + i t b (x_1)] \cdot \overline{P_{\varepsilon}P_{+} w_n} - \overline{P_{\varepsilon}P_{+} \partial_{x_1}^{\, 3} w_n [a_2^3(x_1) + i t b (x_1)]} \cdot P_{\varepsilon}P_{+} w_n \\
& \quad + P_{\varepsilon}P_{+} e^{\mu} \Phi_n (x_1) H   \cdot \overline{P_{\varepsilon}P_{+} w_n} - \overline{P_{\varepsilon}P_{+} e^{\mu} \Phi_n (x_1) H } \cdot P_{\varepsilon}P_{+} w_n \\
& \quad + P_{\varepsilon}P_{+} V w_n   \cdot \overline{P_{\varepsilon}P_{+} w_n} - \overline{P_{\varepsilon}P_{+} V w_n } \cdot P_{\varepsilon}P_{+} w_n  ,
\end{align}
and taking the imaginary part in the equation above yields
\begin{align}
& \partial_t \abs{P_{\varepsilon}P_{+} w_n}^2 + 2 \im (\D^2  P_{\varepsilon}P_{+} w_n \cdot \overline{P_{\varepsilon}P_{+} w_n})  \label{eq wn}\\
& = 2 \im (P_{\varepsilon}P_{+} w_n [q_0 (x_1) ] \cdot  \overline{P_{\varepsilon}P_{+} w_n} ) \label{eq w0}\\
& \quad + 2 \im ( P_{\varepsilon}P_{+} \partial_{x_1} w_n [a_1^2 (x_1) + q_1 (x_1)] \cdot \overline{P_{\varepsilon}P_{+} w_n} ) \label{eq w1}\\
& \quad + 2 \im (P_{\varepsilon}P_{+} \partial_{x_1}^{\, 2} w_n [-a_2^2(x_1) + q_2 (x_1)] \cdot \overline{P_{\varepsilon}P_{+} w_n} )  \label{eq w2}\\
& \quad + 2 \im (P_{\varepsilon}P_{+} \partial_{x_1}^{\, 3} w_n [a_3^2(x_1) ] \cdot \overline{P_{\varepsilon}P_{+} w_n} )  \label{eq w3}\\
& \quad + 2 \re (P_{\varepsilon}P_{+} \partial_{x_1}^{\, 3} w_n [t b (x_1)] \cdot \overline{P_{\varepsilon}P_{+} w_n} )  \label{eq w4}\\
& \quad + 2 \im (P_{\varepsilon}P_{+} e^{\mu} \Phi_n (x_1) H   \cdot \overline{P_{\varepsilon}P_{+} w_n} )  \label{eq w5}\\
& \quad + 2 \im (P_{\varepsilon}P_{+} V w_n   \cdot \overline{P_{\varepsilon}P_{+} w_n}) \label{eq w6} .
\end{align}

Now we will integrate \eqref{eq wn} and estimate each term in this integration.

\noindent {\it \underline{Easy terms.}}
Since for all $n \in \N$, $w_n \in L_x^2 (\R^d)$, $e^{\mu} \Phi_n (x_1) H \in L_x^2 (\R^d)$,  we have
\begin{align}
\im \int_{\R^d} \D^2  P_{\varepsilon}P_{+} w_n \cdot \overline{P_{\varepsilon}P_{+} w_n} \, dx =0  . 
\end{align}
Also we have for terms \eqref{eq w5} and \eqref{eq w6}
\begin{align}
\abs{\im \int_{\R^d} P_{\varepsilon}P_{+} e^{\mu} \Phi_n (x_1) H   \cdot \overline{P_{\varepsilon}P_{+} w_n} \, dx} & \leq c \norm{e^{\mu} \Phi_n (x_1) H }_{L_x^2} \norm{P_{\varepsilon} w_n}_{L_x^2} , \\
\abs{\im \int_{\R^d} P_{\varepsilon}P_{+} V w_n   \cdot \overline{P_{\varepsilon}P_{+} w_n} \, dx} & \leq c \norm{V }_{L_x^{\infty}} \norm{P_{\varepsilon} w_n}_{L_x^2}^2  ,
\end{align}
and term \eqref{eq w0} by \eqref{eq Decay}
\begin{align}
\abs{\im \int_{\R^d} P_{\varepsilon}P_{+} w_n [q_0 (x_1) ] \cdot  \overline{P_{\varepsilon}P_{+} w_n} \, dx } \leq c \norm{q_0}_{L_x^{\infty}} \norm{P_{\varepsilon} w_n}_{L_x^2}^2 \leq \frac{c}{n} \norm{P_{\varepsilon} w_n}_{L_x^2}^2 .
\end{align}

\noindent {\it \underline{Preparation.}}
To deal with other terms \eqref{eq w1} - \eqref{eq w4}, we recall Calder\'on first commutator estimates in \cite{Cal, Stein} which were also used in \cite{KPV_CPAM}
\begin{align}\label{eq Calderon}
\begin{aligned}
\norm{[P_{\pm} ; a] \partial_{x_1}f}_{L^2} \leq c \norm{\partial_{x_1}a}_{L^{\infty}} \norm{f}_{L^2} ,\\
\norm{\partial_{x_1} [P_{\pm} ; a]f}_{L^2} \leq c \norm{\partial_{x_1}a}_{L^{\infty}} \norm{f}_{L^2} , \\
\norm{[P_{\varepsilon} ; a] \partial_{x_1}f}_{L^2} \leq c \norm{\partial_{x_1}a}_{L^{\infty}} \norm{f}_{L^2} ,\\
\norm{\partial_{x_1} [P_{\varepsilon} ; a]f}_{L^2} \leq c \norm{\partial_{x_1}a}_{L^{\infty}} \norm{f}_{L^2} .
\end{aligned}
\end{align}
We also recall Claim 1 and Claim 2 in \cite{KPV_CPAM} here.

\begin{claim}[Claim 1 and Claim 2 in \cite{KPV_CPAM}]\label{claim}
Using Calder\'on first commutator estimates, we have
\begin{enumerate}
\item
For $a^2 (x_1) \geq 0$
\begin{align}
\im \int_{\R^d} P_{\varepsilon} P_{+} (a^2 (x_1) \partial_{x_1} w_n) \cdot \overline{P_{\varepsilon} P_{+} w_n} \, dx  = \im \int_{\R^d} \partial_{x_1}  P_{\varepsilon} P_{+} (a (x_1)w_n) \cdot  \overline{P_{\varepsilon} P_{+} (a (x_1) w_n)}  \, dx + \mathcal{O} (\frac{\norm{P_{\varepsilon} w_n}_{L_x^2}^2}{n}) .
\end{align}

\item
For $b(x_1)$ pure imagery
\begin{align}
\im \int_{\R^d} P_{\varepsilon} P_{+} (b (x_1) \partial_{x_1} w_n) \cdot \overline{P_{\varepsilon} P_{+} w_n} \, dx = \mathcal{O} (\frac{\norm{P_{\varepsilon} w_n}_{L_x^2}^2}{n}) .
\end{align}
\end{enumerate}
\end{claim}

\noindent {\it \underline{Term \eqref{eq w1}.}}
For the contribution from $a_1^2 (x_1)$ in term \eqref{eq w1}, Item (1) in Claim \ref{claim} gives
\begin{align}
\im \int_{\R^d} P_{\varepsilon} P_{+} (a_1^2 (x_1) \partial_{x_1} w_n) \cdot \overline{P_{\varepsilon} P_{+} w_n} \, dx  = \im \int_{\R^d} \partial_{x_1}  P_{\varepsilon} P_{+} (a_1 (x_1)w_n) \cdot  \overline{P_{\varepsilon} P_{+} (a_1 (x_1) w_n)}  \, dx + \mathcal{O} (\frac{\norm{P_{\varepsilon} w_n}_{L_x^2}^2}{n}) .
\end{align}

Then using Parseval's identity, we write 
\begin{align}
\im \int_{\R^d} \partial_{x_1}  P_{\varepsilon} P_{+} (a_1 (x_1)w_n)  \cdot  \overline{P_{\varepsilon} P_{+} (a_1 (x_1) w_n)}  \, dx & = \im \int_{\R^d} (i \xi_1) \widehat{P_{\varepsilon} P_{+} (a_1 (x_1)w_n)}  \cdot  \overline{\widehat{P_{\varepsilon} P_{+} (a_1 (x_1)w_n)}} \, d\xi\\
& = \re \int_{\R^d} \xi_1 \abs{\widehat{P_{\varepsilon} P_{+} (a_1 (x_1)w_n)}}^2 \, d\xi . \label{eq Absorb}
\end{align}

Now let us turn to the contribution from $q_1 (x_1)$ to \eqref{eq w1}. Since $\norm{q_1(x)}_{L_x^{\infty}} \leq \frac{c}{n}$, hence for $n$ large enough, 
\begin{align}
a_{1}^2 (x_1) + \re q_1 (x_1) = \widetilde{a}_{1}^2 (x_1) \geq 0.
\end{align}
Then item (2) in Claim \ref{claim} yields
\begin{align}
\im \int_{\R^d} P_{\varepsilon} P_{+} (\im q_1 (x_1) \partial_{x_1} w_n) \cdot \overline{P_{\varepsilon} P_{+} w_n} \, dx = \mathcal{O} (\frac{\norm{P_{\varepsilon} w_n}_{L_x^2}^2}{n}) .
\end{align}
Combining \eqref{eq Absorb} with 
\begin{align}
\int  \eqref{eq w1} = \re \int_{\R^d}\xi_1 \abs{\widehat{P_{\varepsilon} P_{+} (\widetilde{a}_1 (x_1)w_n)}}^2 \, d\xi + \mathcal{O} (\frac{\norm{P_{\varepsilon} w_n}_{L_x^2}^2}{n}) .
\end{align}

\noindent {\it \underline{Term \eqref{eq w2}.}}
Let us then take \eqref{eq w2} and start with the contribution of $a_2^2$ term. First using the product rule, we write
\begin{align}\label{eq 13}
a_2^2  \partial_{x_1}^{\, 2} w_n = a_2\partial_{x_1}^{\, 2} (a_2w_n) - a_2  (\partial_{x_1}^{\, 2} a_2) w_n - 2 a_2 (\partial_{x_1} a_2) (\partial_{x_1} w_n) .
\end{align}
Notice that using \eqref{eq Decay}
\begin{align}
\abs{\im \int_{\R^d}P_{\varepsilon}P_{+} a_2 (\partial_{x_1}^{\, 2} a_2) w_n \cdot \overline{P_{\varepsilon}P_{+} w_n} \, dx } = \mathcal{O} (\frac{\norm{P_{\varepsilon} w_n}_{L_x^2}^2}{n}) . 
\end{align}
Since by \eqref{eq Decay}
\begin{align}
\norm{a_2 (\partial_{x_1} a_2)}_{L_x^{\infty}} \leq \frac{c}{n} ,
\end{align}
hence when $n$ large enough it can be similarly  absorbed by \eqref{eq Absorb} without changing the sign of $a_1^2$ (just replace $a_1^2$  by a slightly different $\widetilde{a}_1^2$).

For the first term on the right hand side of \eqref{eq 13}, using \eqref{eq Calderon} and integration by parts, we write
\begin{align}
& - \im \int_{\R^d} P_{\varepsilon}P_{+} a_2^2(x_1) \partial_{x_1}^{\, 2} w_n  \cdot \overline{P_{\varepsilon}P_{+} w_n} \, dx \\
& = - \im \int_{\R^d} P_{\varepsilon}P_{+} a_2(x_1)\partial_{x_1}^{\, 2} (a_2(x_1)w_n)  \cdot \overline{P_{\varepsilon}P_{+} w_n} \, dx + \mathcal{O} (\frac{\norm{P_{\varepsilon} w_n}_{L_x^2}^2}{n})  \\
& = - \im \int_{\R^d} a_2(x_1) \partial_{x_1} P_{\varepsilon}P_{+} \partial_{x_1} (a_2(x_1)w_n)  \cdot \overline{P_{\varepsilon}P_{+} w_n} \, dx + \mathcal{O} (\frac{\norm{P_{\varepsilon} w_n}_{L_x^2}^2}{n}) \\
& = \im \int_{\R^d}  P_{\varepsilon}P_{+} \partial_{x_1} (a_2(x_1)w_n)  \cdot  \partial_{x_1} a_2(x_1) \overline{ P_{\varepsilon}P_{+} w_n} \, dx  + \mathcal{O} (\frac{\norm{P_{\varepsilon} w_n}_{L_x^2}^2}{n}) .  \label{eq 10}
\end{align}
Using \eqref{eq Calderon} again, the second factor  inside the integral in \eqref{eq 10} can be written as 
\begin{align}
\partial_{x_1}a_2(x_1) \overline{ P_{\varepsilon}P_{+} w_n} = \overline{\partial_{x_1}a_2(x_1)  P_{\varepsilon}P_{+} w_n} = \overline{\partial_{x_1}  P_{\varepsilon}a_2(x_1)  P_{+} w_n} + \mathcal{O} (\frac{\norm{P_{\varepsilon} w_n}_{L_x^2}^2}{n}) \\
= \overline{\partial_{x_1}  P_{\varepsilon} P_{+} a_2(x_1)   w_n} + \mathcal{O} (\frac{\norm{P_{\varepsilon}P_{+} w_n}_{L_x^2}^2}{n}) = \overline{ P_{\varepsilon} P_{+} \partial_{x_1} (a_2(x_1)   w_n)} + \mathcal{O} (\frac{\norm{P_{\varepsilon} w_n}_{L_x^2}^2}{n}) .
\end{align}
Here the big O notation means that $\partial_{x_1}a_2(x_1) \overline{ P_{\varepsilon}P_{+} w_n} - \overline{\partial_{x_1}  P_{\varepsilon}a_2(x_1)  P_{+} w_n}$  as an operator acting on $w_n$ is
bounded in $L^2$ with norm $O(\frac{1}{n})$.

Therefore
\begin{align}
\eqref{eq 10} & = \im \int_{\R^d}  P_{\varepsilon}P_{+} \partial_{x_1} (a_2(x_1)w_n)  \cdot  \overline{ P_{\varepsilon} P_{+} \partial_{x_1} (a_2(x_1)   w_n)} \, dx + \mathcal{O} (\frac{\norm{P_{\varepsilon} w_n}_{L_x^2}^2}{n}) \\
& = \im \int_{\R^d}  \abs{P_{\varepsilon}P_{+} \partial_{x_1} (a_2(x_1)w_n)}^2 \, dx +  \mathcal{O} (\frac{\norm{P_{\varepsilon} w_n}_{L_x^2}^2}{n}) = \mathcal{O} (\frac{\norm{P_{\varepsilon} w_n}_{L_x^2}^2}{n}) +  \mathcal{O} (\frac{\norm{P_{\varepsilon}P_{+} \partial_{x_1} w_n}_{L_x^2}^2}{n}) .
\end{align}

Using the definition $P_{\varepsilon}$ and the support of its multiplier $\abs{\eta} \leq \frac{1}{\varepsilon}$, we have
\begin{align}\label{eq err1}
\norm{P_{\varepsilon}P_{+} \partial_{x_1} w_n}_{L_x^2} = \norm{\eta_{\varepsilon} \chi_+ \xi_1 \widehat{w_n}}_{L_{\xi}^2} \leq \frac{c}{\varepsilon} \norm{P_{\varepsilon}P_{+} w_n}_{L_x^2} \leq \frac{c}{\varepsilon} \norm{P_{\varepsilon} w_n}_{L_x^2}
\end{align}
As a consequence, we also have
\begin{align}\label{eq err2}
\norm{P_{\varepsilon}P_{\pm} \partial_{x_1}^{\, 2} w_n}_{L_x^2} \leq \frac{c}{\varepsilon^2} \norm{P_{\varepsilon}P_{\pm} w_n}_{L_x^2} \leq \frac{c}{\varepsilon^2} \norm{P_{\varepsilon} w_n}_{L_x^2}.
\end{align}

For the contribution of $q_2$ to \eqref{eq w2}, using \eqref{eq Calderon} and integration by parts, we have
\begin{align}
& - \im \int_{\R^d} P_{\varepsilon}P_{+} q_2(x_1) \partial_{x_1}^{\, 2} w_n  \cdot \overline{P_{\varepsilon}P_{+} w_n} \, dx \\
& = - \im \int_{\R^d} q_2(x_1) P_{\varepsilon}P_{+} \partial_{x_1}^{\, 2} w_n  \cdot \overline{P_{\varepsilon}P_{+} w_n} \, dx + \mathcal{O} (\frac{\norm{P_{\varepsilon}P_{+} \partial_{x_1} w_n}_{L_x^2}^2}{n}) \\
& = - \im \int_{\R^d} q_2(x_1) \partial_{x_1} P_{\varepsilon}P_{+} \partial_{x_1} w_n  \cdot \overline{P_{\varepsilon}P_{+} w_n} \, dx + \mathcal{O} (\frac{\norm{P_{\varepsilon} w_n}_{L_x^2}^2}{\varepsilon^2 n}) \\
& = \im \int_{\R^d} P_{\varepsilon}P_{+} \partial_{x_1} w_n  \cdot  \partial_{x_1} (q_2(x_1)  \overline{ P_{\varepsilon}P_{+} w_n}  ) \, dx + \mathcal{O} (\frac{\norm{P_{\varepsilon} w_n}_{L_x^2}^2}{\varepsilon^2 n}) \\
& = - \im \int_{\R^d} P_{\varepsilon}P_{+}  w_n  \cdot  \partial_{x_1}^{\, 2} (q_2(x_1)  \overline{ P_{\varepsilon}P_{+} w_n}  ) \, dx + \mathcal{O} (\frac{\norm{P_{\varepsilon} w_n}_{L_x^2}^2}{\varepsilon^2 n}) \\
& = -\im \int_{\R^d} P_{\varepsilon}P_{+} w_n  \cdot  (\partial_{x_1}^{\, 2} q_2 (x_1)) \overline{ P_{\varepsilon}P_{+} w_n} \, dx - 2\im \int_{\R^d} P_{\varepsilon}P_{+} w_n  \cdot  (\partial_{x_1} q_2 (x_1)) \overline{ P_{\varepsilon}P_{+}\partial_{x_1} w_n} \, dx  \\
& \quad - \im \int_{\R^d} P_{\varepsilon}P_{+} w_n  \cdot  q_2 (x_1) \overline{ P_{\varepsilon}P_{+} \partial_{x_1}^{\, 2} w_n}  \, dx + \mathcal{O} (\frac{\norm{P_{\varepsilon} w_n}_{L_x^2}^2}{\varepsilon^2 n})
\end{align}
where the first term in above is of the size 
\begin{align}
\im \int_{\R^d} P_{\varepsilon}P_{+} w_n  \cdot  (\partial_{x_1}^{\, 2} q_2 (x_1)) \overline{ P_{\varepsilon}P_{+} w_n} \, dx = \mathcal{O} (\frac{\norm{P_{\varepsilon} w_n}_{L_x^2}^2}{n})
\end{align}
and the second term in above 
\begin{align}
\im \int_{\R^d} P_{\varepsilon}P_{+} w_n  \cdot  (\partial_{x_1} q_2 (x_1)) \overline{ P_{\varepsilon}P_{+} \partial_{x_1} w_n} \, dx
\end{align}
can be absorbed by \eqref{eq Absorb} when $n$ is large enough.

Now, for the third term in above, notice that $\norm{q_2}_{L_x^\infty} \leq \frac{c}{n}$, and we have
\begin{align}
\abs{\im \int_{\R^d} P_{\varepsilon}P_{+} \partial_{x_1} w_n  \cdot  q_2 (x_1) \overline{ P_{\varepsilon}P_{+} \partial_{x_1} w_n}  \, dx } & \leq \norm{q_2}_{L_x^\infty}  \norm{P_{\varepsilon}P_{+}\partial_{x_1} w_n}_{L_x^2}^2 \leq \frac{c}{\varepsilon^2 n} \norm{P_{\varepsilon} w_n}_{L_x^2}^2.
\end{align}
We will choose $\varepsilon$ (depending on $n$) later.

Hence
\begin{align}
\int \eqref{eq w2} = \mathcal{O} (\frac{\norm{P_{\varepsilon} w_n}_{L_x^2}^2}{n}) + \mathcal{O} (\frac{\norm{P_{\varepsilon} w_n}_{L_x^2}^2}{\varepsilon^2 n} )
\end{align}
Similarly, we have
\begin{align}
\mathcal{O} (\frac{\norm{P_{\varepsilon} \partial_{x_1}^{\, 2} w_n}_{L_x^2}^2}{n} ) = \mathcal{O} (\frac{\norm{P_{\varepsilon} w_n}_{L_x^2}^2}{\varepsilon^4 n} ) .
\end{align}

\noindent {\it \underline{Term \eqref{eq w3}.}}
Take $a_3^2(x_1) \partial_{x_1}^{\, 3} w_n$ in \eqref{eq w3}, and write
\begin{align}\label{eq a_3}
\begin{aligned}
a_3^2(x_1) \partial_{x_1}^{\, 3} w_n & = a_3 \partial_{x_1}^{\, 3} (a_3 (x_1)w_n) - a_3 (\partial_{x_1}^{\, 3} a_3 (x_1)) w_n\\
& \quad - 3 a_3(x_1) (\partial_{x_1}^{\, 2} a_3 (x_1)) (\partial_{x_1} w_n)  - 3 a_3(x_1) (\partial_{x_1} a_3 (x_1)) (\partial_{x_1}^{\, 2} w_n ). 
\end{aligned}
\end{align}
Then we bound the contribution of the second term in \eqref{eq a_3} by 
\begin{align}
\abs{\im \int_{\R^d} P_{\varepsilon}P_{+} [a_3 (\partial_{x_1}^{\, 3} a_3 (x_1)) w_n]  \cdot \overline{P_{\varepsilon}P_{+} w_n} \, dx } = \mathcal{O} (\frac{\norm{P_{\varepsilon} w_n}_{L_x^2}^2}{n}) 
\end{align}
and have \eqref{eq Absorb} absorb the contribution of the third term in \eqref{eq a_3} 
\begin{align}
\im \int_{\R^d} P_{\varepsilon}P_{+} [a_3(x_1) (\partial_{x_1}^{\, 2} a_3 (x_1)) (\partial_{x_1} w_n) ] \cdot \overline{P_{\varepsilon}P_{+} w_n} \, dx .
\end{align}

Using the same calculation in \eqref{eq w2} and \eqref{eq err1}, we have the bound for the contribution of the fourth term in \eqref{eq a_3} 
\begin{align}
\abs{\im \int_{\R^d} P_{\varepsilon}P_{+} [a_3(x_1) (\partial_{x_1} a_3 (x_1)) (\partial_{x_1}^{\, 2} w_n )]  \cdot \overline{P_{\varepsilon}P_{+} w_n} \, dx } = \mathcal{O} (\frac{\norm{P_{\varepsilon} w_n}_{L_x^2}^2}{n}) + \mathcal{O} (\frac{\norm{P_{\varepsilon} w_n}_{L_x^2}^2}{\varepsilon^2 n} ).
\end{align}

Now we only need to control the contribution of the  first term in \eqref{eq a_3}, that is,
\begin{align}
& \quad \im \int_{\R^d} P_{\varepsilon}P_{+} a_3^2(x_1) \partial_{x_1}^{\, 3} w_n  \cdot \overline{P_{\varepsilon}P_{+} w_n} \, dx \\
& = \im \int_{\R^d} P_{\varepsilon}P_{+} [a_3(x_1) \partial_{x_1}^{\, 3} (a_3 (x_1)w_n)] \cdot \overline{P_{\varepsilon}P_{+} w_n} \, dx + \mathcal{O} (\frac{\norm{P_{\varepsilon} w_n}_{L_x^2}^2}{n}) + \mathcal{O} (\frac{\norm{P_{\varepsilon} w_n}_{L_x^2}^2}{\varepsilon^2 n} )\\
& = \im \int_{\R^d} a_3(x_1) P_{\varepsilon}P_{+} \partial_{x_1}^{\, 3} (a_3 (x_1)w_n) \cdot \overline{P_{\varepsilon}P_{+} w_n} \, dx + \mathcal{O} (\frac{\norm{P_{\varepsilon} w_n}_{L_x^2}^2}{n}) + \mathcal{O} (\frac{\norm{P_{\varepsilon} w_n}_{L_x^2}^2}{\varepsilon^2 n} ) + \mathcal{O} (\frac{\norm{P_{\varepsilon} \partial_{x_1}^{\, 2} w_n}_{L_x^2}^2}{n} )  \\
& = - \im \int_{\R^d}  P_{\varepsilon}P_{+} \partial_{x_1}^{\, 2} (a_3 (x_1)w_n) \cdot \partial_{x_1} a_3(x_1) \overline{P_{\varepsilon}P_{+} w_n} \, dx + \mathcal{O} (\frac{\norm{P_{\varepsilon} w_n}_{L_x^2}^2}{n}) + \mathcal{O} (\frac{\norm{P_{\varepsilon} w_n}_{L_x^2}^2}{\varepsilon^2 n} ) + \mathcal{O} (\frac{\norm{P_{\varepsilon} w_n}_{L_x^2}^2}{\varepsilon^4 n} )\\
& = - \im \int_{\R^d}  P_{\varepsilon}P_{+} \partial_{x_1}^{\, 2} (a_3 (x_1)w_n) \cdot  \overline{\partial_{x_1} P_{\varepsilon}P_{+} (a_3 w_n)} \, dx + \mathcal{O} (\frac{\norm{P_{\varepsilon} w_n}_{L_x^2}^2}{n}) + \mathcal{O} (\frac{\norm{P_{\varepsilon} w_n}_{L_x^2}^2}{\varepsilon^2 n} ) + \mathcal{O} (\frac{\norm{P_{\varepsilon} w_n}_{L_x^2}^2}{\varepsilon^4 n} )\\
& = - \im \int_{\R^d}  P_{\varepsilon}P_{+}  (a_3 (x_1)w_n) \cdot  \overline{\partial_{x_1}^{\, 3} P_{\varepsilon}P_{+} (a_3(x_1) w_n)} \, dx + \mathcal{O} (\frac{\norm{P_{\varepsilon} w_n}_{L_x^2}^2}{n}) + \mathcal{O} (\frac{\norm{P_{\varepsilon} w_n}_{L_x^2}^2}{\varepsilon^2 n} ) + \mathcal{O} (\frac{\norm{P_{\varepsilon} w_n}_{L_x^2}^2}{\varepsilon^4 n} )\\
& = \im \int_{\R^d}  \overline{P_{\varepsilon}P_{+}  (a_3 (x_1)w_n)} \cdot  \partial_{x_1}^{\, 3} P_{\varepsilon}P_{+} (a_3(x_1) w_n) \, dx + \mathcal{O} (\frac{\norm{P_{\varepsilon} w_n}_{L_x^2}^2}{n}) + \mathcal{O} (\frac{\norm{P_{\varepsilon} w_n}_{L_x^2}^2}{\varepsilon^2 n} ) + \mathcal{O} (\frac{\norm{P_{\varepsilon} w_n}_{L_x^2}^2}{\varepsilon^4 n} )\\
& = \im \int_{\R^d}  \overline{\widehat{P_{\varepsilon}P_{+}  (a_3 (x_1)w_n)}} \cdot  (i \xi_1)^3 \widehat{P_{\varepsilon}P_{+} (a_3(x_1) w_n)} \, d\xi + \mathcal{O} (\frac{\norm{P_{\varepsilon} w_n}_{L_x^2}^2}{n}) + \mathcal{O} (\frac{\norm{P_{\varepsilon} w_n}_{L_x^2}^2}{\varepsilon^2 n} ) + \mathcal{O} (\frac{\norm{P_{\varepsilon} w_n}_{L_x^2}^2}{\varepsilon^4 n} )\\
& = - \re \int_{\R^d} \xi_1^3 \abs{\widehat{P_{\varepsilon}P_{+} (a_3(x_1) w_n)}}^2 \, d\xi + \mathcal{O} (\frac{\norm{P_{\varepsilon} w_n}_{L_x^2}^2}{n}) + \mathcal{O} (\frac{\norm{P_{\varepsilon} w_n}_{L_x^2}^2}{\varepsilon^2 n} ) + \mathcal{O} (\frac{\norm{P_{\varepsilon} w_n}_{L_x^2}^2}{\varepsilon^4 n} ) .
\end{align}

\noindent {\it \underline{Term \eqref{eq w4}.}}
Similarly, we compute
\begin{align}
& \quad \int_{\R^d} P_{\varepsilon}P_{+} b(x_1) \partial_{x_1}^{\, 3} w_n  \cdot \overline{P_{\varepsilon}P_{+} w_n} \, dx \\
& = \int_{\R^d} b(x_1) P_{\varepsilon}P_{+}  \partial_{x_1}^{\, 3} w_n  \cdot \overline{P_{\varepsilon}P_{+} w_n} \, dx + \mathcal{O} (\frac{\norm{P_{\varepsilon} \partial_{x_1}^{\, 2} w_n}_{L_x^2}^2}{n}) \\
& = \int_{\R^d} b(x_1) \partial_{x_1}^{\, 3} P_{\varepsilon}P_{+} w_n  \cdot \overline{P_{\varepsilon}P_{+} w_n} \, dx + \mathcal{O} (\frac{\norm{P_{\varepsilon} w_n}_{L_x^2}^2}{\varepsilon^4 n}) \\
& = - \int_{\R^d} \partial_{x_1}^{\, 2} P_{\varepsilon}P_{+} w_n  \cdot \overline{ \partial_{x_1} b(x_1) P_{\varepsilon}P_{+} w_n} \, dx + \mathcal{O} (\frac{\norm{P_{\varepsilon} w_n}_{L_x^2}^2}{\varepsilon^4 n}) \\
& = - \int_{\R^d} \partial_{x_1}^{\, 2} P_{\varepsilon}P_{+} w_n  \cdot \overline{ \partial_{x_1} P_{\varepsilon}P_{+} ( b(x_1) w_n)} \, dx +  \mathcal{O} (\frac{\norm{P_{\varepsilon} w_n}_{L_x^2}^2}{n}) + \mathcal{O} (\frac{\norm{P_{\varepsilon} w_n}_{L_x^2}^2}{\varepsilon^4 n})  \\
& = - \int_{\R^d} P_{\varepsilon}P_{+} w_n  \cdot \overline{ \partial_{x_1}^{\, 3} P_{\varepsilon}P_{+} ( b(x_1) w_n)} \, dx +  \mathcal{O} (\frac{\norm{P_{\varepsilon} w_n}_{L_x^2}^2}{n}) + \mathcal{O} (\frac{\norm{P_{\varepsilon} w_n}_{L_x^2}^2}{\varepsilon^4 n})  \\
& = - \int_{\R^d} \overline{ \overline{P_{\varepsilon}P_{+} w_n}  \cdot P_{\varepsilon}P_{+}\partial_{x_1}^{\, 3}  ( b(x_1) w_n)} \, dx +  \mathcal{O} (\frac{\norm{P_{\varepsilon} w_n}_{L_x^2}^2}{n}) + \mathcal{O} (\frac{\norm{P_{\varepsilon} w_n}_{L_x^2}^2}{\varepsilon^4 n}) .
\end{align}
Then
\begin{align}
\re \int_{\R^d} P_{\varepsilon}P_{+} b(x_1) \partial_{x_1}^{\, 3} w_n  \cdot \overline{P_{\varepsilon}P_{+} w_n} \, dx =  \mathcal{O} (\frac{\norm{P_{\varepsilon} w_n}_{L_x^2}^2}{n}) + \mathcal{O} (\frac{\norm{P_{\varepsilon} w_n}_{L_x^2}^2}{\varepsilon^4 n}) .
\end{align}

Finally, \eqref{eq wn} becomes
\begin{align}
\partial_t \abs{P_{\varepsilon}P_{+} w_n}^2 & = \re \int_{\R^d} \xi_1 \abs{\widehat{P_{\varepsilon} P_{+} (\widetilde{a}_1 (x_1)w_n)}}^2 \, d\xi  - \re \int_{\R^d} \xi_1^3 \abs{\widehat{P_{\varepsilon}P_{+} (a_3(x_1) w_n)}}^2 \, d\xi \\
& \quad + c \norm{V}_{L_x^{\infty}} \norm{P_{\varepsilon} w_n}_{L_x^2}^2 + c \norm{e^{\mu} \Phi_n (x_1) H }_{L_x^2} \norm{P_{\varepsilon} w_n}_{L_x^2} \\
& \quad + \mathcal{O} (\frac{\norm{P_{\varepsilon} w_n}_{L_x^2}^2}{n}) + \mathcal{O} (\frac{\norm{P_{\varepsilon} w_n}_{L_x^2}^2}{\varepsilon^2 n}) + \mathcal{O} (\frac{\norm{P_{\varepsilon} w_n}_{L_x^2}^2}{\varepsilon^4 n}) ,
\end{align}
where $a_1^2 (x_1) = 4 (\beta \varphi_n^2)^3$, $\widetilde{a}_1^2 (x_1) =  4 (\beta \varphi_n^2)^3 + \mathcal{O}(\frac{1}{n})$ and $a_3^2 (x_1) = 4 \beta \varphi_n^2$.

Since $a_1^2 (x_1) $ and $\widetilde{a}_1^2 (x_1)$ are close enough, we consider the error term of the following form
\begin{align}
&  \quad \re \int_{\R^d} \xi_1 \abs{\widehat{P_{\varepsilon} P_{+} (a_1 (x_1)w_n)}}^2 \, d\xi  - \re \int_{\R^d} \xi_1^3 \abs{\widehat{P_{\varepsilon}P_{+} (a_3(x_1) w_n)}}^2 \, d\xi \\
& = 4 \re \int_{\R^d} \xi_1 \beta^3 \abs{\widehat{P_{\varepsilon} P_{+} (\varphi_n^3 w_n)}}^2 - \xi_1^3 \beta  \abs{\widehat{P_{\varepsilon}P_{+} (\varphi_n w_n)}}^2 \, d\xi \\
& = 4 \re \int_{\R^d} \xi_1 \beta^3 \abs{\widehat{P_{\varepsilon} P_{+} (\varphi_n^3 w_n)}}^2  - \xi_1 \beta^3  \abs{\widehat{P_{\varepsilon}P_{+} (\varphi_n w_n)}}^2   + \xi_1 \beta^3  \abs{\widehat{P_{\varepsilon}P_{+} (\varphi_n w_n)}}^2  - \xi_1^3 \beta  \abs{\widehat{P_{\varepsilon}P_{+} (\varphi_n w_n)}}^2  \, d\xi\\
& = 4 \re \int_{\R^d} \xi_1 \beta^3 (\abs{\widehat{P_{\varepsilon} P_{+} (\varphi_n^3 w_n)}}^2  -  \abs{\widehat{P_{\varepsilon}P_{+} (\varphi_n w_n)}}^2 ) \, d\xi  + 4 \re \int_{\R^d} (\xi_1 \beta^3 - \xi_1^3 \beta ) \abs{\widehat{P_{\varepsilon}P_{+} (\varphi_n w_n)}}^2 \, d\xi .
\end{align}

The second term above is negative under the definition of the projection $P_+$. 

Next, we focus on the error term 
\begin{align}
&  \quad \abs{\re \int_{\R^d} \xi_1 \beta^3 (\abs{\widehat{P_{\varepsilon} P_{+} (\varphi_n^3 w_n)}}^2  -  \abs{\widehat{P_{\varepsilon}P_{+} (\varphi_n w_n)}}^2 ) \, d\xi } \\
& \leq \abs{\re \int_{\R^d} \xi_1 \beta^3 (\abs{\widehat{P_{\varepsilon} P_{+} (\varphi_n^3 w_n)}}^2  -  \abs{\widehat{P_{\varepsilon}P_{+} w}}^2 ) \, d\xi} +  \abs{\re \int_{\R^d} \xi_1 \beta^3 (\abs{\widehat{P_{\varepsilon} P_{+} (\varphi_n w_n)}}^2  -  \abs{\widehat{P_{\varepsilon}P_{+} w}}^2 ) \, d\xi} ,
\end{align}
where $P_{\varepsilon} P_{+} \varphi_n^3 w_n \to  P_{\varepsilon} P_{+} w$, with $w =e^{-i\beta^4 t} e^{\beta x_1} u$ in $L_{\xi}^2$, hence $\abs{\xi_1}^{\frac{1}{2}} P_{\varepsilon} P_{+} w_n \to \abs{\xi_1}^{\frac{1}{2}}  P_{\varepsilon} P_{+} w$ in $L_{\xi}^2$. Then
\begin{align}
\int_{\R^d} \xi_1  \abs{\widehat{P_{\varepsilon} P_{+} (\varphi_n^3 w_n)}}^2 = \int_{\R^d} \xi_1  \abs{\widehat{P_{\varepsilon}P_{+} w}}^2 + o(1) \\
\int_{\R^d}  \xi_1 \abs{\widehat{P_{\varepsilon} P_{+} (\varphi_n w_n)}}^2  = \int_{\R^d} \xi_1   \abs{\widehat{P_{\varepsilon}P_{+} w}}^2 + o(1) 
\end{align}
and
\begin{align}
\re \int_{\R^d} \xi_1 \beta^3 (\abs{\widehat{P_{\varepsilon} P_{+} (\varphi_n^3 w_n)}}^2  -  \abs{\widehat{P_{\varepsilon}P_{+} (\varphi_n w_n)}}^2 ) \, d\xi  = o(1).
\end{align}
Therefore
\begin{align}
\re \int_{\R^d} \xi_1 \abs{\widehat{P_{\varepsilon} P_{+} (\widetilde{a}_1 (x_1)w_n)}}^2 \, d\xi  - \re \int_{\R^d} \xi_1^3 \abs{\widehat{P_{\varepsilon}P_{+} (a_3(x_1) w_n)}}^2 \, d\xi  = o(1).
\end{align}

Combining all the computations above, we have
\begin{align}
\partial_t \int_{\R^d} \abs{P_{\varepsilon} P_{+} w_n}^2 \, dx & \leq  c \norm{e^{\mu} \Phi_n (x_1) H }_{L_x^2} \norm{P_{\varepsilon} w_n}_{L_x^2} + c \norm{V }_{L_x^{\infty}} \norm{P_{\varepsilon} w_n}_{L_x^2}^2 \\
& \quad + \mathcal{O} (\frac{\norm{P_{\varepsilon} w_n}_{L_x^2}^2}{n}) + \mathcal{O} (\frac{\norm{P_{\varepsilon} w_n}_{L_x^2}^2}{\varepsilon^2 n}) + \mathcal{O} (\frac{\norm{P_{\varepsilon} w_n}_{L_x^2}^2}{\varepsilon^4 n})  + o(1) \\
& \leq c \norm{e^{\mu} \Phi_n (x_1) H }_{L_x^2} \norm{P_{\varepsilon} w_n}_{L_x^2} + c \norm{V }_{L_x^{\infty}} \norm{P_{\varepsilon} w_n}_{L_x^2}^2 + \mathcal{O} (\frac{\norm{P_{\varepsilon} w_n}_{L_x^2}^2}{\varepsilon^4 n})  + o(1) .  \label{eq P+}
\end{align}

Arguing similarly for $P_-$, we obtain
\begin{align}
\partial_t \int_{\R^d} \abs{P_{\varepsilon} P_{-} w_n}^2 \, dx & \geq  - c \norm{e^{\mu} \Phi_n (x_1) H }_{L_x^2} \norm{P_{\varepsilon} w_n}_{L_x^2} - c \norm{V }_{L_x^{\infty}} \norm{P_{\varepsilon} w_n}_{L_x^2}^2 \\
& \quad + \mathcal{O} (\frac{\norm{P_{\varepsilon} w_n}_{L_x^2}^2}{n}) + \mathcal{O} (\frac{\norm{P_{\varepsilon} w_n}_{L_x^2}^2}{\varepsilon^2 n}) + \mathcal{O} (\frac{\norm{P_{\varepsilon} w_n}_{L_x^2}^2}{\varepsilon^4 n})  + o(1) \\
& \geq  - c \norm{e^{\mu} \Phi_n (x_1) H }_{L_x^2} \norm{P_{\varepsilon} w_n}_{L_x^2} - c \norm{V }_{L_x^{\infty}} \norm{P_{\varepsilon} w_n}_{L_x^2}^2 + \mathcal{O} (\frac{\norm{P_{\varepsilon} w_n}_{L_x^2}^2}{\varepsilon^4 n})  + o(1) . \label{eq P-}
\end{align}

\noindent {\bf Step 4: Estimating $L^2$ norm of $w_n$.}

By the definition of the supremum, there exists $t_n \in [0,1]$ such that
\begin{align}\label{eq sup}
\norm{P_{\varepsilon}w_n (t_n, \cdot)}_{L_x^2}^2 \geq \frac{1}{2} \sup_{t \in [0,1]} \norm{P_{\varepsilon}w_n (t, \cdot)}_{L_x^2}^2 ,
\end{align}
when we choose $\varepsilon = n^{-\frac{1}{10}}$.

Now the fundamental theorem of calculus in time $t$ applying on \eqref{eq P+} and \eqref{eq P-} yields  
\begin{align}
\norm{P_{\varepsilon}P_{+} w_n (1, \cdot)}_{L_x^2}^2 & - \norm{P_{\varepsilon}P_{+} w_n (t_n, \cdot)}_{L_x^2}^2  \\
& \geq  - \int_{t_n}^1 \norm{V}_{L_x^{\infty}} \norm{P_{\varepsilon} w_n}_{L_x^2}^2 - \int_{t_n}^1  \norm{e^{\mu}\Phi_n H }_{L_x^2} \norm{P_{\varepsilon} w_n}_{L_x^2} \,  dt + \mathcal{O} (\frac{\norm{P_{\varepsilon} w_n}_{L_x^2}^2}{\varepsilon^4 n})   + o(1)\\
\norm{P_{\varepsilon}P_{-} w_n (t_n, \cdot)}_{L_x^2}^2 & - \norm{P_{\varepsilon}P_{-} w_n (0, \cdot)}_{L_x^2}^2 \\
& \leq \int_0^{t_n} \norm{V}_{L_x^{\infty}} \norm{P_{\varepsilon} w_n}_{L_x^2}^2 + \int_0^{t_n}  \norm{e^{\mu} \Phi_n H }_{L_x^2} \norm{P_{\varepsilon} w_n}_{L_x^2} \,  dt + \mathcal{O} (\frac{\norm{P_{\varepsilon} w_n}_{L_x^2}^2}{\varepsilon^4 n}) + o(1) .
\end{align}

Then
\begin{align}
\norm{P_{\varepsilon}w_n (t_n, \cdot)}_{L_x^2}^2 & = \norm{P_{\varepsilon}P_{+} w_n (t_n, \cdot)}_{L_x^2}^2 + \norm{P_{\varepsilon}P_{-} w_n (t_n, \cdot)}_{L_x^2}^2 \\
& \leq  \norm{P_{\varepsilon}P_{+} w_n (1, \cdot)}_{L_x^2}^2  + \int_{t_n}^1 \norm{V}_{L_x^{\infty}} \norm{P_{\varepsilon} w_n}_{L_x^2}^2 \, dt + \int_{t_n}^1 \norm{e^{\mu} \Phi_n H }_{L_x^2} \norm{P_{\varepsilon} w_n}_{L_x^2} \, dt\\
& \quad + \norm{P_{\varepsilon}P_{-} w_n (0, \cdot)}_{L_x^2}^2  + \int_0^{t_n} \norm{V}_{L_x^{\infty}} \norm{P_{\varepsilon} w_n}_{L_x^2}^2 \, dt + \int_0^{t_n} \norm{e^{\mu} \Phi_n H }_{L_x^2} \norm{P_{\varepsilon} w_n}_{L_x^2} \, dt \\
& \quad + \mathcal{O} (\frac{\norm{P_{\varepsilon} w_n}_{L_x^2}^2}{\varepsilon^4 n})  + o(1) \\
& \leq  \norm{P_{\varepsilon} w_n (1, \cdot)}_{L_x^2}^2 + \norm{P_{\varepsilon} w_n (0, \cdot)}_{L_x^2}^2 + \int_{0}^1 \norm{V}_{L_x^{\infty}} \norm{P_{\varepsilon} w_n}_{L_x^2}^2 \, dt + \int_{0}^1 \norm{e^{\mu} \Phi_n H }_{L_x^2} \norm{P_{\varepsilon} w_n}_{L_x^2} \, dt  \\ 
& \quad +  \frac{c}{\e^4 n} \norm{P_\e w_n}_{L_x^2}^2 + o(1) . \label{eq 23}
\end{align}

Since we chose $\varepsilon = n^{-\frac{1}{10}}$,  for $n$ large enough, we will have
\begin{align}
1 - \frac{c}{\e^4 n} = 1 - \frac{c}{n^{\frac{3}{5}}} > \frac{1}{2} .
\end{align}

Then let the left-hand side of \eqref{eq 23} absorb the $\frac{c}{\e^4 n} \norm{P_\e w_n}_{L_x^2}^2$ term, and we write
\begin{align}
\norm{P_{\varepsilon}w_n (t_n, \cdot)}_{L_x^2}^2 & \leq 2 \norm{P_{\varepsilon} w_n (1, \cdot)}_{L_x^2}^2 + 2 \norm{P_{\varepsilon} w_n (0, \cdot)}_{L_x^2}^2 + 2 \int_{0}^1 \norm{V}_{L_x^{\infty}} \norm{P_{\varepsilon} w_n}_{L_x^2}^2 \, dt +  2 \int_{0}^1 \norm{e^{\mu} \Phi_n H }_{L_x^2} \norm{P_{\varepsilon} w_n}_{L_x^2} \, dt  \\
& \leq \norm{P_{\varepsilon}w_n (t_n, \cdot)}_{L_x^2}^2 \leq 2 \norm{P_{\varepsilon} w_n (1, \cdot)}_{L_x^2}^2 + 2 \norm{P_{\varepsilon} w_n (0, \cdot)}_{L_x^2}^2 \\
& \quad + 2 \sup_{t \in [0,1]} \norm{P_{\varepsilon} w_n}_{L_x^2}^2 \int_{0}^1 \norm{V}_{L_x^{\infty}} \, dt + \frac{1}{100} \sup_{t \in [0,1]} \norm{P_{\varepsilon} w_n}_{L_x^2}^2 + 100 (\int_{0}^1 \norm{e^{\mu} \Phi_n H }_{L_x^2} \, dt )^2 .
\end{align}
By choosing $\varepsilon_0$ in \eqref{eq V} such that
\begin{align}
\int_0^1 \norm{V}_{L_x^{\infty}} \, dt < \frac{1}{100}
\end{align}
we have
\begin{align}
\norm{P_{\varepsilon}w_n (t_n, \cdot)}_{L_x^2}^2 &  \leq 2 \norm{P_{\varepsilon} w_n (1, \cdot)}_{L_x^2}^2 + 2 \norm{P_{\varepsilon} w_n (0, \cdot)}_{L_x^2}^2 \\
& \quad + \frac{1}{50} \sup_{t \in [0,1]} \norm{P_{\varepsilon} w_n}_{L_x^2}^2  + \frac{1}{100} \sup_{t \in [0,1]} \norm{P_{\varepsilon} w_n}_{L_x^2}^2 + 100 (\int_{0}^1 \norm{e^{\mu} \Phi_n H }_{L_x^2} \, dt )^2 \\
& \leq 2 \norm{P_{\varepsilon} w_n (1, \cdot)}_{L_x^2}^2 + 2 \norm{P_{\varepsilon} w_n (0, \cdot)}_{L_x^2}^2  + \frac{3}{100} \sup_{t \in [0,1]} \norm{P_{\varepsilon} w_n}_{L_x^2}^2 + 100 (\int_{0}^1 \norm{e^{\mu} \Phi_n H }_{L_x^2} \, dt )^2 .
\end{align}

Using \eqref{eq sup}, we write
\begin{align}
\frac{1}{2}\sup_{t \in [0,1]} \norm{P_{\varepsilon}w_n (t, \cdot)}_{L_x^2}^2 & \leq \norm{P_{\varepsilon}w_n (t_n, \cdot)}_{L_x^2}^2 \\
& \leq 2 \norm{P_{\varepsilon} w_n (1, \cdot)}_{L_x^2}^2 + 2 \norm{P_{\varepsilon} w_n (0, \cdot)}_{L_x^2}^2  + \frac{3}{100} \sup_{t \in [0,1]} \norm{P_{\varepsilon} w_n}_{L_x^2}^2 + 100 (\int_{0}^1 \norm{e^{\mu} \Phi_n H }_{L_x^2} \, dt )^2 
\end{align}
then
\begin{align}
\sup_{t \in [0,1]} \norm{P_{\varepsilon}w_n (t, \cdot)}_{L_x^2}^2 & \leq 10 \norm{P_{\varepsilon} w_n (1, \cdot)}_{L_x^2}^2 + 10 \norm{P_{\varepsilon} w_n (0, \cdot)}_{L_x^2}^2 + 1000 (\int_{0}^1 \norm{e^{\mu} \Phi_n H }_{L_x^2} \, dt )^2 .
\end{align}

Notice that we chose $\varepsilon = n^{-\frac{1}{10}}$ and now we only have one limit in $n$ to take. 

Letting $n \to \infty$, we obtain the desired inequality,
\begin{align}
\sup_{t \in [0,1]} \norm{u (t)}_{L_x^2(e^{2\beta x_1} \, dx)}^2 \lesssim \norm{u_1}_{L_x^2(e^{2\beta x_1} \, dx)}^2 + \norm{u_0}_{L_x^2(e^{\beta x_1} \, dx)}^2 +  \norm{ H }_{L_t^1 L_x^2(e^{2\beta x_1} \, dx)}^2  .
\end{align}
Now we complete the proof of Lemma \ref{lem con}.
\end{proof}

\section{Upgraded Exponential Decay Estimate}\label{sec Logcon+}
In this section, we prove an interior estimate for the rapidly decaying solutions, and upgrade it to a super-linear exponential decrease estimate.

\subsection{Linear exponential decay estimate}
\begin{lem}[Linear exponential decay estimate in all directions]\label{lem allcon}
If in addition to the hypothesis in Lemma \ref{lem con} one has that for some $\beta >0$
\begin{align}
u_0, u_1 \in L_x^2 (e^{2\beta \abs{x}} \, dx)
\end{align}
and $H \in L_t^1([0,1] : L_x^2 (e^{\beta \abs{x}} \, dx))$, then
\begin{align}
\sup_{t \in [0,1]} \norm{u(t)}_{L^2 (e^{2\beta \abs{x} / d} \, dx)}^2 \leq C(d) ( \norm{u_0}_{L^2 (e^{2\beta \abs{x} }\, dx)}^2 + \norm{u_1}_{L^2 (e^{2\beta \abs{x}}  \, dx)}^2 + \norm{H}_{L_t^1 L_x^2  (e^{2\beta \abs{x}}  \, dx)}^2 )
\end{align}
with $C(d)$ independent of $\beta >0$. 
\end{lem}
\begin{proof}[Proof of Lemma \ref{lem allcon}]
By Lemma \ref{lem con}, we have that for any $\beta>0$, 
\begin{align}
\sup_{t \in [0,1]} \norm{u(t)}_{L^2 (e^{\pm 2\beta x_j} \, dx)}^2 \leq C ( \norm{u_0}_{L^2 (e^{2\beta |x|} \, dx)}^2 + \norm{u_1}_{L^2 (e^{2\beta |x|} \, dx)}^2 + \norm{H}_{L_t^1 L_x^2  (e^{2\beta |x|} \, dx)}^2 ) := \Phi, 
\end{align}
for any $j=1, \ldots, d$. Hence, for any $t\in [0,1]$, 
\begin{align}
\norm{u}_{L^2(e^{2\beta \abs{x}/d} dx)}^2 & = \int_{\R^d} \abs{u}^2 e^{2\beta \abs{x} / d} \, dx  \leq \int_{\R^d} \abs{u}^2 e^{2\beta \sum_j |x_j| / d} \, dx \\
& =  \int_{\R^d} \prod_j \parenthese{\abs{u}^{\frac{2}{d}} e^{2 \beta |x_j| / d} } \, dx \leq \prod_j \parenthese{\int_{\R^d} \abs{u}^2 e^{2 \beta |x_j| } \, dx }^{\frac{1}{d}} \\
& \leq \sum_j \int_{\R^d} \abs{u}^2 e^{2 \beta |x_j| } \, dx = \sum_j \norm{u }_{e^{2 \beta |x_j|} \, dx}^2 \leq C(d) \Phi\\
& = C(d) (\norm{u_0}_{L^2 (e^{2\beta |x|} \, dx)}^2 + \norm{u_1}_{L^2 (e^{2\beta |x|} \, dx)}^2 +  \norm{H}_{L_t^1 L_x^2  (e^{2\beta |x|} \, dx)}^2 )
\end{align}
This completes the proof of Lemma \ref{lem allcon}.
\end{proof}

\subsection{Super-linear exponential decrease}
\begin{lem}[Super-linear exponential decay estimate]\label{lem Super-logcon}
In addition to the hypotheses of Lemma \ref{lem con}, we assume that for some $\lambda>0$, and $\alpha>1$, $u_0, u_1 \in L^2(e^{\lambda|x|^\alpha}\,dx)$. Additionally, let $u \in C^1([0,1] : H^3(\mathbb{R}^d))$, then, there exists $c_\alpha>0$ such that
\begin{align*}
    \sup_{t \in [0,1]} \int_{|x|\ge c_\alpha} |u(t,x)|^2 e^{\lambda|x|^{\alpha}/(10 d)^{\alpha}} \,dx & \lesssim \norm{u_0}_{L^2(e^{\lambda|x|^\alpha}\,dx)}^2 +\norm{u_1}_{L^2(e^{\lambda|x|^\alpha}\,dx)}^2 \\
    & \quad + \int_0^1 \int_{\R^d} \abs{H(t,x)}^2 e^{\lambda \abs{x}^{\alpha}} \, dx \, dt+ \sum_{j=1}^d \sum_{l=0}^3 \int_0^1 \int_{\R^d} |\partial_{x_j}^l u(t,x)|^2\,dx \, dt.
\end{align*}
\end{lem}
We note that the factor $10$ on the left-hand side of the inequality is not essential, and it suffices for it to be slightly greater than $2$.

\begin{proof}[Proof of Lemma \ref{lem Super-logcon}]

Let $\eta (x) \in C^\infty$ be  non-decreasing, radial and such that 
\begin{align}
\eta (x) = 
\begin{cases}
0, & \text{if } \abs{x} \leq 1, \\
1, & \text{if } \abs{x} \geq 2. 
\end{cases}
\end{align}
We also define $\eta_R(x) = \eta (\frac{x}{R})$.

Let $u_R (t,x) = \eta_R (x) u(t,x) $, where $\eta_R(x)$ defined above.
Then using \eqref{eq ChainRule2}, we have 
\begin{align}
(i \partial_t + \D^2) u_R = V u_R + \widetilde{H}_R ,
\end{align}
where
\begin{align}
\widetilde{H}_R = \eta_R H  - 4 \sum_{j=1}^d (\partial_{x_j} \eta_R )\partial_{x_j}^{3} u - 6 \sum_{j=1}^d (\partial_{x_j}^{\, 2} \eta_R) \partial_{x_j}^{\, 2} u - 4 \sum_{j=1}^d (\partial_{x_j}^{\, 3} \eta_R )\partial_{x_j} u - \sum_{j=1}^d (\partial_{x_j}^{\, 4} \eta_R ) u .
\end{align}

We now use Lemma \ref{lem con} to conclude that
\begin{align}\label{eq 1}
\norm{u_R(t)}_{L^2(e^{2\beta \abs{x} /d} \, dx )}^2 \lesssim \norm{u_R(0)}_{L^2(e^{2\beta \abs{x}} \, dx) }^2 + \norm{u_R(1)}_{L^2 (e^{2\beta \abs{x}} \, dx)}^2 + \norm{\widetilde{H}_R}_{L_t^1 L_x^2 (e^{2\beta \abs{x}} \, dx)}^2.
\end{align}

Hence, using the definition of $\eta_R$ and \eqref{eq 1}, we have 
\begin{align}
& \quad \int_{\abs{x} > 2R} \abs{u (t,x)}^2 e^{2 \beta \abs{x}/d} \, dx  \leq \int_{\R^d} \abs{u_R (t,x)}^2 e^{2 \beta \abs{x}/d} \, dx \\
& \lesssim  \sum_{k=0}^1 \int_{\abs{x} > R} \abs{u_{k,R}}^2 e^{2 \beta \abs{x}} \, dx + \int_0^1 \int_{\abs{x} > R} \abs{H(t,x)}^2 e^{2\beta \abs{x}} \, dx \, dt \\
& \quad + \sum_{j=1}^d \int_{R < \abs{x} < 2 R} e^{2\beta \abs{x}} (\abs{u}^2 R^{-8} + \abs{\partial_{x_j} u}^2 R^{-6} + \abs{\partial_{x_j}^{\, 2} u}^2 R^{-4} + \abs{\partial_{x_j}^{\, 3} u}^2 R^{-2}) \, dx \\
&\lesssim \sum_{k=0}^1 \int_{\abs{x} > R} \abs{u_{k,R}}^2 e^{2 \beta \abs{x}} \, dx + \int_0^1 \int_{\abs{x} > R} \abs{H(t,x)}^2 e^{2\beta \abs{x}} \, dx \, dt \\
& \quad +  \sum_{j=1}^d e^{ 4 \beta R} \int_{R < \abs{x} < 2 R} (\abs{u}^2 R^{-8} + \abs{\partial_{x_j} u}^2 R^{-6} + \abs{\partial_{x_j}^{\, 2} u}^2 R^{-4} + \abs{\partial_{x_j}^{\, 3} u}^2 R^{-2}) \, dx 
\end{align}
where $u_{0, R}=u_R(0, x)$, and $u_{1, R}=u_{R}(1,x)$.
By multiplying the inequality above by $e^{-4\beta R}$, we write 
\begin{align}\label{eq 4}
A & :=e^{-4\beta R}\int_{\abs{x} > 2R} \abs{u (t,x)}^2 e^{2 \beta \abs{x}/d} \, dx \\
& \lesssim e^{-4\beta R} \sum_{k=0}^1 \int_{\abs{x} > R} \abs{u_{k,R}}^2 e^{2 \beta \abs{x}} \, dx + e^{-4\beta R} \int_0^1 \int_{\abs{x} >R}  \abs{H(t,x)}^2  e^{2 \beta \abs{x}} \, dx \, dt \\
& \quad + \sum_{j=1}^d \int_{R < \abs{x} < 2 R} (\abs{u}^2 R^{-8} + \abs{\partial_{x_j} u}^2 R^{-6} + \abs{\partial_{x_j}^{\, 2} u}^2 R^{-4} + \abs{\partial_{x_j}^{\, 3} u}^2 R^{-2}) \, dx \\
& =: D_1+D_2 + D_3 . \label{eq 12}
\end{align}

We fix $4\beta R=2bR^{\alpha}$. 
Integrating the inequality \eqref{eq 4} in $R$ in the interval $[0,\infty )$, and consider the resulting terms separately. Using Fubini's theorem, the $D_1$ term can be written as
\begin{align}
\int_0^{\infty} e^{-2b R^{\alpha}} \sum_{k=0}^1 \int_{\abs{x} > R} \abs{u_j (x)}^2 e^{2 \beta \abs{x}} \, dx \, dR = \sum_{j=0}^1 \int_{\abs{x} >1} \parenthese{\int_1 ^r e^{-2b R^{\alpha} + b R^{\alpha -1} r} \, dR} \abs{u_j (x)}^2 \, dx ,
\end{align}
where $r = \abs{x}$.

To deduce an upper bound for this expression we see that $\varphi (R)  = b R^{\alpha -1} (r -2R)$ has  its maximum at 
\begin{align}
R_M = \frac{(\alpha -1) r}{2 \alpha } < \frac{r}{2} ,
\end{align}
hence
\begin{align}
\int_1^r e^{-2bR^{\alpha} + b R^{\alpha -1} r} \, dR \leq r e^{\varphi (R_M)} = r e^{b (\alpha -1)^{\alpha -1} r^{\alpha} / (2^{\alpha -1} \alpha^{\alpha})} = r e^{b_{\alpha} r^{\alpha}} = \abs{x} e^{b_{\alpha} \abs{x}^{\alpha}} ,
\end{align}
that is 
\begin{align}
b_{\alpha} = b (\alpha -1)^{\alpha -1} / (2^{\alpha -1} \alpha^{\alpha}) .
\end{align}

This estimate yields the bound
\begin{align}
\sum_{j=0}^1 \int_{\R^d} \abs{u_j (x)}^2 e^{b_{\alpha} \abs{x}^{\alpha}} \abs{x} \, dx .
\end{align}
A similar argument provides the following upper bound for the term coming from $D_2$ in \eqref{eq 12}
\begin{align}
\int_0^1 \int_{\R^d} \abs{H(t,x)}^2 e^{b_{\alpha} \abs{x}^{\alpha}} \abs{x} \, dx  \, dt.
\end{align}

Next, we shall deduce a lower bound for the term arising from $A$. Using again Fubini’s theorem this can be written as
\begin{align}
\int_1^{\infty} e^{-2b R^{\alpha}} \int_{\abs{x} > 2R} \abs{u (t,x)}^2 e^{2 \beta \abs{x} / d} \, dx \, dR = 2 \int_{2}^{\infty} \int_{\mathcal{S}^{d-1}}  \parenthese{\int_1^{\frac{r}{2}} e^{-2bR^{\alpha} + b R^{\alpha -1}r/d} \, dR } \abs{u(t,x)}^2  r^{d-1} \, dS dr .
\end{align}
Since $\eta (R) = -2b R^{\alpha} + b R^{\alpha -1} r/ d = b R^{\alpha-1} (r/d - 2R)$ has its maximum at 
\begin{align}
\widetilde{R}_M = \frac{(\alpha -1) r}{2 \alpha d} <  \frac{r}{2d} \leq \frac{r}{2} ,
\end{align}
we take $R_0 = \frac{(\alpha -1) r}{ 10 \alpha d}$, $r > c_{\alpha}$ and $c_{\alpha} > \frac{10 \alpha d}{\alpha -1} > 2$ to bound from below the integral as
\begin{align}
\int_1^{\frac{r}{2}} e^{-2b R^{\alpha} + b R^{\alpha -1}r/d} \, dR & \geq \int_{R_0}^{\widetilde{R}_M} e^{b R^{\alpha -1} (r/d - 2R)} \, dR \\
& \geq e^{b R_0^{\alpha -1} (r/d - 2 R_0)} (\widetilde{R}_M - R_0) \geq e^{b R_0^{\alpha -1} (r/d - 2 \widetilde{R}_M)} (\widetilde{R}_M - R_0) \\
& \geq \frac{2}{5} \frac{\alpha -1}{\alpha} \frac{r}{d} e^{b(\alpha -1)^{\alpha -1} r^{\alpha} / 10^{\alpha -1} \alpha^{\alpha} d^{\alpha}} = \frac{2}{5} \frac{\alpha -1}{\alpha} \frac{r}{d} e^{b_{\alpha} r^{\alpha} / (5^{\alpha -1} d^{\alpha}) } .
\end{align}
This last expression is thus a lower bound for the exponential part of the integrand on the left hand side of \eqref{eq 4}.

Fixing $b$ such that $b_{\alpha} + \varepsilon = b(\alpha -1)^{\alpha -1} / (2^{\alpha -1} \alpha^{\alpha}) + \varepsilon = \lambda$ with $\varepsilon >0$ small enough.

To bound the $D_3$ term, we compute for $R \gg 1$
\begin{align}
& \quad \sum_{j=1}^d \int_1^{\infty} \int_0^1 \int_{R < \abs{x} < 2 R} (\abs{u}^2 R^{-8} + \abs{\partial_{x_j} u}^2 R^{-6} + \abs{\partial_{x_j}^{\, 2} u}^2 R^{-4} + \abs{\partial_{x_j}^{\, 3} u}^2 R^{-2}) \, dx \, dt \, dR\\
& \leq \sum_{j=1}^d \int_1^{\infty} \int_0^1 \int_{R < \abs{x} < 2R} (\abs{u}^2 R^{-2} + \abs{\partial_{x_j} u}^2 R^{-2} + \abs{\partial_{x_j}^{\, 2} u}^2 R^{-2} + \abs{\partial_{x_j}^{\,3} u}^2 R^{-2}) \, dx \, dt \, dR\\
& = \sum_{j=1}^d \int_0^1 \int_{\abs{x} \geq 1} (\abs{u}^2 + \abs{\partial_{x_j} u}^2 + \abs{\partial_{x_j}^{\, 2} u}^2 + \abs{\partial_{x_j}^{\, 3} u}^2) \parenthese{ \int_{\frac{\abs{x}}{2}}^{\abs{x}} \frac{dR}{R^2} \, dR } \, dx \, dt\\
& \leq \sum_{j=1}^d \int_0^1 \int_{\mathbb{R}} (\abs{u}^2 + \abs{\partial_{x_j} u}^2 + \abs{\partial_{x_j}^{\, 2} u}^2 + \abs{\partial_{x_j}^{\, 3} u}^2 ) \, dx \, dt.
\end{align}
Hence, we obtain that
\begin{align}
\sup_{t \in [0,1]} \int_{|x|\ge c_\alpha} |u(t,x)|^2 e^{\lambda|x|^{\alpha}/(10 d)^{\alpha}} \,dx & \lesssim \norm{u_0}_{L^2(e^{\lambda|x|^\alpha}\,dx)}^2 +\norm{u_1}_{L^2(e^{\lambda|x|^\alpha}\,dx)}^2 \\
& \quad + \int_0^1 \int_{\R^d} \abs{H(t,x)}^2 e^{a \abs{x}^{\alpha}} \, dx \, dt+ \sum_{j=1}^d \sum_{l=0}^3 \int_0^1 \int_{\R^d} |\partial_{x_j}^l u(t,x)|^2\,dx \, dt.
\end{align}
This completes the proof of Lemma \ref{lem Super-logcon}. 
\end{proof}

\begin{rmk}
The results in this section extend to equations of the form
\begin{align}
i \partial_t u + \Delta u = V_1 u+ V_2 \overline{u} + H
\end{align}
with the potentials $V_j$, where $j=1,2$ satisfying the assumption \eqref{eq V}.
\end{rmk}

\subsection{Logarithmic convexity generalized from \cite{HHZ}}\label{ssec Logcon}

A key ingredient in \cite{HHZ} is the following logarithmic convexity in any dimensions.
\begin{lem}[Proposition 1.3 in \cite{HHZ}]\label{lem HHZ}
Suppose $V \in L^{\infty}(\R^d)$ is real-valued, and $u \in C(\R ; L^2 (\R^d))$ solves
\begin{align}
i \partial_t u - (-\Delta_x)^m u = V(x)u .
\end{align}
If there exists $\gamma >0$ such that
\begin{align}
e^{\gamma \abs{x}^{\frac{2m}{2m-1}}} u(0,x) , \quad e^{\gamma \abs{x}^{\frac{2m}{2m-1}}} u(1,x) \in L_x^2 (\R^d),
\end{align}
then for $t \in [0,1]$, we have
\begin{align}
\norm{e^{\gamma \abs{x}^{\frac{2m}{2m-1}}} u(t,x)}_{L_x^2} \leq C e^{\frac{t(1-t)}{4} \norm{V}_{L^{\infty}}^2} \norm{e^{\gamma \abs{x}^{\frac{2m}{2m-1}}} u(0,x)}_{L_x^2}^{1-t} \norm{e^{\gamma \abs{x}^{\frac{2m}{2m-1}}} u(1,x)}_{L_x^2}^{t} .
\end{align}
\end{lem}

\begin{rmk}[Comparison between Lemma \ref{lem HHZ} (which is Proposition 1.3 in \cite{HHZ}) and Lemma \ref{lem con}] 
\begin{enumerate}
\item 
In \cite{HHZ}, the authors established a logarithmic convexity for higher order Schr\"odinger operators in any dimensions, considering real, bounded, and time-independent potentials. Their proof relied on two key ingredients: (i) an estimate for higher-order heat kernels, which can be extended to our `separable' fourth-order Schr\"odinger operator (ii) a formal commutator estimate, as presented in Lemma 2 of \cite{EKPV_JEMS}, which remain valid in our case as well.

\item 
To obtain a nonlinear unique continuation result, it becomes necessary to allow the potential $V$ to take complex values and be time-dependent, which is missing in Lemma \ref{lem HHZ}. In our Lemma \ref{lem con}, the energy estimate method, inspired by \cite{KPV_CPAM}, allows complex-valued and time-dependent potentials, thereby enabling us to obtain the unique continuation for nonlinear equations. 

\item

In Lemma \ref{lem con}, the $H^3$ regularity of the solution is initially required due to the need for taking multiple derivatives. However, Lemma \ref{lem HHZ} only requires $L^2$ regularity. This is done by introducing   an artificial diffusion term into the differential equation, as was first proposed in \cite{EKPV_JEMS}. We believe that this $H^3$ regularity can be relaxed to $L^2$ with artificial diffusion. 

\end{enumerate}

\end{rmk}

\section{A Carleman Inequality}\label{sec Carleman}

In this section, we prove a Carleman estimate with a quadratic exponential weight for the `separable' equation \eqref{eq 4SE}, which will be used in Section \ref{sec Proof}.

\begin{lem}\label{lem Carleman}
Assume that $R>0$ and $\varphi:[0,1] \to \mathbb{R}$ is a smooth function. Let $u (t,x) \in C_c^\infty(\mathbb{R}\times \R^d)$ with support contained in the set
\begin{align}
\{(t,x) \in [0,1] \times \R^d  :  \abs{\frac{x_1}{R}+ \varphi (t))} \geq 1 \} .
\end{align}
Then there exists $c= c(d, \norm{\varphi'}_{L^{\infty}} , \norm{\varphi''}_{L^{\infty}})$ such that the inequality
\begin{align}
\norm{e^{\alpha (\frac{x_1}{R}+ \varphi (t))^2 + \alpha \sum_{j=2}^d (\frac{x_j}{R})^2} (i\partial_t + \D^2) u}_{L_{t,x}^2}^2  \geq c \frac{\alpha^7}{R^8}  \norm{e^{\alpha (\frac{x_1}{R}+ \varphi (t))^2 + \alpha \sum_{j=2}^d (\frac{x_j}{R})^2} u}_{L_{t,x}^2}^2
\end{align}
holds when $\alpha \geq c R^{\frac{4}{3}}$.
\end{lem}

\begin{proof}[Proof of Lemma \ref{lem Carleman}]
Let $f = e^{\Phi^2} u$, where 
\begin{align}
\Phi^2  (t,x) :=  (\frac{x_1}{R}+ \varphi (t))^2 +  \sum_{j=2}^d (\frac{x_j}{R})^2 =  \psi^2 +  \sum_{j=2}^d (\frac{x_j}{R})^2,
\end{align}
and
\begin{align}
\psi = \frac{x_1}{R}+ \varphi (t) . 
\end{align}

Under this change of variables, we reduce to proving 
\begin{align}
\norm{e^{\alpha\Phi^2} (i\partial_t + \D^2) e^{-\alpha\Phi^2}f}_{L_{t,x}^2}^2 \geq c  \frac{\alpha^7 }{R^8}\norm{f}_{L_{t,x}^2}^2 .
\end{align}

We first write 
\begin{align}
e^{\alpha\Phi^2} (i\partial_t + \D^2) e^{-\alpha\Phi^2}f = : \mathcal{S} f + \mathcal{A} f
\end{align}
where $\mathcal{S}$ and $\mathcal{A}$ are respectively symmetric and anti-symmetric operators (with respect to the $L^2$ norm).

A direct computation gives that
\begin{itemize}
\item 
For $j =1$
\begin{align}
e^{\alpha\Phi^2 } \partial_{x_1}^{\, 4} (e^{\alpha\Phi^2} f) & = \partial_{x_1}^{\, 4} f + \partial_{x_1}^{\, 3} f \left[- \frac{8 \alpha \psi}{R}\right] + \partial_{x_1}^{\, 2} f \left[\frac{24 \alpha^2 \psi^2}{R^2} - \frac{12 \alpha}{R^2}\right] \\
& \quad + \partial_{x_1}f \left[-\frac{32 \alpha^3 \psi^3}{R^3} + \frac{48 \alpha^2 \psi}{R^3}\right] + f\left[ \frac{16 \alpha^4 \psi^4}{R^4} -\frac{48 \alpha^3 \psi^2}{R^4} + \frac{12 \alpha^2}{R^4}\right] ;
\end{align}

\item
For $j = 2 , \cdots , d$
\begin{align}
e^{\alpha\Phi^2 } \partial_{x_j}^{\, 4} (e^{\alpha\Phi^2} f) & = \partial_{x_j}^{\, 4} f + \partial_{x_j}^{\, 3} \left[- \frac{8 \alpha x_j}{R^2}\right] + \partial_{x_j}^{\, 2} \left[\frac{24 \alpha^2 x_j^2}{R^4} - \frac{12 \alpha}{R^2}\right] + \partial_{x_j} f \left[- \frac{32 \alpha^3 x_j^3}{R^6}\right] \\
& \quad + \left[\frac{48 \alpha^2 x_j}{R^4}\right] + f \left[  \frac{16 \alpha^4 x_j^4}{R^8} -\frac{48 \alpha^3 x_j^2}{R^6} + \frac{12 \alpha^2}{R^4}\right] .  
\end{align}
\end{itemize}
Then adding these two cases yields
\begin{align}
e^{\alpha\Phi^2 } \D^2 ( e^{ -\alpha\Phi^2 } f )& = \D^2 f   + \partial_{x_1}^{\, 3} f \left[-  \frac{8 \alpha \psi}{R} \right] + \sum_{j=2}^d   \partial_{x_j}^{\, 3} f \left[- \frac{8 \alpha x_j}{R^2}\right]  \label{eq 14} \\
& \quad + \partial_{x_1}^{\, 2} f \left[\frac{24 \alpha^2 \psi^2}{R^2} - \frac{12 \alpha}{R^2}\right] + \sum_{j=2}^d \partial_{x_j}^{\, 2} f \left[\frac{24 \alpha^2 x_j^2}{R^4} - \frac{12\alpha}{R^2}\right] \\
& \quad + \partial_{x_1} f \left[-\frac{32 \alpha^3 \psi^3}{R^3} +  \frac{48 \alpha^2 \psi}{R^3} \right] + \sum_{j=2}^d \partial_{x_j} f \left[- \frac{32 \alpha^3 x_j^3}{R^6} + \frac{48 \alpha^2 x_j}{R^4}\right] \\
& \quad + f \left[ \frac{16\alpha^4 \psi^4}{R^4} -\frac{48\alpha^3 \psi^2}{R^4} + \sum_{j=2}^d  ( \frac{16 \alpha^4 x_j^4}{R^8} - \frac{48 \alpha^3 x_j^2}{R^6} )  + \frac{12 d \alpha^2 }{R^4}\right] ,
\end{align}
and
\begin{align}
e^{\alpha\Phi^2} i\partial_t (e^{-\alpha\Phi^2} f) = e^{\alpha\Phi^2} i \left(e^{-\alpha\Phi^2} \partial_t f - e^{-\alpha\Phi^2} \alpha \left(\frac{2x_1}{R} \varphi' + 2 \varphi \varphi'\right)f\right) = i \partial_t f -  i \alpha \left(\frac{2x_1}{R} \varphi' + 2 \varphi \varphi'\right) f \label{eq 15}.
\end{align}

We recognize the symmetric and the anti-symmetric parts of the operator in \eqref{eq 14} and \eqref{eq 15}. 
The symmetric operator $\mathcal{S}$ is given by
\begin{align}
\mathcal{S}f & = i \partial_t f  + \D^2 f \\
& \quad + \partial_{x_1}^{\, 2} f \left[\frac{24 \alpha^2 \psi^2}{R^2} \right] + \sum_{j=2}^d \partial_{x_j}^{\, 2} f \left[\frac{24 \alpha^2 x_j^2}{R^4} \right]  \\
& \quad + \partial_{x_1} f \left[  \frac{48 \alpha^2 \psi}{R^3} \right] + \sum_{j=2}^d \partial_{x_j} f \left[ \frac{48 \alpha^2 x_j}{R^4}\right] \\
& \quad + f \left[ \frac{16\alpha^4 \psi^4}{R^4} + \sum_{j=2}^d   \frac{16 \alpha^4 x_j^4}{R^8}   + \frac{12 d \alpha^2 }{R^4}\right] .
\end{align}
We decompose $\mathcal{S}$ into
\begin{align}
\mathcal{S} f = : \mathcal{S}_t f + \mathcal{S}_{x_1} + \sum_{j=2}^d \mathcal{S}_{x_j} ,
\end{align}
where
\begin{align}
\mathcal{S}_t f & = i \partial_t f  ,\\
\mathcal{S}_{x_1} f & = \partial_{x_1}^{\, 4} f  + \partial_{x_1}^{\, 2} f \left[\frac{24 \alpha^2 \psi^2}{R^2} \right] + \partial_{x_1} f \left[  \frac{48 \alpha^2 \psi}{R^3} \right]  + f \left[ \frac{16\alpha^4 \psi^4}{R^4}  + \frac{12  \alpha^2 }{R^4}\right] ,\\
\mathcal{S}_{x_j} f & = \partial_{x_j}^{\, 4} f + \partial_{x_j}^{\, 2} f \left[\frac{24 \alpha^2 x_j^2}{R^4} \right] + \partial_{x_j} f \left[ \frac{48 \alpha^2 x_j}{R^4} \right] + f \left[  \frac{16 \alpha^4 x_j^4}{R^8}   + \frac{12  \alpha^2 }{R^4}\right], \quad j = 2, \cdots , d.
\end{align}

Then the anti-symmetric operator $\mathcal{A}$ is given by
\begin{align}
\mathcal{A}f & = f\left[- 2 i \alpha \left(\frac{x_1}{R} \varphi' +  \varphi \varphi'\right)\right]\\
& \quad + \partial_{x_1}^{\, 3} f \left[-  \frac{8 \alpha \psi}{R} \right] + \sum_{j=2}^d   \partial_{x_j}^{\, 3} f \left[- \frac{8 \alpha x_j}{R^2}\right]  \\
& \quad + \partial_{x_1}^{\, 2} f \left[- \frac{12 \alpha}{R^2}\right] + \sum_{j=2}^d \partial_{x_j}^{\, 2} f \left[ - \frac{12\alpha}{R^2}\right] \\
& \quad + \partial_{x_1} f \left[-\frac{32 \alpha^3 \psi^3}{R^3} \right] + \sum_{j=2}^d \partial_{x_j} f \left[- \frac{32 \alpha^3 x_j^3}{R^6} \right] \\
& \quad + f \left[ -\frac{48\alpha^3 \psi^2}{R^4} + \sum_{j=2}^d  ( - \frac{48 \alpha^3 x_j^2}{R^6} ) \right] .
\end{align}

We again decompose $\mathcal{A}$ into
\begin{align}
\mathcal{A}f =:  \mathcal{A}_{t}f + \mathcal{A}_{x_1} f + \sum_{j=2}^d \mathcal{A}_{x_j} f ,
\end{align}
where
\begin{align}
\mathcal{A}_t f & = f\left[- 2 i \alpha \left(\frac{x_1}{R} \varphi' +  \varphi \varphi'\right) \right] ,\\
\mathcal{A}_{x_1} f & = \partial_{x_1}^{\, 3} f \left[-  \frac{8 \alpha \psi}{R} \right] + \partial_{x_1}^{\, 2} f \left[- \frac{12 \alpha}{R^2}\right] + \partial_{x_1} f \left[-\frac{32 \alpha^3 \psi^3}{R^3} \right] + f \left[ -\frac{48\alpha^3 \psi^2}{R^4} \right] ,\\
\mathcal{A}_{x_j} f & =    \partial_{x_j}^{\, 3} f \left[- \frac{8 \alpha x_j}{R^2}\right]  + \partial_{x_j}^{\, 2} f \left[ - \frac{12\alpha}{R^2}\right]  +  \partial_{x_j} f \left[- \frac{32 \alpha^3 x_j^3}{R^6} \right] + f \left[  - \frac{48 \alpha^3 x_j^2}{R^6} \right] , \quad j = 2, \cdots , d.
\end{align}

Now we will compute the commutator $[\mathcal{S}, \mathcal{A}]$ using term by term. First notice that $[\mathcal{S}_{x_j}, \mathcal{A}_{x_j}] =0$ when $i \neq j$, and $[\mathcal{S}_t , \mathcal{A}_{x_j}] = [\mathcal{S}_{x_j} , \mathcal{A}_t] =0$, for $i \neq 1$. 
This observation implies that we only need to compute the following five cases:
\begin{enumerate}
    \item $[\mathcal{S}_t, \mathcal{A}_t]$
    \item $[\mathcal{S}_t, \mathcal{A}_{x_1}]$
    \item $[\mathcal{S}_{x_1}, \mathcal{A}_t]$
    \item $[\mathcal{S}_{x_1}, \mathcal{A}_{x_1}]$
    \item $[\mathcal{S}_{x_j}, \mathcal{A}_{x_j}]$
\end{enumerate}

For {\it Case (1)} in the list, we have
\begin{align}
\left[\mathcal{S}_t, \mathcal{A}_t \right]f & = \left[i\partial_t , - 2i \alpha \left(\frac{x_1}{R} \varphi' + \varphi \varphi'\right)\right]f = 2 \alpha \left[\partial_t  , \left(\frac{x_1}{R} \varphi' + \varphi \varphi'\right)\right]f \\
& = 2 \alpha \partial_t \left(\left(\frac{x_1}{R} \varphi' + \varphi \varphi'\right)f\right) - 2 \alpha \left(\frac{x_1}{R} \varphi' + \varphi \varphi'\right) \partial_t f \\
& = 2 \alpha \left(\frac{x_1}{R} \varphi'' + \varphi'^2 + \varphi \varphi'' \right) f .
\end{align}

Then consider {\it Case (2)}, we compute
\begin{align}
\left[i \partial_t , \frac{8 \alpha \psi}{R} \partial_{x_1}^{\, 3}\right]f & =  i \partial_t  \left(\frac{8 \alpha \psi}{R} \partial_{x_1}^{\, 3} f\right) - i \frac{8 \alpha \psi}{R} \partial_{x_1}^{\, 3} \left( \partial_t f \right)  = i \frac{8 \alpha \varphi'}{R} \partial_{x_1}^{\, 3} f , \\
\left[i \partial_t , \frac{12 \alpha}{R^2} \partial_{x_1}^{\, 2}\right]f & = 0 ,\\
\left[i \partial_t , \frac{32 \alpha^3 \psi^3}{R^3} \partial_{x_1}\right]f & =  i \partial_t  \left(\frac{32 \alpha^3 \psi^3}{R^3} \partial_{x_1} f\right) - i \frac{32 \alpha^3 \psi^3}{R^3} \partial_{x_1} \left( \partial_t f \right)  = i \frac{96 \alpha^3 \psi^2 \varphi'}{R^3} \partial_{x_1} f , \\
\left[i \partial_t , \frac{48 \alpha^3 \psi^2}{R^4}\right] f & = i \partial_t \left( \frac{48 \alpha^3 \psi^2}{R^4} f \right) - i \frac{48 \alpha^3 \psi^2}{R^4} \left( \partial_t f \right) = i \frac{96 \alpha^3 \psi \varphi'}{R^4} f .
\end{align}
Then adding them together, we have 
\begin{align}
\left[\mathcal{S}_t ,\mathcal{A}_{x_1}\right] f & = -i \frac{8 \alpha \varphi'}{R} \partial_{x_1}^{\, 3} f - i \frac{96 \alpha^3 \psi^2 \varphi'}{R^3} \partial_{x_1} f - i \frac{96 \alpha^3 \psi \varphi'}{R^4} f  .
\end{align}

For {\it Case 3}, we compute first
\begin{align}
\left[ \partial_{x_1}^{\, 4} , 2 i \alpha \left(\frac{x_1}{R} \varphi' +  \varphi \varphi'\right) \right]f & = 2i \alpha \partial_{x_1}^{\, 4} \left(\left(\frac{x_1}{R} \varphi' +  \varphi \varphi'\right)f\right) - 2 i\alpha \left(\frac{x_1}{R} \varphi' +  \varphi \varphi'\right) \partial_{x_1}^{\, 4} f = 2i \alpha \frac{4}{R} \varphi' \partial_{x_1}^{\, 3} f , \\
\left[ \frac{24 \alpha^2 \psi^2}{R^2}  \partial_{x_1}^{\, 2} , 2 i \alpha \left(\frac{x_1}{R} \varphi' +  \varphi \varphi'\right) \right]f & = 2i \alpha \frac{24 \alpha^2 \psi^2}{R^2}  \partial_{x_1}^{\, 2} \left(\left(\frac{x_1}{R} \varphi' +  \varphi \varphi'\right) f\right) - 2i \alpha \left(\frac{x_1}{R} \varphi' +  \varphi \varphi'\right) \frac{24 \alpha^2 \psi^2}{R^2}  \partial_{x_1}^{\, 2} f  \\
& = 2i \alpha \frac{24 \alpha^2 \psi^2}{R^2} \frac{2}{R}\varphi' \partial_{x_1} f ,\\
\left[ \frac{48 \alpha^2 \psi}{R^3} \partial_{x_1} , 2 i \alpha  \left(\frac{x_1}{R} \varphi' +  \varphi \varphi'\right) \right]f & = 2i \alpha \frac{48 \alpha^2 \psi}{R^3} \partial_{x_1} \left(\left(\frac{x_1}{R} \varphi' +  \varphi \varphi'\right)f\right) - 2i \alpha \left(\frac{x_1}{R} \varphi' +  \varphi \varphi'\right) \frac{48 \alpha^2 \psi}{R^3} \partial_{x_1} f \\
& = 2i \alpha \frac{48 \alpha^2 \psi}{R^3} \frac{1}{R}\varphi' f ,\\
\left[\left( \frac{16\alpha^4 \psi^4}{R^4}  + \frac{12  \alpha^2 }{R^4}\right),  2 i \alpha \left(\frac{x_1}{R} \varphi' +  \varphi \varphi'\right) \right]f & = 0  .
\end{align}
Then summing up all the terms above, we have
\begin{align}
\left[\mathcal{S}_{x_1} , \mathcal{A}_t\right] f & = - 2i \alpha \frac{4}{R} \varphi' \partial_{x_1}^{\, 3} f - 2i \alpha \frac{48 \alpha^2 \psi^2}{R^3} \varphi' \partial_{x_1} f - 2i \alpha \frac{48 \alpha^2 \psi}{R^4} \varphi'f .
\end{align}

Next for {\it Cases (4) and (5)} in the list, using \emph{Mathematica} we get
\begin{align}\label{eq 16}
\begin{aligned}
\left[\mathcal{S}_{x_1}, \mathcal{A}_{x_1}\right] f & = f \left[  - \frac{1536 \alpha^5 \psi^2  }{R^8}  + \frac{2048 \alpha^7 \psi^6 }{R^8} \right]   - \frac{6144 \alpha^5 \psi^3 \partial_{x_1} f}{R^7}  + \frac{384 \alpha^3  \partial_{x_1}^{\, 2} f}{R^6}  - \frac{1536 \alpha^5 \psi^4  \partial_{x_1}^{\, 2} f}{R^6}  \\
& \quad + \frac{1536 \psi  \partial_{x_1}^{\, 3} f}{R^5}  +  \frac{384 \alpha^3 \psi^2  \partial_{x_1}^{\, 4} f}{R^4}   - \frac{32 \alpha  \partial_{x_1}^{\, 6} f}{R^2} ,
\end{aligned}
\end{align}
and
\begin{align}\label{eq 17}
\begin{aligned}
\left[\mathcal{S}_{x_j}, \mathcal{A}_{x_j}\right] f & = f \left[ - \frac{1536 \alpha^5 x_j^2}{R^{10}} + \frac{2048 \alpha^7 x_j^6 }{R^{14}} \right]  - \frac{6144 \alpha^5 x_j^3 \partial_{x_j} f}{R^{10}}     + \frac{384 \alpha^3 \partial_{x_j}^{\, 2}f}{R^6} - \frac{1536 \alpha^5 x_j^4 \partial_{x_j}^{\, 2} f}{R^{10}}   \\
& \quad +\frac{1536 \alpha^3 x_j \partial_{x_j}^{\, 3} f}{R^6}  + \frac{384 \alpha^3 x_j^2 \partial_{x_j}^{\, 4} f }{R^6}   - \frac{32\alpha \partial_{x_j}^{\, 6} f}{R^2}.
\end{aligned}
\end{align}

Combining \eqref{eq 16} and \eqref{eq 17}, we obtain a part of the commutator $\left[\mathcal{S} , \mathcal{A}\right]$
\begin{align}
\left[\mathcal{S}_{x_1} , \mathcal{A}_{x_1} \right]f + \sum_{j=2}^d \left[\mathcal{S}_{x_j}, \mathcal{A}_{x_j}\right]f & = f \left[  - \frac{1536 \alpha^5 \psi^2  }{R^8}  + \frac{2048 \alpha^7 \psi^6 }{R^8} + \sum_{j=2}^d \left(- \frac{1536 \alpha^5 x_j^2}{R^{10}} + \frac{2048 \alpha^7 x_j^6 }{R^{14}}\right) \right] \\
& \quad  - \frac{6144 \alpha^5 \psi^3 \partial_{x_1} f}{R^7} + \sum_{j=2}^d \left(- \frac{6144 \alpha^5 x_j^3 \partial_{x_j} f}{R^{10}} \right) \\
& \quad +  \frac{384 \alpha^3  \partial_{x_1}^{\, 2} f}{R^6}  - \frac{1536 \alpha^5 \psi^4  \partial_{x_1}^{\, 2} f}{R^6} + \sum_{j=2}^d \left(\frac{384 \alpha^3 \partial_{x_j}^{\, 2} f}{R^6} - \frac{1536 \alpha^5 x_j^4 \partial_{x_j}^{\, 2}  f}{R^{10}} \right)\\
& \quad + \frac{1536 \psi  \partial_{x_1}^{\, 3} f}{R^5} + \sum_{j=2}^d \frac{1536 \alpha^3 x_j \partial_{x_j}^{\, 3} f}{R^6}  \\
& \quad + \frac{384 \alpha^3 \psi^2  \partial_{x_1}^{\, 4} f}{R^4} + \sum_{j=2}^d \frac{384 \alpha^3 x_j^2 \partial_{x_j}^{\, 4} f }{R^6}  \\
& \quad  - \frac{32 \alpha  \partial_{x_1}^{\, 6} f}{R^2} + \sum_{j=2}^d \left(- \frac{32\alpha \partial_{x_j}^{\, 6} f}{R^2} \right) .
\end{align}

Together with all the five cases in the list, we arrive at
\begin{align}
\left[\mathcal{S},\mathcal{A}\right] f & = \left[\mathcal{S}_t , \mathcal{A}_t\right]f + \left[\mathcal{S}_t, \mathcal{A}_{x_1}\right]f + \left[\mathcal{S}_{x_1} ,\mathcal{A}_t\right]f + \left[\mathcal{S}_{x_1} , \mathcal{A}_{x_1} \right]f + \sum_{j=2}^d \left[\mathcal{S}_{x_j}, \mathcal{A}_{x_j}\right]f\\
& = 2 \alpha \left(\frac{x_1}{R} \varphi'' + \varphi'^2 + \varphi \varphi'' \right) f  -i \frac{16 \alpha \varphi'}{R} \partial_{x_1}^{\, 3} f - i \frac{192 \alpha^3 \psi^2 \varphi'}{R^3} \partial_{x_1} f - i \frac{192 \alpha^3 \psi \varphi'}{R^4} f \\
& \quad + f \left[  - \frac{1536 \alpha^5 \psi^2  }{R^8}  + \frac{2048 \alpha^7 \psi^6 }{R^8} + \sum_{j=2}^d \left(- \frac{1536 \alpha^5 x_j^2}{R^{10}} + \frac{2048 \alpha^7 x_j^6 }{R^{14}}\right) \right] \\
& \quad + - \frac{6144 \alpha^5 \psi^3 \partial_{x_1} f}{R^7} + \sum_{j=2}^d \left(- \frac{6144 \alpha^5 x_j^3 \partial_{x_j} f}{R^{10}} \right)\\
& \quad + \frac{384 \alpha^3  \partial_{x_1}^{\, 2} f}{R^6}  - \frac{1536 \alpha^5 \psi^4  \partial_{x_1}^{\, 2} f}{R^6} + \sum_{j=2}^d \left(\frac{384 \alpha^3 \partial_{x_j}^{\, 2} f}{R^6} - \frac{1536 \alpha^5 x_j^4 \partial_{x_j}^{\, 2}  f}{R^{10}} \right) \\
& \quad + \frac{1536 \psi  \partial_{x_1}^{\, 3} f}{R^5} + \sum_{j=2}^d \frac{1536 \alpha^3 x_j \partial_{x_j}^{\, 3} f}{R^6}  \\
& \quad +  \frac{384 \alpha^3 \psi^2  \partial_{x_1}^{\, 4} f}{R^4} + \sum_{j=2}^d \frac{384 \alpha^3 x_j^2 \partial_{x_j}^{\, 4} f }{R^6}  \\
& \quad  - \frac{32 \alpha  \partial_{x_1}^{\, 6} f}{R^2} + \sum_{j=2}^d \left(- \frac{32\alpha \partial_{x_j}^{\, 6} f}{R^2} \right).
\end{align}

Next we compute the inner product $\inner{f, \left[\mathcal{S},\mathcal{A}\right]f}$ and our aim is to find an lower bound for it.
\begin{align}
& \inner{f, \left[\mathcal{S},\mathcal{A}\right]f}_{L_{t,x}^2 \times L_{t,x}^2} \label{eq SA}\\
& = \iint \bar{f} \left[2 \alpha \left(\frac{x_1}{R} \varphi'' + \varphi'^2 + \varphi \varphi'' \right) f  -i \frac{16 \alpha \varphi'}{R} \partial_{x_1}^{\, 3} f - i \frac{192 \alpha^3 \psi^2 \varphi'}{R^3} \partial_{x_1} f - i \frac{192 \alpha^3 \psi \varphi'}{R^4} f \right] \, dx \, dt \label{eq SA0}\\
& \quad + \iint \abs{f}^2 \left[  - \frac{1536 \alpha^5 \psi^2  }{R^8}  + \frac{2048 \alpha^7 \psi^6 }{R^8} + \sum_{j=2}^d \left(- \frac{1536 \alpha^5 x_j^2}{R^{10}} + \frac{2048 \alpha^7 x_j^6 }{R^{14}}\right) \right] \, dx \, dt \label{eq SA1}\\
& \quad + \iint \bar{f} \partial_{x_1} f \left[ - \frac{6144 \alpha^5 \psi^3}{R^7} \right]+ \sum_{j=2}^d \bar{f}\partial_{x_j} f \left[- \frac{6144 \alpha^5 x_j^3 }{R^{10}} \right] \, dx \, dt \label{eq SA2}\\
& \quad + \iint \bar{f}\partial_{x_1}^{\, 2} f \left[  \frac{384 \alpha^3  }{R^6}  - \frac{1536 \alpha^5 \psi^4 }{R^6}\right] + \sum_{j=2}^d  \bar{f}\partial_{x_j}^{\, 2} f \left[\frac{384 \alpha^3 }{R^6} - \frac{1536 \alpha^5 x_j^4 }{R^{10}}\right] \, dx \, dt \label{eq SA3}\\
& \quad + \iint \bar{f}\partial_{x_1}^{\, 3} f\left[ \frac{1536 \psi  }{R^5}\right] + \sum_{j=2}^d \bar{f} \partial_{x_j}^{\, 3} f \left[\frac{1536 \alpha^3 x_j }{R^6} \right] \, dx \, dt \label{eq SA4}\\
& \quad + \iint \bar{f} \partial_{x_1}^{\, 4} f \left[ \frac{384 \alpha^3 \psi^2 }{R^4}\right] + \sum_{j=2}^d \bar{f}\partial_{x_j}^{\, 4} f \left[ \frac{384 \alpha^3 x_j^2 }{R^6} \right] \, dx \, dt \label{eq SA5}\\
& \quad + \iint \bar{f} \partial_{x_1}^{\, 6} f \left[- \frac{32 \alpha }{R^2}\right] + \sum_{j=2}^d \bar{f} \partial_{x_j}^{\, 6} f \left[- \frac{32\alpha }{R^2} \right] \, dx \, dt .\label{eq SA6}
\end{align}

To this end, we need to preform a few integration by parts term by term.

\noindent {\it \underline{Term \eqref{eq SA0}.}}
First take the last term in \eqref{eq SA0}. 
An integration by parts yields 
\begin{align}
\iint - i \frac{192 \alpha^3  \psi^2  \varphi'}{R^3}  \bar{f} \partial_{x_1} f \, dx \, dt = \iint i \frac{192 \alpha^3  \psi^2 \varphi'}{R^3}  \partial_{x_1} \overline{f} f \, dx \, dt  + \iint i \frac{192 \alpha^3 \psi \varphi'}{R^4}  2  \abs{f}^2 \, dx \, dt ,
\end{align}
which implies
\begin{align}
\iint i \frac{192 \alpha^3 \psi \varphi'}{R^4}  \abs{f}^2 \, dx \, dt = - \iint i\frac{192 \alpha^3 \psi^2 \varphi'}{R^3}   \re \left(\bar{f} \partial_{x_1} f\right) \, dx \, dt.
\end{align}
Hence the last two terms in \eqref{eq SA0} is given by
\begin{align}
& \quad \iint - i \frac{192 \alpha^3  \psi^2 \varphi'}{R^3}  \bar{f} \partial_{x_1} f - i \frac{192 \alpha^3 \psi \varphi' }{R^4}   \abs{f}^2 \, dx \, dt \\
& = \iint - i \frac{192 \alpha^3\psi^2 \varphi'  }{R^3}  \bar{f} \partial_{x_1} f \, dx \, dt + \iint  i \frac{192 \alpha^3 \psi^2 \varphi'}{R^3}  \re \left(\bar{f} \partial_{x_1} f\right) \, dx \, dt\\
& = \iint \frac{192 \alpha^3 \psi^2 \varphi'}{R^3}   \im \left(\bar{f} \partial_{x_1} f\right) \, dx \, dt .
\end{align}
The second term in \eqref{eq SA0} can be written as 
\begin{align}
\iint -i \frac{16 \alpha \varphi'}{R} \bar{f} \partial_{x_1}^{\, 3} f \, dx \, dt & = \iint i \frac{16 \alpha \varphi'}{R} \partial_{x_1}\bar{f} \partial_{x_1}^{\, 2} f \, dx \, dt = \iint - i \frac{16 \alpha \varphi'}{R} \partial_{x_1}^{\, 2} \bar{f} \partial_{x_1} f \, dx \, dt
\end{align}
which implies
\begin{align}
\iint -i \frac{16 \alpha \varphi'}{R} \bar{f} \partial_{x_1}^{\, 3} f \, dx \, dt =  \iint - \frac{16 \alpha \varphi'}{R} \im \left(\partial_{x_1}\bar{f} \partial_{x_1}^{\, 2} f \right) \, dx \, dt .
\end{align}

Therefore, 
\begin{align}
\eqref{eq SA0} & = \iint \abs{f}^2 \left[2 \alpha \left(\left(\frac{x_1}{R} + \varphi\right)\varphi'' +  \varphi'^2 \right) \right] - \frac{16 \alpha \varphi'}{R} \im \left(\partial_{x_1}\bar{f} \partial_{x_1}^{\, 2} f\right) + \frac{192 \alpha^3 \psi^2 \varphi'}{R^3} \im \left(\bar{f} \partial_{x_1}f\right)  \, dx \, dt .
\end{align}

\noindent {\it \underline{Term \eqref{eq SA1}.}}
Rewriting it in the following form
\begin{align}
\eqref{eq SA1} & = \iint \abs{f}^2 \left[  - \frac{1536 \alpha^5 \psi^2  }{R^8}  + \frac{2048 \alpha^7 \psi^6 }{R^8} + \sum_{j=2}^d \left(- \frac{1536 \alpha^5 x_j^2}{R^{10}} + \frac{2048 \alpha^7 x_j^6 }{R^{14}}\right) \right]  \, dx \, dt\\
& = \iint \abs{f}^2 \left[- \frac{1536 \alpha^5 }{R^8} \Phi^2 + \frac{2048 \alpha^7 }{R^8} \left(\psi^6 + \sum_{j=2}^d \left(\frac{x_j}{R}\right)^6\right) \right]  \, dx \, dt .
\end{align}

\noindent {\it \underline{Terms \eqref{eq SA2} and \eqref{eq SA3}.}}
Again integrating by parts yields
\begin{align}
\eqref{eq SA3} & = \iint \bar{f}\partial_{x_1}^{\, 2} f \left[  \frac{384 \alpha^3  }{R^6}  - \frac{1536 \alpha^5 \psi^4 }{R^6}\right] + \sum_{j=2}^d  \bar{f}\partial_{x_j}^{\, 2} f \left[\frac{384 \alpha^3 }{R^6} - \frac{1536 \alpha^5 x_j^4 }{R^{10}}\right]  \, dx \, dt\\
& = \iint \abs{\partial_{x_1}f}^2 \left[ - \frac{384 \alpha^3  }{R^6}  + \frac{1536 \alpha^5 \psi^4 }{R^6}\right] + \sum_{j=2}^d \abs{\partial_{x_j}f}^2 \left[ - \frac{384 \alpha^3 }{R^6} + \frac{1536 \alpha^5 x_j^4 }{R^{10}}\right]  \, dx \, dt\\
& \quad + \iint  \bar{f}\partial_{x_1}f \left[\frac{6144 \alpha^5 \psi^3}{R^7}\right] + \sum_{j=2}^d  \bar{f} \partial_{x_j} f \left[\frac{6144 \alpha^5 x_j^3}{R^{10}}\right]  \, dx \, dt . \label{eq 19}
\end{align}

Noticing that  the last two terms  in \eqref{eq 19} is the opposite of \eqref{eq SA2}, hence 
\begin{align}
\eqref{eq SA2} + \eqref{eq SA3} & = \iint  \abs{\partial_{x_1}f}^2 \left[ - \frac{384 \alpha^3  }{R^6}  + \frac{1536 \alpha^5 \psi^4 }{R^6}\right] + \sum_{j=2}^d \abs{\partial_{x_j}f}^2 \left[ - \frac{384 \alpha^3 }{R^6} + \frac{1536 \alpha^5 x_j^4 }{R^{10}}\right]  \, dx \, dt .
\end{align}

\noindent {\it \underline{Terms \eqref{eq SA4} and \eqref{eq SA5}.}}
Preforming integration by parts again, we write  
\begin{align}
\text{Term 1 in }\eqref{eq SA5} & = \iint  \bar{f} \partial_{x_1}^{\, 4} f \frac{384 \alpha^3 \psi^2  }{R^4}  \, dx \, dt\\
& = \iint - \partial_{v} \bar{f} \partial_{x_1}^{\, 3}f \frac{384 \alpha^3 \psi^2  }{R^4} - \bar{f} \partial_{x_1}^{\, 3} f \frac{768 \alpha^3 \psi  }{R^5}  \, dx \, dt\\
& = \iint \abs{\partial_{x_1}^{\, 2} f}^2 \frac{384 \alpha^3 \psi^2  }{R^4} + \partial_{x_1} \bar{f} \partial_{x_1}^{\, 2}f \frac{768 \alpha^3 \psi  }{R^5} - \bar{f} \partial_{x_1}^{\, 3} f \frac{768 \alpha^3 \psi  }{R^5}  \, dx \, dt\\
& = \iint \abs{\partial_{x_1}^{\, 2} f}^2 \frac{384 \alpha^3 \psi^2  }{R^4} - \bar{f} \partial_{x_1}^{\, 3} f \frac{768 \alpha^3 \psi  }{R^5} - \bar{f} \partial_{x_1}^{\, 2} f \frac{768 \alpha^3  }{R^6}  - \bar{f} \partial_{x_1}^{\, 3} f \frac{768 \alpha^3 \psi  }{R^5}  \, dx \, dt\\
& = \iint \abs{\partial_{x_1}^{\, 2} f}^2 \frac{384 \alpha^3 \psi^2  }{R^4} - \bar{f} \partial_{x_1}^{\, 3} f \frac{1536 \alpha^3 \psi  }{R^5} + \abs{\partial_{x_1} f}^2  \frac{768 \alpha^3  }{R^6}   \, dx \, dt , \label{eq 20}
\end{align}
and
\begin{align}
\text{Term 2 in }\eqref{eq SA5} & = \iint \bar{f} \partial_{x_j}^{\, 4} f \frac{384 \alpha^3 x_j^2  }{R^6}  \, dx \, dt \\
& = \iint - \partial_{x_j} \bar{f} \partial_{x_j}^{\, 3} f \frac{384 \alpha^3 x_j^2  }{R^6} - \bar{f} \partial_{x_j}^{\, 3} f \frac{768\alpha^3 x_j}{R^6}  \, dx \, dt\\
& = \iint \abs{\partial_{x_j}^{\, 2} f}^2 \frac{384 \alpha^3 x_j^2  }{R^6} + \partial_{x_j} \bar{f} \partial_{x_j}^{\, 2} f \frac{768\alpha^3 x_j}{R^6} - \bar{f} \partial_{x_j}^{\, 3} f \frac{768\alpha^3 x_j}{R^6}  \, dx \, dt\\
& = \iint \abs{\partial_{x_j}^{\, 2} f}^2 \frac{384 \alpha^3 x_j^2  }{R^6} - \bar{f} \partial_{x_j}^{\, 3} f \frac{768\alpha^3 x_j}{R^6} - \bar{f} \partial_{x_j}^{\, 2} \frac{768\alpha^3 }{R^6} - \bar{f} \partial_{x_j}^{\, 3} f \frac{768\alpha^3 x_j}{R^6}  \, dx \, dt\\
& = \iint \abs{\partial_{x_j}^{\, 2} f}^2 \frac{384 \alpha^3 x_j^2  }{R^6} - \bar{f} \partial_{x_j}^{\, 3} f \frac{1536\alpha^3 x_j}{R^6} + \abs{\partial_{x_j} f}^2  \frac{768\alpha^3 }{R^6}   \, dx \, dt . \label{eq 21}
\end{align}

Noticing that the second terms in \eqref{eq 20} and \eqref{eq 21} show up in \eqref{eq SA4}   with  the opposite sign, we have
\begin{align}
\eqref{eq SA4} + \eqref{eq SA5}  & = \iint \abs{\partial_{x_1}^{\, 2} f}^2 \frac{384 \alpha^3 \psi^2  }{R^4} + \abs{\partial_{x_1} f}^2  \frac{768 \alpha^3  }{R^6} + \sum_{j=2}^d \abs{\partial_{x_j}^{\, 2} f}^2 \frac{384 \alpha^3 x_j^2  }{R^6} + \sum_{j=2}^d \abs{\partial_{x_j} f}^2  \frac{768\alpha^3 }{R^6}  \, dx \, dt .
\end{align}

\noindent {\it \underline{Term \eqref{eq SA6}.}}
Finally, we write 
\begin{align}
\eqref{eq SA6} & = \iint \bar{f}  \partial_{x_1}^{\, 6} f \left[ - \frac{32 \alpha }{R^2}\right] + \sum_{j=2}^d \bar{f}\partial_{x_j}^{\, 6} f \left[- \frac{32\alpha }{R^2}\right]   \, dx \, dt = \iint  \abs{\partial_{x_1}^{\, 3} f}^2 \frac{32\alpha }{R^2} + \sum_{j=2}^d  \abs{\partial_{x_j}^{\, 3} f}^2 \frac{32\alpha }{R^2}  \, dx \, dt .
\end{align}

Therefore, summarizing Terms \eqref{eq SA0} - \eqref{eq SA6} we conclude that
\begin{align}
\inner{f , \left[\mathcal{S}, \mathcal{A}\right]f}_{L_{t,x}^2 \times L_{t,x}^2} & = \iint \abs{f}^2 \left[2 \alpha \left(\left(\frac{x_1}{R} + \varphi\right)\varphi'' +  \varphi'^2 \right) \right] - \frac{16 \alpha \varphi'}{R} \im \left(\partial_{x_1}\bar{f} \partial_{x_1}^{\, 2} f\right) + \frac{192 \alpha^3 \psi^2 \varphi'}{R^3} \im \left(\bar{f} \partial_{x_1}f\right)  \, dx \, dt\\
& \quad + \iint \abs{f}^2 \left[- \frac{1536 \alpha^5 }{R^8} \Phi^2 + \frac{2048 \alpha^7 }{R^8} \left(\psi^6 + \sum_{j=2}^d \left(\frac{x_j}{R}\right)^6\right) \right]  \, dx \, dt\\
& \quad + \iint \abs{\partial_{x_1}f}^2 \left[ \frac{384 \alpha^3  }{R^6}  + \frac{1536 \alpha^5 \psi^4 }{R^6} \right] + \sum_{j=2}^d \abs{\partial_{x_j}f}^2 \left[ \frac{384 \alpha^3 }{R^6} + \frac{1536 \alpha^5 x_j^4 }{R^{10}} \right]  \, dx \, dt\\ 
& \quad + \iint \abs{\partial_{x_1}^{\, 2} f}^2 \frac{384 \alpha^3 \psi^2  }{R^4} + \sum_{j=2}^d \abs{\partial_{x_j}^{\, 2} f}^2 \frac{384 \alpha^3 x_j^2  }{R^6}  \, dx \, dt\\
& \quad + \iint  \abs{\partial_{x_1}^{\, 3} f}^2 \frac{32\alpha }{R^2} + \sum_{j=2}^d  \abs{\partial_{x_j}^{\, 3} f}^2 \frac{32\alpha }{R^2}  \, dx \, dt .
\end{align}

To estimates the mixed terms above, we employ Cauchy–Schwarz inequality to control
\begin{align}\label{eq 7}
\begin{aligned}
& \abs{\im \left(\partial_{x_1} \overline{f} \partial_{x_1}^{\,2} f\right)} \leq  \omega \abs{\partial_{x_1} f}^2 + \frac{1}{\omega} \abs{\partial_{x_1}^{\, 2} f}^2 , \\
& \abs{\im \left(\overline{f} \partial_{x_1} f\right)} \leq \rho  \abs{f}^2  +  \frac{1}{\rho} \abs{\partial_{x_1} f}^2  ,
\end{aligned}
\end{align}
where $\omega ,  \rho$ are arbitrary positive  constants, and will be chosen later.

Now we obtain a lower bound of $\inner{f , \left[\mathcal{S}, \mathcal{A}\right] f}$
\begin{align}
&  \inner{f , \left[\mathcal{S}, \mathcal{A}\right] f}_{L_{t,x}^2 \times L_{t,x}^2} \\
& \quad \geq  \iint \abs{f}^2 \left[- \frac{1536 \alpha^5 }{R^8} \Phi^2 + \frac{2048 \alpha^7 }{R^8} \left(\psi^6 + \sum_{j=2}^d \left(\frac{x_j}{R}\right)^6\right) + 2 \alpha \left(\left(\frac{x_1}{R} + \varphi\right)\varphi'' +  \varphi'^2 \right) - \rho \frac{192 \alpha^3 \psi^2 \varphi'}{R^3}\right]  \, dx \, dt \label{eq In0}\\
& \qquad + \iint \abs{\partial_{x_1}f}^2 \left[ \frac{384 \alpha^3  }{R^6}  + \frac{1536 \alpha^5 \psi^4 }{R^6}  - \omega \frac{16 \alpha \varphi'}{R} - \frac{192 \alpha^3 \psi^2 \varphi'}{\rho R^3} \right] + \sum_{j=2}^d \abs{\partial_{x_j}f}^2 \left[ \frac{384 \alpha^3 }{R^6} + \frac{1536 \alpha^5 x_j^4 }{R^{10}} \right]  \, dx \, dt \label{eq In1}\\ 
& \qquad + \iint \abs{\partial_{x_1}^{\, 2} f}^2 \left[\frac{384 \alpha^3 \psi^2  }{R^4} -  \frac{16 \alpha \varphi'}{ \omega R} \right]+ \sum_{j=2}^d \abs{\partial_{x_j}^{\, 2} f}^2 \frac{384 \alpha^3 x_j^2  }{R^6}  \, dx \, dt \label{eq In2}\\
& \qquad + \iint  \abs{\partial_{x_1}^{\, 3} f}^2 \frac{32\alpha }{R^2} + \sum_{j=2}^d  \abs{\partial_{x_j}^{\, 3} f}^2 \frac{32\alpha }{R^2}  \, dx \, dt . \label{eq In3}
\end{align}
By choosing $\rho \sim R^{\frac{1}{3}}$, $\omega \sim R^{\frac{1}{3}}$ (with suitable constants and $\alpha = c R^{\frac{4}{3}}$ (where $c = c\left(d, \norm{\varphi'}_{L^{\infty}} , \norm{\varphi''}_{L^{\infty}}\right)$), we can make the first term and last two terms \eqref{eq In0} absorbed by the second term in \eqref{eq In0}; and hide the terms with negative signs in \eqref{eq In1} and \eqref{eq In2} by the first positive terms in  \eqref{eq In1} and \eqref{eq In2} respectively. Since the two terms in \eqref{eq In3} are both non-negative, we then finally obtain 
\begin{align}
\inner{f, \left[\mathcal{S}, \mathcal{A}\right] f }_{L_{t,x}^2 \times L_{t,x}^2} \geq  c  \frac{\alpha^7 }{R^8}\norm{f}_{L_{t,x}^2}^2 .
\end{align}

Recall $f = e^{\alpha\Phi^2} u$ and $\Phi^2  \left(t,x\right)=  \left(\frac{x_1}{R}+ \varphi \left(t\right)\right)^2 +  \sum_{j=2}^d \left(\frac{x_j}{R}\right)^2$, then we have
\begin{align}
\norm{e^{\alpha\Phi^2} \left(i\partial_t + \D^2\right) u}_{L_{t,x}^2}^2 & \geq \iint \bar{f} \left[\mathcal{S}, \mathcal{A}\right] f \, dx \, dt  \geq c  \frac{\alpha^7 }{R^8}\norm{f}_{L_{t,x}^2}^2 = c \frac{\alpha^7}{R^8}  \norm{e^{\alpha\Phi^2} u}_{L_{t,x}^2}^2,
\end{align}
which finishes the proof of Lemma \ref{lem Carleman}.
\end{proof}

\begin{rmk}
We can only prove the estimate for the operator $i\partial_t +\D^2$ because it is not obvious how to obtain the desired inequality for the operator $i\partial_t +\Delta^2$ containing mixed terms since the latter resists being lower-bounded in a positive fashion. We believe that this inequality is highly delicate and nontrivial because it is far from clear at first sight that it would follow using the same methods that proved a previous Carleman inequality for the classical Schr\"odinger operator $i\partial_t +\Delta$ as well as a result of the very large number of terms involved in the calculations, which might not be expected to sum together to produce a positive lower bound with the condition $\alpha\ge cR^{4/3}$ (which gives us the sharpest possible unique continuation theorem). In fact, the proof is quite delicate and it relies on a series of very careful estimates that if not done correctly will give a weaker inequality that is only valid for a smaller range of $\alpha$. 
\end{rmk}

\section{Lower Bound Estimates}\label{sec Lower}
In this section, we prove a lower bound for solutions with {\it fast decay}, which will be used in the next section to prove the main theorem by way of a contradiction argument.
\begin{lem}[Lower bounds]\label{lem Lower bound}
Let $u\in C^1([0,1] : H^3(\mathbb{R}^d))$ solve \eqref{eq 4SE} and let $B_R :=\{x\in\mathbb{R}^d:|x|\le R\}$.

If
\begin{align}
\int_{1/2-1/8}^{1/2+1/8}\int_{B_1}|u(t,x)|^2 \, dx \, dt & \ge 1 , \label{eq B1}\\
\int_0^1 \int_{\R^d} \sum_{j=1}^d |u|^2+|\partial_{x_j} u|^2 + |\partial_{x_j}^{\, 2} u|^2 + |\partial_{x_j}^{\, 3} u|^2\,dx  \, dt & \leq A^2 \label{eq A} \\
\norm{V}_{L_{t,x}^{\infty} ([0,1] \times \R^d)} &  \leq L \label{eq L}
\end{align}
for some $A, L > 0$.

Then there exists $R_0 = R_0 (d, A, L) >0$ and  $c= c(d)$ such that 
\begin{align}\label{eq Gamma}
\gamma(R):=\left(\int_0^1 \int_{R-1<|x|<R} \sum_{j=1}^d |u|^2+|\partial_{x_j} u|^2 + |\partial_{x_j}^{\, 2} u|^2 + |\partial_{x_j}^{\, 3} u|^2\,dx  \, dt\right)^{\frac{1}{2}} \ge  c R^{\frac{2}{3}} e^{-c R^{\frac{4}{3}}}, 
\end{align}
for all $R > R_0$. 
\end{lem}

\begin{proof}[Proof of Lemma \ref{lem Lower bound}]
Let us start with introducing some cutoff functions
\begin{itemize}
\item 
Let $\eta (x) \in C^\infty (\R)$ be  non-decreasing, radial and such that 
\begin{align}
\eta (x) = 
\begin{cases}
0, & \text{if } \abs{x} \leq 1, \\
1, & \text{if } \abs{x} \geq 2. 
\end{cases}
\end{align}
We also define $\eta_R(x) = \eta (\frac{x}{R})$.

\item
Let $\theta_R(x)\in C^\infty (\R^d)$ be non-decreasing, radial and such that 
\begin{align}
\theta_R (x) = 
\begin{cases}
1, & \text{if } \abs{x} \leq R-1, \\
0, & \text{if } \abs{x} \geq R. 
\end{cases}
\end{align}
We also define $\varphi_R(x) = 1- \theta_R(x)$.

\item
Let  $\varphi \in C^\infty $ satisfy $\varphi(t) \in [0,M]$ on $[0,1]$ and 
\begin{align}
\varphi (t) = 
\begin{cases}
M , & \text{on } [\frac{1}{2}-\frac{1}{8},\frac{1}{2}+\frac{1}{8}], \\
0, & \text{on } [0,\frac{1}{4}]\cup [\frac{3}{4}, 1]. 
\end{cases}
\end{align}
\end{itemize}

We now apply Lemma \ref{lem Carleman} to the function
\begin{align}
f(t,x) & = \sigma(t,x) u(t,x),  
\end{align}
where 
\begin{align}
\sigma(t,x) = \theta_R(x)\eta\left( \frac{x_1}{R} + \varphi(t)   \right) 
\end{align}

Let $\Phi (t,x) : = \abs{\frac{x}{R}+\varphi(t) \vec{e}_1}$, where $\vec{e}_1$ is the unit vector $(1,0, \cdots ,0)$. 
We see that $f$ is compactly supported on $ [0,1] \times \R^d$ and satisfies the hypothesis of Lemma \ref{lem Carleman}. In fact, we notice that $f =u$ on $[\frac{1}{2} - \frac{1}{8}, \frac{1}{2} + \frac{1}{8}] \times B_{R-1}$, and  $\Phi (t,x)  \geq \abs{\frac{x_1}{R} + \varphi(t) }  \geq M-1 $.

Using our hypothesis \eqref{eq B1}
we see that
\begin{align}\label{eq Lower}
\norm{e^{\alpha \Phi^2}f}_{L_{t,x}^2} \ge e^{(M-1)^2\alpha}\norm{u}_{L_{t,x}^2( [\frac{1}{2}-\frac{1}{8},\frac{1}{2}+\frac{1}{8}] \times B_1)}\ge e^{(M-1)^2\alpha}.
\end{align}

By the chain rule, we write
\begin{align}\label{eq 9}
\begin{aligned}
&  \quad  (i\partial_t+ \D^2 -V)f(t,x)\\ 
& =  \sum_{k=0}^3 \parenthese{ \eta \sum_{j=2}^d  C_k \partial_{x_j}^{\, 3-k} \theta_R (x)  +  \sum_{l=1}^{3-k} D_k \partial_{x_1}^{\, l} \eta \partial_{x_1}^{\, 3-k-l} \theta_R (x) } \partial_{x_j}^{\, k} u + i \theta_R \varphi'  \theta_R\partial_{x_1} \eta (\frac{x_1}{R} + \varphi(t)) u.
\end{aligned}
\end{align}

There are two types of terms in the decomposition of \eqref{eq 9}. 
\begin{itemize}
\item 
$\eta \times \partial \theta$ type: the support of such type is $[0,1] \times B_R \setminus B_{R-1} $ and $1 \leq \abs{\frac{x_1}{R} + \varphi(t) } \leq M+1$, hence we know $\Phi^2 \in [1, (M+1)^2 + (d-1)]$.

\item
$\partial \eta \times \theta_R $ or $\partial \eta \times \partial \theta_R $ type: the support of this type is $ [0,1] \times B_R   $ and $1 \leq \abs{\frac{x_1}{R} + \varphi (t)} \leq 2$, hence $\Phi^2 \in [1, 4 + (d-1) ]$.
\end{itemize}

Combining Lemma \ref{lem Carleman}  with the computation of $(i\partial_t+ \D^2 )f(t,x)$ in \eqref{eq 9}, we see that
\begin{align}
c^{\frac{1}{2}}\frac{\alpha^{\frac{7}{2}}}{ R^4} \norm{e^{\alpha \Phi^2}f}_{L_{t,x}^2} &\le \norm{e^{\alpha \Phi^2} (i\partial_t+\D^2)f}_{L_{t,x}^2} \leq \norm{e^{\alpha \Phi^2} (i\partial_t+\D^2 -V)f}_{L_{t,x}^2} + \norm{e^{\alpha \Phi^2} V f}_{L_{t,x}^2}   \\ 
& \leq \norm{e^{\alpha \Phi^2} \text{RHS of \eqref{eq 9}}}_{L_{t,x}^2} + L \norm{e^{\alpha \Phi^2} f}_{L_{t,x}^2} \\
& \leq e^{((M+1)^2+(d-1)) \alpha} \gamma(R) + e^{(3+d) \alpha} A + e^{(3+d) \alpha} \gamma(R) + L \norm{e^{\alpha \Phi^2} f}_{L_{t,x}^2}. \label{eq 18}
\end{align}

Choosing  $\alpha=cR^{\frac{4}{3}}$, we can hide the third term on the right-hand side of \eqref{eq 18} when $R \geq R_0(d, L)$. Then utilizing our  lower bound for $\norm{e^{\alpha \Phi^2}f}_{L_{t,x}^2}$ in \eqref{eq Lower}, we deduce that
\begin{align}
c R^{\frac{2}{3}} e^{(M-1)^2 \alpha} \leq e^{((M+1)^2+d-1) \alpha} \gamma (R) + e^{(3+d) \alpha} A  \\
c R^{\frac{2}{3}} e^{((M-1)^2- (3+d)) \alpha}  \leq e^{((M+1)^2+d-1 - (3+d))  \alpha } \gamma (R) +  A  .
\end{align}
Then by requiring $(M-1)^2 \geq (3+d)$ (note here the equality also works), we can hide $A$ into left-hand side of this inequality, hence for all $R\ge R_0(d, A,L)$, 
\begin{align}
c R^{\frac{2}{3}} e^{((M-1)^2- (3+d)) \alpha}  \leq e^{((M+1)^2+d-1 - (3+d))  \alpha } \gamma (R) .
\end{align}
Simplifying the two exponentials, we get 
\begin{align}
c R^{\frac{2}{3}} \leq e^{((M+1)^2+d-1 -(M-1)^2 ) \alpha } \gamma (R)  
\end{align}
which implies
\begin{align}
\gamma(R) \geq c R^{\frac{2}{3}} e^{-c R^{\frac{4}{3}}} ,
\end{align}
for all $R\ge R_0(d, A,L)$.

Now we finish the proof of Lemma \ref{lem Lower bound}.
\end{proof}

\section{Proof of Main Theorems}\label{sec Proof}

In this section, we prove the main theorems by contradiction. 
\subsection{Proof of Theorem \ref{thm Main1}}
If $u\not\equiv 0$, we can assume that $u$ satisfies the hypotheses of Lemma \ref{lem Lower bound}, after a translation, dilation, and multiplication by a constant. Thus, there exist constants $R_0(B), c'(B)$ depending on $B:=\norm{u}_{L_t^1 H_x^3([0,1] \times \mathbb{R}^d)}$  and a universal constant $c$ such that 
\begin{align}
\gamma(R) \ge c'R^{2/3} e^{-cR^{4/3}} \qquad \text{for all } R\ge R_0.
\end{align}

Let $\theta_R(x)\in C^\infty$ be non-decreasing, radial and such that 
\begin{align}
\theta_R (x) = 
\begin{cases}
1, & \text{if } \abs{x} \leq R-1, \\
0, & \text{if } \abs{x} \geq R. 
\end{cases}
\end{align}

Now take $\varphi_R(x)=1-\theta_R(x)$, where $\theta_R(x)$ is as defined above. Now, \eqref{eq ChainRule2} gives
\begin{align}
(i\partial_t +\D^2) (u\varphi_R) &= V_R (u\varphi_R) -\sum_{j=1}^d 4 (\partial_{x_j} \varphi)\partial_{x_j}^{\, 3} u + 6(\partial_{x_j}^{\, 2} \varphi) \partial_{x_j}^{\, 2} u + 4 (\partial_{x_j}^{\, 3} \varphi) \partial_{x_j} u + (\partial_{x_j}^{\, 4}\varphi) u := V_R (u\varphi_R) + H.
\end{align}
where $V_R (t,x)= \varphi_{R-1} V(t,x)$.
We apply Lemma \ref{lem Super-logcon} to the previous equation to find that  for $\alpha >1$
\begin{align}
&\quad \sup_{t\in[0,1]} \int_{|x|\ge R} |u(t,x)|^2 e^{\lambda|x|^{\alpha}/(10 d)^{\alpha}} \,dx \\
&\leq \sup_{t\in[0,1]} \int_{|x|\ge c_\alpha} |u\varphi_R (t,x)|^2 e^{\lambda|x|^{\alpha}/(10 d)^{\alpha}} \,dx \\
&\le c\left( \norm{u_0}_{L^2(e^{\lambda|x|^\alpha}  \, dx)}^2 +\norm{u_1}_{L^2(e^{\lambda|x|^\alpha} \, dx)}^2 \right) + c \int_0^1\int_{R-1\le|x|\le R}\left(|u|^2+|\nabla u|^2+|\nabla^{2} u|^2+|\nabla^{3} u|^2\right) e^{\lambda|x|^{\alpha}} \, dx\,dt \\
&  \quad  +c \sum_{j=1}^d \sum_{l=0}^3 \int_0^1 \int_{\mathbb{R}} |\partial_{x_j}^l(u\varphi)|^2 \, dx\,dt\\
& \le c_0 +c_0 \exp(\lambda R^\alpha).
\end{align}
In preparation for another application of Lemma \ref{lem Super-logcon}, we calculate 
\begin{align}
    (i\partial_t+ \D^2)\partial_{x_k} (u\varphi_R) & = V_R (\partial_{x_k}(u \varphi_R)) + \partial_{x_k} V_R (u \varphi_R)  \\
    & \quad - \sum_{j=1}^d  4  (\partial_{x_j}  \varphi_R)\partial_{x_j}^{\, 3} \partial_{x_k} u + 4  (\partial_{x_j} \partial_{x_k} \varphi_R)\partial_{x_j}^{\, 3} u  + 6(\partial_{x_j}^{\, 2}  \varphi_R) \partial_{x_j}^{\, 2} \partial_{x_k} u + 6(\partial_{x_j}^{\, 2} \partial_{x_k} \varphi_R) \partial_{x_j}^{\, 2} u \\
    & \quad  + 4 (\partial_{x_j}^{\, 3} \varphi_R) \partial_{x_j} \partial_{x_k} u  + 4 (\partial_{x_j}^{\, 3} \partial_{x_k} \varphi_R) \partial_{x_j} u  + (\partial_{x_j}^{\, 4}\varphi_R) \partial_{x_k}  u  +  (\partial_{x_j}^{\, 4} \partial_{x_k} \varphi_R) u  \\
    & = V_R (\partial_{x_k} (u \varphi_R)) + \widetilde{H}_k  .
\end{align}
Applying Lemma \ref{lem Super-logcon} to the above equation, we find that
\begin{align}
    & \quad \sup_{t\in[0,1]} \int_{|x|\ge R} |\partial_{x_k} u(t,x)|^2 e^{\lambda|x|^\alpha/(10d)^\alpha}\,dx \\
    &\le \sup_{t\in[0,1]} \int_{|x|\ge c_\alpha} |\partial_{x_k} (u\varphi_R)|^2(t,x) e^{\lambda|x|^\alpha/(10d)^\alpha}\,dx \\
    &\le c_0 + c \int_0^1\int_{R-1\le|x|\le R}\left(|u|^2+|\nabla u |^2+|\nabla^{2} u|^2+|\nabla^{3} u|^2+|\nabla^{4} u|^2 \right) e^{\lambda|x|^{\alpha}} \, dx\,dt \\
    & \quad + \int_0^1 \int_{\R^d} \abs{u \varphi_R \partial_{x_k}V_R}^2 e^{\lambda|x|^\alpha/(10d)^\alpha}\,dx  \, dt \\
    & \quad + \sum_{j=1}^d \sum_{l=0}^3 \int_0^1 \int_{\mathbb{R}} |\partial_{x_j}^l\partial_{x_k}(u\varphi)|^2\\
    &\le c + c \,\exp( \lambda R^\alpha).
\end{align}
We repeat this application of corollary up to the equations of  $(i\partial_t+\D^2)\partial_{x_k}^2 (u\,\varphi_R)$, $(i\partial_t+\D^2)\partial_{x_k}^3 (u\,\varphi_R)$  and combine all the conclusions to see that 
\begin{align}
    \sup_{t\in[0,1]} \int_{|x|\ge R}\left(|u|^2+|\nabla u|^2+|\nabla^{2} u|^2+|\nabla^3 u|^2  \right)e^{\lambda|x|^{\alpha}/(10 d)^\alpha} \, dx \le c_0 + c_0\,\exp(\lambda R^\alpha). \label{H3_bound}
\end{align}
Thus, for all $\mu$ such that $\mu R-1> R$,
\begin{align}
    c_0 \left(\mu R\right)^{2/3} e^{-c(\mu R)^{4/3}}e^{\lambda(\mu R-1)^\alpha /(10 d)^\alpha}&\le c_0 \gamma(\mu R)e^{\lambda(\mu R-1)^\alpha /(10 d)^\alpha} \\
    &\le c \int_0^1 \int_{\mu R-1\le |x|\le \mu R} \sum_{l=0}^3 |\partial_x^{\,l}u(t,x)|^2 e^{\lambda(\mu R-1)^{\alpha}/(10 d)^\alpha} \, dx\,dt \\
    &\le c \int_0^1 \int_{\mu R-1\le |x|\le \mu R} \sum_{l=0}^3 |\partial_x^{\,l}u(t,x)|^2 e^{\lambda|x|^{\alpha}/(10 d)^\alpha} \, dx\,dt \\
    &\le c \int_0^1 \int_{\mu R-1\le |x|} \sum_{l=0}^3 |\partial_x^{\,l}u(t,x)|^2 e^{\lambda|x|^{\alpha}/(10 d)^\alpha} \, dx\,dt \\
    &\le c \int_0^1 \int_{ R \le |x|} \sum_{l=0}^3 |\partial_x^{\,l}u(t,x)|^2 e^{\lambda|x|^{\alpha}/(10 d)^\alpha} \, dx\,dt \\
    &\le c+c\,\exp(\lambda R^\alpha),
\end{align}
where the last inequality is just \eqref{H3_bound}. Since $\alpha>4/3$, taking large enough $\mu$, we obtain a contradiction as the left hand side of the chain of inequalities is unbounded, while the right side is bounded. Thus, $u\equiv 0$ identically.

We have finished the proof of Theorem \ref{thm Main1}.

\subsection{Proof of Theorem \ref{thm Main2}}
    We consider the difference of the two solutions
    \begin{align}
        w(t,x)= u_1(t,x) - u_2(t,x),
    \end{align}
    which satisfy the equation
    \begin{align}
        i\partial_t w + \D^2 w = \left(\frac{F(u_1, \overline{u}_1) - F(u_2, \overline{u}_2)}{u_1-u_2}\right)w. 
    \end{align}
    We repeat the same arguments with the new potential
    \begin{align}
        V(t,x) = \frac{F(u_1, \overline{u}_1) - F(u_2, \overline{u}_2)}{u_1-u_2},
    \end{align}
    which satisfies the necessary hypotheses on $V(t,x)$ given the regularity on $F(u, \overline{u})$. By the same contradiction argument, we obtain Theorem \ref{thm Main2}.

\bibliography{references}
\bibliographystyle{plain}
\end{document}